%% file: klosin2011.tex
\documentclass{amsart}
\usepackage{amsmath, amsfonts, amssymb}
\input{makros}

\input{settings}

\begin{document}

\title[The Maass space and the Bloch-Kato conjecture]
{The Maass space for $\U(2,2)$ and the Bloch-Kato conjecture for the symmetric square motive of a modular form\footnotetext{2000 Mathematics subject classification 11F33 (primary), 11F55, 11F67, 11F80, 11F30 (secondary)}}
\author{Krzysztof Klosin}
\date{November 1, 2011}

\maketitle

\begin{abstract}
Let $K=\bfQ(i\sqrt{D_K})$ be an imaginary quadratic field of discriminant $-D_K$. We introduce a notion of an adelic Maass space $\mS_{k, -k/2}^{\rm M}$ for automorphic forms on the quasi-split unitary group $\U(2,2)$ associated with $K$ and prove that it is stable under the action of all Hecke operators. When $D_K$ is prime we obtain a Hecke-equivariant descent from $\mS_{k,-k/2}^{\rm M}$ to the space of elliptic cusp forms $S_{k-1}(D_K, \chi_K)$, where $\chi_K$ is the quadratic character of $K$. For a given $\phi \in   S_{k-1}(D_K, \chi_K)$, a prime $\ell >k$, we then construct (mod $\ell$) congruences between the Maass form corresponding to $\phi$ and hermitian modular forms orthogonal to $\mS_{k,-k/2}^{\rm M}$ whenever $\val_{\ell}(L^{\rm alg}(\Symm \phi, k))>0$. This gives a proof of the holomorphic analogue of the unitary version of Harder's conjecture. Finally, we use these congruences to provide evidence for the Bloch-Kato conjecture for the motives $\Symm\rho_{\phi}(k-3)$ and $\Symm\rho_{\phi}(k)$, where $\rho_{\phi}$ denotes the Galois representation attached to $\phi$. 

\end{abstract}

\section{Introduction} \label{Introduction}

In 1990 Bloch and Kato \cite{BlochKato90} formulated a conjecture whose version relates the order of a Selmer group of a motive $M$ to a special value of an $L$-function of $M$. This is a very far-reaching conjecture which is currently known only in a handful of cases, mostly concerning the situations when $M$ arises from a one-dimensional Galois representation. However, in 2004 Diamond, Flach and Guo proved a very strong result in a three-dimensional case \cite{DiamondFlachGuo04}. Indeed, they proved the Bloch-Kato conjecture for the adjoint motive $\ad^0\rho_\phi$ (and its Tate twist $\ad^0 \rho_\phi(1)$) of the $\ell$-adic Galois representation $\rho_{\phi}$ attached to a classical modular form $\phi$ without any restrictions on the weight ($\geq 2$) or the level of $\phi$. Their proof was highly influenced by the ideas that were first applied by Taylor and Wiles in their proof of Fermat's Last Theorem (\cite{TaylorWiles95}, \cite{Wiles95}). 

In 2009 the author proved a (weaker) result providing evidence for the conjecture for a different Tate twist of $\ad^0\rho_\phi$ (more precisely for $\ad^0\rho_\phi(-1)\chi=\Symm \rho_{\phi}(k-3)$ and $\ad^0\rho_\phi(2)\chi = \Symm \rho_{\phi}(k)$) for modular forms $\phi$ of any weight $k-1$ (with $k$ divisible by 4), level 4 and non-trivial character $\chi$ \cite{Klosin09}. The method was different from that of \cite{DiamondFlachGuo04} (but similar to the one used by Brown \cite{Brown07} who worked with Saito-Kurokawa lifts). It relied on constructing congruences between a certain lift of $\phi$ (called the \emph{Maass lift}) to the unitary group $\U(2,2)$ defined with respect to the field $\bfQ(i)$ and hermitian modular forms (i.e., forms on $\U(2,2)$) which were orthogonal to the \emph{Maass space} (the span of such lifts). The elements in the relevant Selmer groups were then constructed using ideas of Urban \cite{Urban01}. Unfortunately, some of the methods implemented in \cite{Klosin09} relied significantly on the fact that the class number of $\bfQ(i)$ is one and could not be directly generalized to deal with other imaginary quadratic fields.

In this paper we develop new tools - among them a new notion of an adelic Maass space and a Rankin-Selberg type formula - which work in sufficient generality. As a consequence we are in particular able to extend the results of \cite{Klosin09} to all imaginary quadratic fields of prime discriminant $-D_K$, i.e., to all modular forms $\phi$ of any prime level $D_K$, of arbitrary weight $k-1$ (with $k$ divisible by the number of roots of unity contained in $K$), and nebentypus $\chi_K$ being the quadratic character associated with the extension $K/\bfQ$. Our result on the one hand provides evidence for the Bloch-Kato conjecture for the motives $\Symm \rho_{\phi}(k-3)$ and $\Symm\rho_{\phi}(k)$ for a rather broad family of modular forms $\phi$. On the other hand the congruence itself (between a Maass lift and a hermitian modular form orthogonal to the Maass space) provides a proof 
of a holomorphic analogue to a conjecture recently formulated by Dummigan extending the so-called Harder's conjecture concerning  Siegel modular forms \cite{Dummigan11Harderconjecture}.

As alluded above, the first difficulty that one encounters  in dealing with a general imaginary quadratic field is the lack of a proper notion of the Maass space in this case. The definition introduced by Krieg \cite{Krieg91} does not allow one to define the action of the Hecke operators at non-principal primes. The more recent (very elegant) results due to Ikeda \cite{Ikeda08} while dealing with class number issues, are not quite sufficient for our purposes. So, in this paper we introduce a new adelic version of the Maass space and carefully study its properties, especially its invariance under the action of the Hecke algebras. This provides us with a correct analogue of the classical Maass lift to the space of automorphic forms on $\U(2,2)(\AQ)$ for any imaginary quadratic field of prime discriminant. We then proceed to construct a congruence between the Maass lift and hermitian modular forms orthogonal to the Maass space.

The method of exhibiting elements in Selmer groups of automorphic forms via constructing congruences between automorphic forms on a higher-rank group has been used by several authors. The original idea can be dated back to the influential paper of Ribet \cite{Ribet76} on the converse to Herbrand's Theorem, where for a certain family of Dirichlet characters $\chi$ elements in the $\chi$-eigenspace of the class group of a cyclotomic field are constructed by first exhibiting a congruence between a certain Eisenstein series (associated with $\chi$) and a cusp form on $\GL_2$. Higher-rank analogues of this method have been applied to provide evidence for (one inequality in) the Bloch-Kato conjecture for several motives by Bella\"iche and Chenevier, Brown, Berger, B\"ocherer, Dummigan, Schulze-Pillot, as well as Agarwal and the author (\cite{BellaicheChenevier04},  \cite{Brown07}, \cite{B09}, \cite{Klosin09}, \cite{BochererDummiganSchulzePillotpreprint}, \cite{AgarwalKlosin10preprint}). An extension of these ideas was also used to prove results towards the Main Conjecture of Iwasawa Theory by Mazur and Wiles, Urban (\cite{MazurWiles84}, \cite{Urban01}) and very recently by Skinner and Urban  \cite{SkinnerUrbanpreprint}.

The general idea is the following. Given an automorphic form $\phi$ on an algebraic group $M$ (with an associated Galois representation $\rho_{\phi}$) one lifts $\phi$ to an automorphic form on $G$ in which $M$ can be realized as a subgroup (in our case $M=\Res_{K/\bfQ}(\GL_{2/K})$ is the Levi subgroup of a Siegel parabolic of $G=\U(2,2)$). The Galois representation attached to the lift is reducible and has irreducible components related to $\rho_{\phi}$. Assuming divisibility (by a uniformizer $\varpi$ in some extension of $\bfQ_{\ell}$) of a certain $L$-value associated with $\phi$ one shows (this is usually the technically difficult part) that the lift is congruent (mod $\varpi$)  to an automorphic form $\pi$ on $G$ whose Galois representation $\rho_{\pi}$ is irreducible. Because of the congruence, the mod $\varpi$ reduction of $\rho_{\pi}$ must be reducible, but (because $\rho_{\pi}$ was irreducible) it can be chosen to represent a non-split extension of its irreducible components, thus giving rise to a non-zero element in some Selmer group (related to $\rho_{\phi}$). 

Let $\phi \in S_{k-1}(D_K, \chi_K)$ be a newform. In our case the lifting procedure is the Maass lift, which produces an automorphic form $f_{\phi, \chi}$ on $\U(2,2)(\AQ)$ (which depends on a certain character $\chi$ of the class group of $K$) whose associated automorphic representation is CAP in the sense of Piatetski-Shapiro \cite{Piatetski-Shapiro83}. Even though the desired congruence is between Hecke eigenvalues of the lift  $f_{\phi, \chi}$ and those of $\pi$, we first construct a congruence between Fourier coefficients of these forms and only then deduce the Hecke eigenvalue congruence. The former congruence is achieved by first defining a certain hermitian modular form $\Xi$ (essentially a product a Siegel Eisenstein series and a hermitian theta series) and writing it as:

\[\Xi = \frac{\left< \Xi, f_{\phi, \chi}\right>}{\left< f_{\phi, \chi}, f_{\phi, \chi}\right>} f_{\phi, \chi} + g',\] where  $g'$ is a hermitian modular form orthogonal to $f_{\phi, \chi}$. The form $\Xi$ has nice arithmetic properties (in particular its Fourier coefficients are $\varpi$-adically integral) and we show that the inner products can be expressed by certain $L$-values. In particular the inner product in the denominator is related to $L^{\rm alg}(\Symm \phi, k)$. Choosing $\Xi$ so that the special $L$-values contributing to the inner product in the numerator make it  a $\varpi$-adic unit and assuming the $\varpi$-adic valuation of $L^{\rm alg}(\Symm \phi, k)$ is positive we get a congruence between $f_{\phi, \chi}$ and a scalar multiple $g$ of $g'$. To ensure that $g$ itself is not a Maass lift we construct a certain Hecke operator $T^{\rm h}$ that kills the ``Maass part'' of $g$. 

Let us now briefly elaborate on the technical difficulties that one encounters in the current paper as opposed to the case of $K=\bfQ(i)$ which was studied in \cite{Klosin09}. First of all, as mentioned above, the Maass space and the Maass lift in the case of the field $\bfQ(i)$ are well-understood thanks to the work of Kojima, Gritsenko and Krieg (\cite{Kojima}, \cite{Gritsenko86}, \cite{Gritsenko90}, \cite{Krieg91}) and in \cite{Klosin09} we simply invoke the relevant definitions and properties of these objects. In the current paper we introduce a notion of an adelic Maass space for a general imaginary quadratic field and prove that it is a Hecke-stable subspace of the space of hermitian modular forms. This result does not use the assumption that $D_K$ is prime. One of the difficulties in extending the classical notion of Kojima and Krieg is the fact that when the class number of $K$ is greater than one the classical Maass space (which was defined by Krieg for all imaginary quadratic fields) is not stable under the action of the local Hecke algebras at non-principal primes of $K$. This is one of the reasons why we chose to formulate the theory in an adelic language, even though it would in principle be possible to extend the classical definitions of Krieg and work with several copies of the hermitian upper half-space. However, we think that the action of the Hecke operators as well as the role played by a central character are most transparent in the adelic setting. When $D_K$ is prime we are able to relate our lift to the results of Krieg and Gritsenko and derive explicit formulas for the descent of the Hecke operators. We also prove an $L$-function identity relating the standard $L$-function of a Maass lift to the $L$-function of the base change to $K$ of the modular form $\phi$. All this is the content of section \ref{Maass space}. A yet another notion of the Maass lift has in the meantime been introduced by Ikeda \cite{Ikeda08} using a different approach. This notion agrees with ours in the case of a trivial central character, but not all the formulas necessary for our arithmetic applications are present in \cite{Ikeda08}.

On the other hand a lot of attention in \cite{Klosin09} was devoted to computing the Petersson norm of a Maass lift (\cite{Klosin09}, section 4). Here, however, we use a formula due to Sugano (cf. \cite{Ikeda08}) to tackle the problem. To derive the congruence one also needs to be able to express the  inner product $\left< \Xi, f_{\phi, \chi}\right>$ by certain $L$-values related to $\phi$. This calculation (drawing heavily on the work of Shimura) becomes somewhat involved in the case of class number larger than one.
 The relevant computations are carried out in section \ref{Inner product formula}. Since it does not add much to the computational complexity we prove all the results for the group $\U(n,n)$ for a general $n>1$, obtaining this way a general Rankin-Selberg type formula that might be of independent interest (Theorem \ref{ransel}). In that section we also prove the integrality of the Fourier coefficients of a certain hermitian theta series involved in the definition of $\Xi$. Finally in the current paper our construction of the Hecke operator $T^{\rm h}$ is somewhat cleaner, because we work more with completed Hecke algebras. This is carried out in section \ref{Completed Hecke algebras}.

In section \ref{Congruence} we collect all the results to arrive at the desired congruence, first between the Fourier coefficients of $f_{\phi, \chi}$ and a hermitian modular form orthogonal to the Maass space (Theorem \ref{mainthm}) and then between the Hecke eigenvalues of $f_{\phi, \chi}$ and those of some hermitian Hecke eigenform $f$ also orthogonal to the Maass space (Corollary \ref{cormain78}). The latter can only be achieved modulo the first power of $\varpi$  even if $L^{\rm alg}(\Symm \phi, k)$ is divisible by a higher power of $\varpi$. This is not a shortcoming of our method but a consequence of the fact that there may be more than one $f$ congruent to $f_{\phi, \chi}$ and the $L$-value (conjecturally) controls contributions from all such $f$. In fact this is precisely what we prove by studying congruences between $f_{\phi, \chi}$ and \emph{all} the possible eigenforms $f$ orthogonal to the Maass space and as a result give a lower bound (in terms of $L^{\rm alg}(\Symm \phi, k)$) on the index of an analogue of the classical Eisenstein ideal (which we in our case call the \emph{Maass ideal}) in the appropriate Hecke algebra (section \ref{The Maass ideal}). Finally, we demonstrate how our results imply the holomorphic analogue of Dummigan's version of Harder's conjecture for the group $\U(2,2)$ (section \ref{Unitary analogue of Harder's conjecture}).

The congruence can be used to deduce the existence of certain non-zero elements in the Selmer groups $H^1_f(K, \ad^0 \rho_f(-1))$ and $H^1_f(K, \ad^0 \rho_f(2))$ and hence get a result towards the Bloch-Kato conjecture for these motives. For this we use a theorem of Urban \cite{Urban01}. The relevant results are stated in section \ref{The Bloch-Kato conjecture}. We note that the above Selmer groups are over $K$, while the conjecture relates  $L^{\rm alg}(\Symm \phi, k)$ to the order of the corresponding Selmer group over $\bfQ$. More precisely, under some mild assumptions one has $\val_{\ell}(L^{\rm alg}(\Symm \phi, k)) = \val_{\ell}(L^{\rm alg}(\Symm \phi, k-3)) = \val_{\ell}(L^{\rm alg}(\ad \phi, -1, \chi_K))$, so $L^{\rm alg}(\Symm \phi, k)$ should control the order of $H^1_f(\bfQ, \Symm \rho_{\phi}(k-3))= H^1_f(\bfQ, \ad^0 \rho_{\phi}(-1)\chi_K)$ (we are grateful to Neil Dummigan for pointing out an error in an earlier version of the article). 

The author would like to thank Chris Skinner who initially suggested this problem to him. Over the course of work on this paper the author benefited greatly from conversations and email correspondence with many people and would like to express his particular gratitude to Neil Dummigan, G\"unther Harder, Tamotsu Ikeda and Jacques Tilouine. He is also grateful to the Department of Mathematics of the Universit\'e Paris 13, where part of this work was carried out, for a friendly and stimulating atmosphere during the author's stay there in 2009 and 2010 and similarly to the Max-Planck-Institut in Bonn, where the author spent the Summer of 2010. Finally, we would like to thank Neil Dummigan for sending us his preprint on the analogue of Harder's conjecture for the group $\U(2,2)$.

\section{Notation and terminology} \label{Notation and terminology}

In this section we introduce some basic concepts and establish notation
which will be used throughout this paper unless explicitly indicated
otherwise.
\subsection{Number fields and Hecke characters} \label{Number fields and
Hecke characters}

Throughout this paper $\ell$ will always denote a fixed odd prime. We write
$i$ for $\sqrt{-1}$. Let $K=\bfQ(i\sqrt{D_K})$ be a fixed imaginary quadratic extension
of $\bfQ$ of discriminant $-D_K$, and let $\OK$ be the ring of integers
of $K$. We will write $\Cl_K$ for the class group of $K$ and $h_K$ for 
$\# \Cl_K$. For $\a \in K$, denote by $\ov{\a}$
the image of $\a$ under the
non-trivial automorphism of $K$. Set $N\a := N(\a) := \a \ov{\a}$, and 
for
an ideal $\mathfrak{n}$ of $\OK$, set $N\mathfrak{n}:=
\#(\OK/\mathfrak{n})$. As remarked below we will always view $K$ as a
subfield of $\bfC$. For $\a \in \bfC$, $\ov{\a}$ will denote the complex
conjugate of $\a$ and we set $|\a|:= \sqrt{\a \ov{\a}}$.

Let $L$ be a number field with ring of integers $\Oo_L$. For a place $v$
of $L$, denote by $L_v$ the completion of $L$ at $v$ and by
$\mathcal{O}_{L,v}$ the valuation ring of $L_v$. If $p$ is a place of
$\bfQ$, we set $L_p:= \bfQ_p \otimes_{\bfQ} L$ and $\mathcal{O}_{L,p}:=
\bfZ_p \otimes_{\bfZ} \Oo_L$. We also allow $p=\infty$.
Set $\hat{\bfZ}=\invlim_n \bfZ/n\bfZ = \prod_{p \nmid \infty} \bfZ_p$ and similarly $\hat{\Oo}_K = \prod_{v \nmid \infty} \Oo_{K,v}$. 
For a finite $p$, 
let
$\val_p$ denote the $p$-adic valuation on $\bfQ_p$. For
notational convenience we also define $\val_p(\iy):= \iy$. If $\a \in
\bfQ_p$, then $|\a|_{\bfQ_p}:= p^{-\val_p(\a)}$ denotes the
$p$-adic norm of $\a$. For $p=\iy$, $|\cdot|_{\bfQ_{\iy}} = |\cdot
|_{\bfR} = |\cdot |$ is the usual absolute value on $\bfQ_{\iy}=\bfR$.

In this paper we fix once and for all an algebraic closure $\ov{\bfQ}$ 
of
the rationals and algebraic closures $\ov{\bfQ}_p$ of $\bfQ_p$, as
well as compatible embeddings $\ov{\bfQ} \hookrightarrow \ov{\bfQ}_p
\hookrightarrow
\bfC$
for all finite places $p$ of $\bfQ$. We extend $\val_p$ to a function 
from
$\ov{\bfQ}_p$ into $\bfQ$. Let $L$ be a number field. We write $G_L$ for
$\Gal(\ov{\bfQ}/L)$. If $\fp$ is a prime of $L$, we also write
$D_{\fp}\subset G_L$ for
the decomposition group of $\fp$ and $I_{\fp}\subset D_{\fp}$ for the
inertia group of
$\fp$. The chosen embeddings allow us to identify $D_{\fp}$ with
$\Gal(\ov{L}_{\fp} /L_{\fp})$. We will always write $\Frob_{\fp}\in D_{\fp}/I_{\fp}$ to denote the \emph{arithmetic} Frobenius.

For a local field $E$ (which for us will always be a finite extension of $\bfQ_p$ for some prime $p$) and a choice of a uniformizer $\varpi \in E$, we will write $\val_{\varpi}: E \to \bfZ$ for the $\varpi$-adic valuation on $E$. 

For a number field $L$ let $\mathbf{A}_L$ denote the ring of adeles of 
$L$ and put $\adele
:=
\mathbf{A}_{\bfQ}$. Write $\mathbf{A}_{L, \iy}$ and $\mathbf{A}_{L,
\textup{f}}$ for the
infinite part
and the finite part of $\mathbf{A}_L$ respectively. For $\a = (\a_p) \in
\adele$ set
$|\a|_{\adele}:= \prod_p |\a|_{\bfQ_p}$. By a \textit{Hecke character} 
of
$\mathbf{A}_L^{\times}$ (or of $L$, for short) we
mean a continuous homomorphism $$\psi: L^{\times} \setminus
\mathbf{A}_L^{\times} \rightarrow
\bfC^{\times}.$$ 
 The trivial
Hecke character will be denoted by $\mathbf{1}$. The character
$\psi$ factors into a product of local characters $\psi=\prod_v \psi_v$,
where $v$ runs over
all places of $L$. If $\mathfrak{n}$ is the ideal of the ring of 
integers
$\Oo_L$ of $L$ such
that
\begin{itemize}
\item $\psi_v(x_v)=1$ if $v$ is a finite place of $L$, $x_v \in
\mathcal{O}_{L,v}^{\times}$
and
$x-1 \in
\mathfrak{n} \mathcal{O}_{L,v}$
\item no ideal $\mathfrak{m}$ strictly containing $\mathfrak{n}$ has the
above property,

\end{itemize}

\no then $\mathfrak{n}$ will be called the \textit{conductor of $\psi$}.
If
$\fm$ is an ideal of
$\Oo_L$, then we set $\psi_{\fm}:= \prod \psi_v$, where the product runs
over all the finite
places $v$ of $L$ such that $v \mid \fm$. For
a
Hecke character
$\psi$ of $\mathbf{A}_L^{\times}$, denote by $\psi^*$ the associated
ideal character. Let
$\psi$ be a Hecke character of $\AK^{\times}$. We will
sometimes think 
of $\psi$ as a character of $(\R \GL_{1/K})(\adele)$, where $\R$ stands for the Weil restriction of scalars. We have a 
factorization
$\psi = \prod_p
\psi_p$ into local
characters $\psi_p: \left(\R \GL_{1/K}\right) (\bfQ_p) \rightarrow
\bfC^{\times}$. For $M \in
\bfZ$, we set $\psi_M:= \prod_{p\neq \iy, \hs p \mid M} \psi_p$. If 
$\psi$
is a Hecke
character of $\AK^{\times}$, we set $\psi_{\bfQ} =
\psi|_{\adele^{\times}}$.

\subsection{The unitary group} \label{The unitary group}

For
any affine group scheme $X$ over $\bfZ$ and any $\bfZ$-algebra $A$, we
denote by $x \mapsto \ov{x}$
the automorphism of $(\Res_{\OK/\bfZ}X_{\OK})(A)$ induced by the non-trivial
automorphism of
$K/\bfQ$. Note that $(\Res_{\OK/\bfZ}X_{\OK})(A)$ can be identified with a
subgroup of $\GL_n(A \otimes \OK)$ for some $n$. In what follows we
always specify such an identification.
Then for $x \in
(\Res_{\OK/\bfZ}X_{\OK})(A)$ we write $x^t$ for the transpose of $x$, and set
$x^*:= \ov{x}^t$ and $\hat{x}:= (\ov{x}^t)^{-1}$. Moreover, we write
$\diag(a_1, a_2, \dots, a_n)$ for the $n\times n$-matrix with $a_1, a_2,
\dots a_n$ on the diagonal and all the off-diagonal entries equal to 
zero.

We will denote by
$\bfG_a$ the additive group and by $\bfG_m$ the multiplicative group. To the imaginary quadratic extension $K/\bfQ$ one
associates the
unitary
similitude group scheme over $\bfZ$: $$G_n:=\GU(n,n) = \{A \in 
\Res_{\OK/\bfZ}
\GL_{2n}
\sep AJ\bar{A}^t =
\mu(A) J \},
$$ where
$J=\bmat &-I_n \\ I_n& \emat$, with $I_n$ denoting the $n \times n$
identity
matrix and $\mu(A)\in \bfG_m$.
We will also make use of the groups $$U_n=\U(n,n) = \{A \in \GU(n,n) 
\sep
\mu(A)=1
\},$$ and $$\SU(n,n) = \{A\in U_n \sep \det A=1\}.$$ For $x \in
\textup{Res}_{\OK/\bfZ}(\GL_n)$, we write $p_x$ for $\bmat x \\ &\hat{x}
\emat \in U_n$.
Since the case $n=2$ will be of particular interest to us we set
$G=G_2$, $U=U_2$.

Note
that if
$p$ is
inert or ramified in $K$, then $K_p/\bfQ_p$ is a degree two extension of
local fields and $a \mapsto \ov{a}$ induces the non-trivial automorphism
in $\Gal(K_p/\bfQ_p)$. If $p$ splits in $K$, denote by $\iota_{p,1},
\iota_{p,2}$ the two distinct embeddings of $K$ into $\bfQ_p$.
Then the map $a\otimes b \mapsto (\iota_{p,1}(a)b, \iota_{p,2}(a)b)$,
induces a $\bfQ_p$-algebra isomorphism $K_p
\cong \bfQ_p \times \bfQ_p$, and $a \mapsto \ov{a}$ corresponds
on the right-hand side to the automorphism defined by $(a,b) \mapsto
(b,a)$. We
denote the isomorphism $\bfQ_p \times
\bfQ_p \xrightarrow{\sim} K_p$ by $\iota_p$. For a matrix $g=(g_{ij})$
with entries   
in $\bfQ_p \times \bfQ_p$ we also
set $\iota_p(g)= (\iota_p(g_{ij}))$. For a split prime $p$ the map 
$\iota_p^{-1}$ identifies $U_n(\bfQ_p)$ with $$U_{n,p} =
\{(g_1,g_2) \in \GL_{2n}(\bfQ_p)
\times \GL_{2n}(\bfQ_p) \mid g_1 J g_2^t = J \}.$$ Note that the map   
$(g_1,g_2) \mapsto g_1$ gives a (non-canonical) isomorphism $U_n(\bfQ_p)
\cong \GL_{2n}(\bfQ_p)$. Similarly, one has $G_n(\bfQ_p) \cong \GL_{2n}(\bfQ_p) \times \bfG_m(\bfQ_p)$.

 In $U=U_2$ 
we choose a maximal torus
$$T=\left\{\bmat a&&&\\ &b&&\\ &&\hat{a}&\\ &&&\hat{b} \emat \sep a,b 
\in
\R
\bfG_{m/K} \right \},$$

\noindent and a Borel subgroup $B=TU_B$ with unipotent
radical
$$U_B=\left\{\bmat 1&\a&\b&\g\\ &1&\bar{\g}-\bar{\a}\f&\f\\&&1&\\
&&-\bar{\a}&1 \emat \sep \a, \b, \g \in \R \bfG_{a/K}, \hf \f \in \Ga, \hf
\b+\g\bar{\a}
\in \Ga \right\}.$$
Let $$T_{\bfQ} = \left\{\bmat a&&&\\ &b&&\\ &&a^{-1}&\\ &&&b^{-1} \emat
\sep a,b \in \bfG_m
\right \}$$ denote the maximal $\bfQ$-split torus contained in $T$. Let
$R(U)$ be the set of
roots of
$T_{\bfQ}$, and denote by $e_j$, $j=1,2$, the root defined by
$$e_j:  \bmat a_1&&&\\ &a_2&&\\ &&a_1^{-1}&\\ &&&a_2^{-1} \emat \mapsto
a_j.$$
The choice of $B$ determines a subset $R^+(U)\subset R(U)$ of positive  
roots. We have $$R^+(U)= \{ e_1+e_2, e_1-e_2,
2e_1,
2e_2\}.$$ We fix a set $\Delta(U) \subset R^+(U)$ of simple roots
$$\Delta(U):= \{ e_1-e_2, 2e_2 \}.$$
If $\theta \subset \Delta(U)$, denote the  
parabolic subgroup corresponding to $\theta$ by $P_{\theta}$.  
We have $P_{\Delta(U)} = U$ and $P_{\emptyset} = B$. The other
two possible subsets of $\Delta(U)$ correspond to maximal
$\bfQ$-parabolics of
$U$:
\begin{itemize}
\item the Siegel parabolic $P:=P_{\{e_1-e_2\}}=M_PU_P$ with Levi 
subgroup
$$M_P = \left\{ \bmat A&\\ & \hat{A} \emat \sep A \in \R \GL_{2/K} 
\right\},$$
and (abelian) unipotent radical
$$U_P = \left \{\bmat 1&&b_1&b_2\\ &1&\ov{b}_2& b_4\\ &&1& \\ &&&1 \emat
\sep b_1, b_4 \in \Ga, \hf b_2 \in \R \bfG_{a/K} \right \}$$

\item the Klingen parabolic $Q:=P_{\{2e_2\}}=M_QU_Q$ with Levi subgroup
$$M_Q = \left \{ \bmat x&&&\\&a&&b\\ &&\hat{x}& \\ &c&&d \emat \sep x 
\in
\R \bfG_{m/K}, \hf \bmat a&b\\ c&d \emat \in \U(1,1) \right \},$$
and (non-abelian) unipotent radical
$$U_Q =\left\{ \bmat 1&\a&\b&\g\\ &1& \bar{\g} &\\ &&1&\\ && -\bar{\a}&1
\emat \sep \a, \b, \g \in \R \bfG_{a/K}, \hf \b+\g\bar{\a} \in \Ga \right\}$$

\end{itemize}

Similarly in $U_n$ we denote by $T_n$ the diagonal torus and by $P_n$ 
the
Siegel parabolic (with Levi isomorphic to $\R \GL_{n/K}$).

\noindent For an
associative ring $R$ with identity and an $R$-module $N$
we write   $M_n(N)$ for the $R$-module of $n \times n$-matrices with entries in $N$. Let $x=\bsmat
A&B\\C&D \esmat \in
M_{2n}(N)$ with $A, B, C, D
\in
M_n(N)$. Define $a_x=A$,
$b_x=B$, $c_x=C$, $d_x=D$.

For $M\in \bfQ$, $N \in \bfZ$ such that $MN \in \bfZ$ we will denote by
$D_n(M, N)$ the group
$U_n(\bfR) \hs
\prod_{p \nmid \iy} \mK_{0,n,p}(M, N) \subset U_n(\adele)$, where
\begin{multline} \mK_{0,n,p}(M,N) = \left \{x \in
U_n(\quf_p) \mid a_x, d_x \in M_n(\Oo_{K,p}) \right. ,\\
\left. b_x\in M_n(M^{-1}\Oo_{K,p}), \hf c_x \in
M_n(MN\Oo_{K,p})\right \}.\end{multline} If $M=1$, denote $D_n(M,N)$
simply by $D_n(N)$ and
$\mK_{0,n,p}(M,N)$
by $\mK_{0,n,p}(N)$. For any finite $p$, the group $\mK_{0,n,p}:= 
\mK_{0,n,p}(1)
=
U_n(\bfZ_p)$ is a
maximal (open) compact subgroup of $U_n(\bfQ_p)$. Note that if $p \nmid 
N$,
then $\mK_{0,n,p}=
\mK_{0,n,p}(N)$. We write $\mK_{0, n,\textup{f}}(N):= \prod_{p \nmid \iy}
\mK_{0,n,p}(N)$ and $\mK_{0,n,
\textup{f}}: =\mK_{0, \textup{f}}(1)$. Note that $\mK_{0, n,\textup{f}}$ 
is
a
maximal (open)
compact subgroup of
$U_n(\adele_{\textup{f}})$. Set $$\mK^+_{0,n, \iy}:= \left\{ \bmat A &B \\ -B 
&
A
\emat \in U_n(\bfR)
\mid A,B \in \GL_n(\bfC), AA^* + BB^* = I_n, A B^* = B A^*\right\}.$$
Then $\mK_{0, n,\iy}^+$ is a maximal compact subgroup of $U_n(\bfR)$. We will denote by $\mK_{0, n,\iy}$ the subgroup of $G_n(\bfR)$ generated by $\mK_{0, n,\iy}^+$ and $J$. Then $\mK_{0, n,\iy}$ is a maximal compact subgroup of $G_n(\bfR)$. Let
$$U(m):= \left\{ A \in
\GL_m(\bfC) \mid A A^* = I_m \right\}.$$ We have $$\mK^+_{0,n, \iy} =
U_n(\bfR) \cap U(2n)
\xrightarrow{\sim} U(n) \times U(n),$$ where the last isomorphism is 
given
by $$\bmat A &B \\
-B & A \emat \mapsto (A+iB, A-iB) \in U(n) \times U(n).$$ Finally, set
$\mK_{0,n}(N):= \mK_{0, n,\iy}^+
\mK_{0, n,\textup{f}}(N)$ and $\mK_{0,n}:= \mK_{0,n}(1)$. The last group is a
maximal
compact subgroup of
$U_n(\adele)$.

Similarly, we define $\mK_{1,n}(N)=\mK_{0,n,\infty}^+ \mK_{1, n,\tuf}(N),$ where
$\mK_{1,n,\tuf}(N)=\prod_{p \nmid \infty} \mK_{1,n,p}(N)$,
$$\mK_{1,n,p}(N)= \{x \in \mK_{0,n,p}(N) \mid a_x-I_n\in M_n(N\Oo_{K,p})\}.$$

Let $M \in \bfQ$, $N\in \bfZ$ be such that $MN \in \bfZ$. We define
the following congruence subgroups of $U_n(\bfQ)$:
\be \begin{split} \Gamma^{\hh}_{0,n}(M,N) & := U_n(\bfQ) \cap D_n(M,N),\\
\G^{\hh}_{1,n}(M,N) &:= \{\a \in \Gamma^{\hh}_{0,n}(M,N) \mid a_{\a}-1 \in
M_n(N\OK)\},\\
\G^{\hh}_n(M,N)   &:= \{\a \in \G^{\hh}_{1,n}(M,N) \mid b_{\a} \in
M_n(M^{-1}N\OK) \}\end{split},
\ee
and set $\G_{0,n}^{\hh}(N) := \Gamma^{\hh}_{0,n}(1,N)$, $\G^{\hh}_{1,n}(N) :=
\G^{\hh}_{1,n}(1,N)$ and
$\G^{\hh}_n(N) := \G_n^{\hh}(1,N)$. When $n=2$ we drop it from notation.  Note 
that
the groups
$\G_0^{\hh}(N)$, $\G^{\hh}_1(N)$ and $\G^{\hh}(N)$ are 
$U_n$-analogues
of the standard
congruence subgroups $\G_0(N)$, $\G_1(N)$ and $\G(N)$ of $\SL_2(\bfZ)$. 
In
general the superscript `h' will indicate that an object is in some way
related to the group $U_n$. The letter `h' stands for `hermitian', 
as
this is the standard name of modular forms on $U_n$.

\subsection{Modular forms}
In this paper we will make use of the theory of modular forms on
congruence subgroups of two
different groups: $\SL_2(\bfZ)$ and $U(\bfZ)$. We will use both 
the
classical
and the adelic formulation of the theories. In the adelic framework one
usually speaks of automorphic forms rather than modular forms and in 
this
case $\SL_2$ is usually replaced with $\GL_2$. For more details see e.g.
\cite{Gelbart75}, chapter 3. In the
classical setting
the modular forms on congruence subgroups of $\SL_2(\bfZ)$ will be
referred to as \textit{elliptic
modular forms}, and those on congruence subgroups of $\G_{\bfZ}$
as
\textit{hermitian modular forms}. Since the theory of elliptic modular
forms is well-known we will only summarize the main facts below. Section
\ref{Hermitian modular forms} will be devoted to hermitian modular 
forms.

\subsubsection{Elliptic modular forms}

The theory of elliptic modular forms is well-known, so we omit most of 
the
definitions and
refer the reader to standard sources, e.g. \cite{Miyake89}. Let
$$\mathbf{H}:= \{ z\in \bfC
\mid \Ur (z) >0 \}$$ denote the complex upper half-plane. In the case of
elliptic modular
forms we will denote by $\G_0(N)$ the subgroup of $\SL_2(\bfZ)$ 
consisting
of matrices whose
lower-left entries are divisible by $N$, and by $\G_1(N)$ the subgroup 
of
$\G_0(N)$ consisting
of matrices whose upper left entries are congruent to 1 modulo $N$.
Let $\G \subset \SL_2(\bfZ)$ be a congruence
subgroup. Set $M_m(\G)$ (resp. $S_m(\G)$) to denote the $\bfC$-space of   
elliptic modular
forms (resp. cusp forms) of weight $m$ and level $\G$. We also denote by 
$M_m(N, \p)$ (resp.
$S_m(N, \p)$) the space of elliptic modular forms (resp. cusp forms) of   
weight $m$, level $N$
and
character $\p$. For $f, g \in M_m(\G)$ with either $f$ or $g$ a cusp 
form,
and $\G' \subset
\G$ a finite index subgroup, we define the Petersson inner product
$$\left< f,g \right>_{\G'} := \int_{\G' \setminus \mathbf{H}} f(z)
\ov{g(z)}
(\Ur
z)^{m-2} \hs dx\hs dy,$$
and set $$\left< f,g \right>:=
\frac{1}{[\ov{\SL_2(\bfZ)}:\ov{\G}']} \left< f,g \right>_{\G'},$$ where
$\ov{\SL_2(\bfZ)}:=
\SL_2(\bfZ)/\left<-I_2\right>$ and $\ov{\G}'$ is the image of $\G'$ in
$\ov{\SL_2(\bfZ)}$. The
value $\left< f,g \right>$ is independent of $\G'$.

Every elliptic
modular form $f \in M_m(N, \p)$ possesses a Fourier expansion $f(z) =   
\sum_{n=0}^{\iy} a(n) q^n$, where throughout this paper in such series
$q$ will denote  
$e(z):=e^{2 \pi i z}$. For $\g = \bsmat a&b \\ c&d \esmat \in
\GL^+_2(\bfR)$,
set $j(\g, z) = cz+d$.
  
Let $D=D_K$ be a prime. In this paper we will be particularly interested in the space
$S_m(D_K,\chi_K)$, where $\chi_K$ is the quadratic character of $(\bfZ/D
\bfZ)^{\times}$ associated with the extension $K=\bfQ(\sqrt{-D})$.
Regarded as a function $\bfZ \rightarrow \{ 1, -1 \}$, it assigns the
value $1$ to
all prime numbers $p$ such that $(p)$ splits in
$K$ and the value $-1$ to all prime numbers $p$ such that $(p)$ is inert
in $K$. Note that since the character $\chi_K$ is primitive, the space $S_m(D_K, \chi_K)$ has a basis consisting of primitive
normalized eigenforms. We will denote this (unique) basis by $\mN$. For
$f=
\sum_{n=1}^{\iy} a(n) q^n \in \mN$, set $f^{\rho}:=
\sum_{n=1}^{\iy} \ov{a(n)} q^n\in \mN$.

\begin{fact} \textup{(\cite{Miyake89}, section 4.6)} \label{fact1}  One
has $a(p) =
\chi_K(p)\ov{a(p)}$ for any rational prime $p \nmid D_K$.
\end{fact}
\no This implies that
$a(p) = \ov{a(p)}$ if $(p)$ splits in $K$ and $a(p) = -\ov{a(p)}$ if 
$(p)$
is inert in $K$.
   
For $f \in \mN$ and $E$ a finite extension of
$\bfQ_{\ell}$ containing the eigenvalues of $T_n$, $n=1,2, \dots$ we 
will
denote by $\rho_f: G_{\bfQ}
\rightarrow \GL_2(E)$ the Galois representation attached to
$f$ by Deligne (cf. e.g., \cite{DDT}, section 3.1). We will write
$\ov{\rho}_f$ for the reduction of $\rho_f$ modulo a uniformizer of $E$
with respect to some lattice $\Lambda$ in $E^2$. In general 
$\ov{\rho}_f$
depends on
the lattice $\Lambda$, however the isomorphism class of its
semisimplification 
$\ov{\rho}_f^{\tuss}$ is independent of $\Lambda$. Thus, if 
$\ov{\rho}_f$
is irreducible (which we will assume), it is well-defined.

\section{Hermitian modular forms} \label{Hermitian modular forms}

\subsection{Classical theory} \label{Classical theory}
  
For $n>1$, set $\bfi_n:=i I_n$ and define $$\bfH_n:= \{Z \in M_n(\bfC)
\mid -\bfi_n (Z- Z^*) >0 \}.$$ We call $\bfH_n$ the \textit{hermitian
upper half-plane of degree $n$}. The group $G_n^+(\bfR)=\{x \in  
G_n(\bfR) \mid \mu(x)>0\}$ acts transitively
on $\bfH_n$ via $$g Z:= (a_g Z + b_g) (c_gZ + d_g)^{-1}.$$
\begin{definition} \label{congruence subgroups} We say that a subgroup 
$\G
\subset G_n^+(\bfR)$ is a \textit{congruence subgroup} if 
\begin{itemize}
\item $\Gamma$ is commensurable 
 with $U_n(\bfZ)$, and \item there exists $N 
\in \bfZ_{>0}$ such that $\G \supset \G^{\rm h}(N):= \{ g \in U_n(\bfZ) \mid g
\equiv I_{2n} \pmod{N}\}$. \end{itemize} \end{definition}
Note that every congruence subgroup $\Gamma$ must be contained in $U_n(\bfQ)$, because commensurability with $U_n(\bfZ)$ and the fact that $\Gamma \subset G^+_n(\bfR)$ force $\Gamma \subset G_n(\bfQ)$ and $\mu(\Gamma)$ to be a finite subgroup of $\bfR^{\times}_+$ hence to be trivial.

For $g \in G_n^+(\bfR)$ and $Z \in \bfH_n$ set $$j(g, Z):= \det(c_{\g} Z 
+
d_{\g}),$$ and
for a positive integer $k$, a non-negative integer $\nu$ and a function
$F: \bfH_n \rightarrow
\bfC$
define $$F|_{k,\nu}
g(Z):= \det(g)^{-\nu}j(g,Z)^{-k} F(gZ).$$
When $\nu=0$, we will usually drop it from notation and simply write
$F|_kg(Z)$.

\begin{definition} \label{hermforms} Let $\Gamma \subset G_n^+(\bfR)$ be  
a congruence subgroup. We say that a function $F: \bfH_n
\rightarrow \bfC$ is a \textit{hermitian semi-modular form of weight
$(k,\nu)$
and level
$\G$}
if $F|_{k,\nu} \g = F$ for every $\g \in \G$. If in addition $F$ is
holomorphic, we call it a \textit{hermitian modular form of
weight $(k,\nu)$ and level
$\Gamma$}. The space of hermitian semi-modular (resp. modular) forms of  
weight
$(k,\nu)$ and
level
$\Gamma$ will be denoted by $M_{n,k,\nu}^{\textup{sh}}(\Gamma)$ (resp.
$M^{\tuh}_{n,k,\nu}(\Gamma)$). We also set $M^{\rm h}_{n, k,\nu}=M^{\rm h}_{n,k,\nu}(U_n(\bfZ))$. If $n=2$ or $\nu=0$ we drop them from notation. 
\end{definition}

Set $\mJ(K) =\frac{1}{2}\# \OK^{\times}$. Note that $\mJ(K)=1$ when $D_K >12$.

\begin{rem} \label{hel} Suppose $\Gamma \subset U_n(\bfZ)$. It is a Theorem of Hel
Braun (\cite{Braun51}, Theorem I on
p. 143) that $\det U_n(\bfZ) = \{u^2 \mid u \in \OK^{\times}\}.$ This
in particular implies that $(\det U_n(\bfZ))^{-\nu} =\{1\}$ if $\mJ (K) \mid \nu$. In such case,
we have
$M_{n,k,\nu}^{\textup{sh}}(\G)=M_{n,k}^{\textup{sh}}(\G)$. \end{rem}
  
If $\Gamma = \G^{\hh}_{0,n}(N)$ for some $N \in
\bfZ$, then we say
that $F$
is of level $N$. Forms of level 1 will
sometimes be referred to as forms
of \textit{full
level}. One can also define hermitian semi-modular forms with a 
character.
Let
$\Gamma
= \G_{0,n}^{\hh}(N)$ and let $\psi: \AK^{\times} \rightarrow
\bfC^{\times}$ be
a Hecke character
such that for all finite $p$, $\psi_p(a)=1$ for every $a \in \Oo_{K,
p}^{\times}$ with $a-1\in
N \Oo_{K,p}$.
We say that $F$ is \textit{of level $N$ and character $\psi$} if
$$F|_m \g =  \p_N(\det a_{\g}) F \quad \textup{for every} \hs \g \in
\G^{\hh}_{0,n}(N).$$
Denote by $M^{\tush}_{n,k}(N, \p)$ (resp. $M^{\hh}_{n,k}(N, \p)$) the
$\bfC$-space
of
hermitian semi-modular (resp. modular) forms
of weight $k$, level $N$ and character $\psi$. If $n=2$ we drop it from notation.

Write $Z \in \bfH_n$ as $Z=X+\bfi_n Y$, where $X=\Rz (Z)$ and $Y=\Ur(Z)$. Let $M_n\cong \bfG_a^{n^2}$ denote the additive group of $n\times n$ matrices.
A hermitian semi-modular form of level $\G_n^{\hh}(M,N)$ possesses
a
Fourier expansion
$$ F(Z) = \sum_{\tau \in \mathcal{S}_n(M)(\bfZ)} c_F(\tau, Y) e(\tr \tau X),$$
where $\mathcal{S}_n(M)(\bfZ)=\{x \in S_n(\bfZ) \sep \tr
xL(M) \subset \bfZ \}$ with $S_n=\{ h \in \Res_{\OK/\bfZ}M_{n/\OK} \mid h^* = h\}$ and   
$L(M) = S_n(\bfZ) \cap M_n(M \OK)$. As usually when $n=2$, we drop it from
notation.
As we
will be particularly interested in the
case when $M=1$, we set $$\mathcal{S}:=\mathcal{S}(1) = \left\{ \bmat
t_1
& t_2 \\ \ov{t_2} & t_3 \emat \in M_2(K) \mid t_1, t_3 \in \bfZ, t_2
\in
\frac{1}{2} \OK \right\}.$$
If $F$ is
holomorphic
the dependence of $c(h,Y)$ on $Y$ is explicit: $$c_F(h,Y) = e(\tr 
(\bfi_n
h Y)) c_F(h),$$
where
$c_F(h)$ depends only on $h$. Then one can write $$F(Z) = \sum_{h \in
S_n(\bfQ)}
 c_F(h) e(\tr (hZ)).$$

For $F \in
M^{\hh}_{n,k}(\Gamma)$ and $\alpha \in
G^+_{n}(\bfR)$ one has $F|_k {\alpha} \in M_{n,k}^{\rm h}(\alpha^{-1} \Gamma
\alpha)$ and there is an
expansion $$F|_k \a = \sum_{\tau \in S_n(\bfQ)} c_{\alpha} (\tau) e(\tr \tau
Z).$$ We
call $F$ a
\textit{cusp form} if for all $\alpha \in G^+_{n}(\bfR)$,
$c_{\alpha}(\tau) = 0$
for every $\tau$
such that $\det \tau = 0$. Denote by $S^{\hh}_{n,k}(\Gamma)$ (resp.
$S^{\hh}_{n,k}(N,
\p)$) the subspace of cusp forms inside $M^{\hh}_{n,k}(\Gamma)$ (resp.
$M^{\hh}_{n,k}(N,
\psi)$). If $\psi = \mathbf{1}$, set $M^{\tush}_{n,k}(N):= M^{\tush}_{n,k}(N,
\mathbf{1})$ and
$S^{\hh}_{n,k}(N):= S^{\hh}_{n,k}(N, \mathbf{1})$. If $n=2$ we drop it from notation.

\begin{thm} [$q$-expansion principle, \cite{Hida04}, section 8.4]
\label{qexpansion12} Let
$\ell$ be a rational
prime and $N$ a positive integer with $\ell \nmid N$. Suppose all 
Fourier
coefficients of $F
\in M^{\hh}_{n,k}(N, \psi)$ lie inside the
valuation ring $\Oo$ of a finite
extension $E$ of $\bfQ_{\ell}$. If $\g \in U_n(\bfZ)$, then all Fourier
coefficients of
$F|_m \gamma$ also lie in
$\Oo$. \end{thm}

If $F$ and $F'$ are two hermitian modular
forms
of weight $k$, level $\Gamma$ and character $\psi$, and either $F$
or
$F'$ is
a cusp form, we define for any finite index subgroup 
$\Gamma'$
of
$\Gamma$, the
Petersson inner product $$\left< F,F' \right>_{\Gamma'} :=
\int_{\Gamma'
\setminus
\mathbf{H}_n} F(Z) \ov{F'(Z)} (\det Y)^{m-4} dXdY,$$ where $X=\Rz{Z}$ 
and
$Y=\Ur{Z}$, and $$\left< F,F' \right> = [\ov{U_n(\bfZ)}:
\ov{\Gamma}']^{-1}
\left<
F,F' \right>_{\Gamma'},$$
where $\ov{U_n(\bfZ)}:=
U_n(\bfZ)/\left<\bfi_{2n}\right>$ 
and $\ov{\Gamma}'$ is the image 
of
$\G'$ in
$\ov{U_n(\bfZ)}$. The
value $\left< F,F' \right>$ is independent of $\Gamma'$.

\subsection{Adelic theory} \label{Adelic theory}

\begin{notation} We adopt the following notation. If $H$ is an algebraic
group over $\bfQ$, and $g \in H(\AQ)$, we will write $g_{\infty} \in
H(\bfR)$ for the infinity component of $g$ and $g_{\rm f}$ for the finite component of $g$, i.e., $g=(g_{\infty},   
g_{\tuf})$. \end{notation}

\begin{definition} \label{auto} Let $\mK$ be an open
compact subgroup of $G_n(\AQf)$. Write $Z_n$ for the center of $G_n$.
Let
$\mM'_{k,\nu}(\mK)$
denote the $\bfC$-space
consisting of functions $f : G_n(\AQ) \rightarrow \bfC$ satisfying the   
following conditions: 
\begin{itemize} \item $f (\gamma g) = f(g)$ for 
all
$\gamma
\in G_n(\bfQ)$, $g \in G_n(\AQ)$, \item $f (g \kappa) = f(g)$ for all
$\kappa \in \mK$, $g \in G_n(\AQ)$, \item
$f(gu) = (\det u)^{-\nu}j(u,\bfi_n)^{-k}f(g)$ for all $g \in
G_n(\AQ)$, $u \in \mK_{\infty}=\mK_{0,n, \infty}$ (see (10.7.4) in
\cite{Shimura97}),
\item $f(ag) = a^{-2n\nu - nk} f(g)$ for all $g \in
G_n(\AQ)$ and all $a \in \bfC^{\times} = Z_n(\bfR) \subset G_n(\bfR)$. 
\end{itemize}
Let $\psi : \AK^{\times} \rightarrow
\bfC^{\times}$ be
a Hecke character of conductor dividing $N$. Set \begin{multline} \mM'_k(N,\psi):= \{ f
\in \mM'_{k}(\mK_{1,n}(N)) \mid f(\gamma g (\kappa_{\infty},
\kappa_{\tuf})) =\\
=\psi'_N(\det (a_{\kappa_{\tuf}}))^{-1} j(\kappa_{\infty}, \bfi_n)^{-k}
f(g),
g \in G_n(\AQ), \gamma \in G_n(\bfQ), (\kappa_{\infty}, \kappa_{\rm f}) \in K_{0,n}(N)\}.\end{multline} \end{definition}

\begin{rem} Note that the center $Z_{n}(\bfR)$ of $G_n(\bfR)$ acts via the infinite part of a Hecke character of infinity type $(-2n\nu-nk,0)$. 
In particular this action is trivial if $\nu=-k/2$. \end{rem}

It is well-known
(see e.g., \cite{Bump97}, Theorem 3.3.1) that for any
finite subset $\mB$ of $\GL_n(\AKf)$ of cardinality $h_K$ with the
property that the canonical projection $c_K: \AK^{\times}
\twoheadrightarrow  
\Cl_K$ restricted to $\det \mB$ is a bijection, the following
decomposition
holds \be \label{gwiazdka} \GL_n(\AK) = \bigsqcup_{b \in \mB} \GL_n(K) \GL_n(\bfC) b
\GL_n(\hat{\Oo}_K).\ee
We will call any such $\mB$ a \textit{base}. We always assume that a 
base
comes with a fixed ordering, so in particular if we consider a tuple     
$(f_b)_{b \in \mB}$
indexed by elements of $\mB$, and apply a non-trivial permutation $\sigma$ to the
elements
$f_b$, we do not consider the tuples $(f_b)_{b \in \mB}$ and
$(f_{\sigma(b)})_{b \in
\mB}$ to be the same.

\begin{prop} \label{decomp657}
Let $\mK$ be a compact subgroup of $G_n(\AQf)$ such that $\det \mK \supset \hat{\Oo}_K^{\times}$. There exists a finite subset $\mC \subset U_n(\AQf)$ such that the
following
decomposition holds
\be \label{Undecomp} G_n(\AQ) = \bigsqcup_{c
\in \mC} G_n(\bfQ) G^+_n(\bfR) c \mK.\ee Moreover, each
element of $\mC$ can be taken to be of the form $p_b$ for some $b$
in a fixed base $\mB$. The same holds for $U_n$ in place of $G_n$.\end{prop}

\begin{proof} This is proved like 
Lemmas
5.11(4) and 8.14 of \cite{Shimura97}.
\end{proof}

We will call any set $\mC$ of cardinality $h_K$ for which the
decomposition (\ref{Undecomp}) holds a \emph{unitary base}.

\begin{cor} \label{scalarcor} If $(h_K, 2n)=1$ a base $\mB$ can be
chosen
so that for all $b \in \mB$ the matrices $b$ and $p_b$ are
scalar matrices and $b b^* = b^* b = I_n$. \end{cor}

\begin{proof} It follows from the Tchebotarev Density Theorem, that
elements of $\Cl_K$ can be represented by prime ideals. Since all
the inert
ideals are principal, $\Cl_K$ can be represented by prime ideals
lying over split primes of the form $p\OK=\fp \ov{\fp}$. Let $\Sigma$ be 
a
representing
set consisting of such ideals $\fp$. As $(2n,h_K)=1$, the set
$\Sigma^{2n}$
consisting of elements of $\Sigma$
raised to the power $2n$ is also a representing set for  
$\Cl_K$. Moreover, as
$\fp \ov{\fp}$ is a principal ideal, $\ov{\fp} = \fp^{-1}$ as
elements of $\Cl_K$, hence $\Sigma':=\{ \fp^n \ov{\fp}^{-n}\}_{\fp\in
\Sigma}$
also represents all the elements $\Cl_K$. Elements
of $\Sigma'$ can be written adelically as $\alpha_{\fp}^n$,
with $\alpha_{\fp}=(1,1, \dots,
1,
p, p^{-1}, 1,
\dots) \in \AKf$, where $p$ appears on the $\fp$-th place and $p^{-1}$
appears at the $\ov{\fp}$-th place. Set $b_{\fp}= \alpha_{\fp} I_n$.
Then we can take $\mB=\{ b_{\fp}\}_{\fp^n \ov{\fp}^{-n} \in \Sigma'}$
and we have $p_{b_{\fp}}=\alpha_{\fp}
I_{2n}$. It is also clear that $b b^*=b^* b =I_n$.
\end{proof}

\begin{prop}\label{isom99} Suppose that $\mK\subset G_n(\AQf)$ is a compact subgroup such that $\det \mK \supset \hat{\Oo}_K^{\times}$. Let $\mN'_{k,\nu}(\mK)$ denote the $\bfC$-space of functions $f: U_n(\AQ) \rightarrow \bfC$ satisfying the conditions of Definition \ref{auto}, but with $g \in U_n(\AQ)$. Then the map $f \mapsto f|_{U_n(\AQ)}$ gives an isomorphism $\mM'_{k,\nu}(\mK) \cong \mN'_{k,\nu}(\mK)$. \end{prop}

\begin{proof}
One has the following short exact sequence of group schemes over $\bfZ$:
$$1 \to U_n \to G_n \xrightarrow{\mu} \bfG_m \to 1.$$
We first show that the induced map \be \label{changegroup} f \mapsto f|_{U_n(\AQ)}: \quad \mM'_{k,\nu}(\mK) \rightarrow \mN'_{k,\nu}(\mK)\ee is injective. Indeed, let $f \in  \mM'_{k,\nu}(\mK)$. Using Proposition \ref{decomp657} we write any $g \in G_n(\AQ)$ as $g = \gamma g_{\bfR} c$ with $\gamma \in G_n(\bfQ)$, $g_{\bfR} \in G_n^+(\bfR)$, $c  \in U_n(\AQf)$.  Then $f(g) = f(g_{\bfR} c)$. Let $x=\diag((\mu(g_{\bfR}))^{1/2}, (\mu(g_{\bfR}))^{1/2}, \dots)$. Then $\mu(x)=\mu(g_{\bfR})$ and $x \in Z_{n}(\bfR)$. Set $y=x^{-1}g_{\bfR}$. Then $\mu(y)=1$. Hence by Definition \ref{auto} $$f(g) = f(xyc) = \mu(g_{\bfR})^{-n\nu-nk/2}f(yc)= f(yc),$$ so $f$ is completely determined by what it does on $U_n(\AQ)$.

It remains to show the surjectivity of (\ref{changegroup}). To do this we need to show that every $f \in \mN'_{k,\nu}(\mK)$ has an extension to $G_n(\AQ)$. Fix a unitary base $\mC$. Let $f \in \mN'_{k,\nu}(\mK)$. Let $g=\gamma g_{\bfR} p_b \kappa$ with $\gamma\in G_n(\bfQ)$, $g_{\bfR}\in G_n^+(\bfR)$, $c \in \mC$ and $\kappa \in \mK$. Write $g_{\bfR} = \sqrt{\mu(g_{\bfR})} \cdot y$. Then $y \in U_n(\bfR)$. Set $f(g) := \mu(g_{\bfR})^{-n\nu-nk/2} f(yc)$. We need to show that this extension of $f$ is well-defined. Let $g=\gamma' g'_{\bfR} c \kappa'$ be a different decomposition of $g$ with $\gamma' \in G_n(\bfQ)$, $g'_{\bfR} \in G_n^+(\bfR)$, $\kappa' \in \mK$. Then $g=(\gamma g_{\bfR}, \gamma c \kappa) = (\gamma' g'_{\bfR}, \gamma' c \kappa')$, where the first component is the $\infty$-component of $g$ and the second is the finite component of $g$. Thus $g'_{\bfR} = (\gamma')^{-1} \gamma g_{\bfR}$ and $c = (\gamma')^{-1} \gamma c \kappa \kappa'$. The latter equality implies that $\mu((\gamma')^{-1}\gamma) \in \hat{\bfZ}^{\times} \cap \bfQ^{\times} = \{\pm 1\}$ which combined with the first equality and the fact that $g_{\bfR}, g'_{\bfR} \in G^+_n(\bfR)$ implies that $\mu(g'_{\bfR}g_{\bfR}^{-1}) = \mu((\gamma')^{-1}\gamma)=1$, i.e., in particular $\mu(g_{\bfR}) = \mu(g'_{\bfR})$ and $(\gamma')^{-1}\gamma\in U_n(\bfQ)$. Write $g'_{\bfR} = \sqrt{\mu(g'_{\bfR})}y'$ with $y'\in U_n(\bfR)$. We have \begin{multline}\mu(g'_{\bfR})^{-n\nu-nk/2} f(y'c) = \mu(g'_{\bfR})^{-n\nu-nk/2}f((\mu(g'_{\bfR})^{-1/2}(\gamma')^{-1}\gamma g_{\bfR}, c)\\ = \mu(g_{\bfR})^{-n\nu-nk/2}f((\mu(g_{\bfR})^{-1/2}(\gamma')^{-1}\gamma g_{\bfR}, (\gamma')^{-1}\gamma c \kappa\kappa') = \mu(g_{\bfR})^{-n\nu-nk/2}f(yc).\end{multline}
\end{proof}
In view of Proposition \ref{isom99} in what follows we will often not distinguish between automorphic forms defined on $G_n(\AQ)$ and those on $U_n(\AQ)$.

Every $f \in \mM_{n,k,\nu}(\mK)$ possesses a 
Fourier expansion, i.e., for every $q \in
\GL_n(\AK)$, and every $h \in S(\bfQ)$ there exists a complex number 
$c_f(h,q)$ such that one has \be \label{fe143} f\left( \bmat I_n & 
\sigma \\ & I_n \emat \bmat q \\ & \hat{q} \emat
\right) = \sum_{h \in S_n(\bfQ)} c_f(h,q) e_{\AQ} (\tr h \sigma)\ee for every 
$\sigma \in S_n(\AQ)$. Here $e_{\AQ}$ is defined in the following way. 
Let $a = (a_v) \in \AQ$, where $v$ runs over all the places of $\bfQ$. If 
$v=\iy$, set $e_v(a_v) = e^{2 \pi i a_v}$. If $v =p$, a finite prime, set $e_v(a_v) = 
e^{-2 \pi i y}$, where $y$ is a rational number such that $a_v-y \in 
\bfZ_p$. Then we set 
$e_{\AQ}(a) = \prod_v e_v(a_v)$.

Suppose $2 \nmid h_K$.
For $g\in
U_n(\AQ)$, write $g = \g g_0 c  \in U_n(\bfQ)
U_n(\bfR) U_n(\AQf)$ with $c \in U_n(\AQf)$ and $g_0$ such that $\det  
g_0 = |\det g_0| e^{i \varphi}$ satisfies $0 \leq \varphi < 2\pi/\mJ(K)$.
Note that such a $g_0$ exists and $\det
g_0^{\nu}$ is
independent of the choice of $g_0$.
\begin{prop} \label{adelicclassical} Let $f \in \mM_{n,k,\nu}(\mK)$. Let $g=(g_{\infty}, 1) \in U_n(\bfR) U_n(\AQf)$. Set $Z:= g_{\infty}
\bfi_n$. Let $\mC$ be a unitary base.
For $c \in \mC$, set $f_c(Z)
=
(\det g_{\iy})^{\nu}j(g_{\iy}, \bfi_n)^{k}f(g_{\infty}c)$ and write
$\Gamma_c$ for
$U_n(\bfQ) \cap
(G_n^+(\bfR) \times c\mK c^{-1})$.
The
map
$f
\mapsto (f_c)_{c \in \mC}$ defines a $\bfC$-linear isomorphism
$\Phi_{\mC}: \mM'_{n,k,\nu}(\mK) \xrightarrow{\sim} \prod_{c \in \mC}
M^{\textup{sh}}_{n,k,\nu}(\G_c).$ \end{prop}
\begin{proof} This follows from \cite{Shimura97}, section 10. \end{proof} If $h_K$ is odd,
$\mB$
is a base
and $\mC = \{p_b\}_{b \in \mB}$, we write $\G_b$ instead of $\G_{p_b}$ 
and
$f_b$ instead of $f_{p_b}$ for $b \in \mB$, and $\Phi_{\mB}$ instead of
$\Phi_{\mC}$.

\begin{definition} Let $\mC$ be a unitary base. A function $f \in \mM'_{n,k,\nu}(\mK)$ whose image 
under
the
isomorphism $\Phi_{\mC}$ lands in $\prod_{c \in \mC} M^{\tuh}_{n,k,\nu}(\G_c)$
will be  
called
an
\textit{adelic hermitian modular form of weight $(k, \nu)$ and level
$\mK$}. The
space of
hermitian modular forms of weight $(k,\nu)$ will be denoted by
$\mM_{n,k,\nu}(\mK)$.
Moreover, we
set $\mM_{n,k,\nu}:=\mM_{n,k,\nu}(U_n(\hat{\bfZ}))$. When $\nu=0$ or $n=2$ we drop them from notation.
\end{definition} We clearly have \be \label{prod1} \mM_{n,k,\nu}(\mK) 
\cong
\prod_{c
\in \mC}
M^{\tuh}_{n,k,\nu}(\G_c).\ee  
 
Let $\chi: \Cl_K \rightarrow \bfC^{\times}$ be a character and choose a 
base $\mB$ consisting of scalar matrices $b$ such that $bb^*=I_n$. Such a base always exists when $(h_K, 2n)=1$ by Corollary \ref{scalarcor}. Write $Z_{n}$ for the center of $G_n$. Let $z =
\gamma z_0 p_b \kappa$ with $\gamma \in Z_{n}(\bfQ)$, $z_0 \in Z_{n}(\bfR)$, $b \in \mB$ and $\kappa \in (\mK \cap Z_{n}(\AQf))$ be an
element of the center with $z_0= \zeta I_{2n}$. If $f \in
\mM_{n,k,\nu}(\mK)$, then $$f(zg) =
f( z_0 p_b g) = \zeta^{-2n\nu-nk} f(p_bg).$$ Set
$$\mM_{n,k,\nu}^{\chi}(\mK)=\{f \in \mM_{n,k,\nu}(\mK) \mid 
f(p_bg)
= \chi(b) f(g)\},$$ where we consider $b$ as an element of $\Cl_K$ under
the identification $\mB = \Cl_K$ given by $b \mapsto c_K(\det b)$ with $c_K : \AK^{\times} \twoheadrightarrow \Cl_K$. Then
\be \label{de1} \mM_{n,k,\nu}(\mK)=\bigoplus_{\chi \in \Hom(\Cl_K,
\bfC^{\times})} \mM_{n,k,\nu}^{\chi}(\mK).\ee
   
By adapting the proof of Lemma A5.1 in \cite{Shimura00} to the case of a trivial conductor one can show that for every integer $s$ with $\mJ(K) \mid s$, there exists a Hecke character $\beta: \AK^{\times} \to \bfC^{\times}$ unramified everywhere, trivial on $\AQ^{\times}$ and such that $\beta(z) = \frac{z^{2s}}{|z|^{2s}}$ for $z\in \bfC^{\times} = K_{\infty}^{\times}$. We will denote the set of such characters by $\textup{Char}(s)$. Note that every element $\beta$ of ${\rm Char}(s)$ is unitary, i.e., $\ov{\beta(x)} = \beta(x)^{-1}$. If $h_K$ is odd, one has $\# \textup{Char}(s)=h_K$ by \cite{Shimura97}, Lemma 8.14. Using the Iwasawa decomposition one sees that for $g \in U_n(\AQ)$, one  has $\det g = a \ov{a}^{-1}\det Y$ for some $a \in \AK^{\times}$ and $Y \in U_n(\hat{\bfZ})$. By Lemma 5.11(4) in \cite{Shimura97} we see that $\det Y = \kappa \ov{\kappa}^{-1}$ for some $\kappa \in \hat{\Oo}_K^{\times}$. So, absorbing $\kappa$ into $a$ we in fact we have $\det g = a \ov{a}^{-1}$ for some $a \in \AK^{\times}$. Moreover, if $a\ov{a}^{-1} = b \ov{b}^{-1}$, then $\frac{a}{b} \in \AQ^{\times}$. Thus for every integer $s$, every  $\beta \in {\rm Char}(s)$ and every $g\in U_n(\AQ)$, the map $g \mapsto \beta(a)$ is a well-defined character on $U_n(\AQ)$. Abusing notation we will write $\beta(g)$ instead of $\beta(a)$.  

\begin{prop}\label{newmap}
Assume $\mJ(K) \mid \nu$ and $(2n, h_K)=1$. Let $\beta \in {\rm Char}(-\nu)$. Let $\mB$ be a base as in Corollary \ref{scalarcor} - note that then $\beta( p_b)=\beta(\det b)$ for $b \in \mB$. Let $\mK \subset G_n(\AQf)$ be an open compact subgroup. Then $\Gamma := U_n(\bfQ) \cap (G_n^+(\bfR) \times p_b \mK p_b^{-1})$ is independent of $b\in \mB$. Assume $\Gamma \subset U_n(\bfZ)$. 
Then $M_{n,k,\nu}^{\tuh}(\Gamma) =
M_{n,k}^{\tuh}(\Gamma)$ and one has
the following commutative diagram in which all the maps are isomorphisms \be
\label{cen2} \xymatrix{\mM_{n,k}^{\tuh}(\mK) \ar[d]^{\Psi_{\beta}}_{\sim}  
\ar[r]^{\Phi_{0}}_{\sim} & \prod_{b\in \mB}
M_{n,k}^{\tuh}(\Gamma)\ar[d]^{\iota_{\beta}}_{\sim}\\
\mM_{n,k,\nu}^{\tuh}(\mK)\ar[r]^{\Phi_{\nu}}_{\sim}&\prod_{b\in \mB}
M_{n,k}^{\tuh}(\Gamma)},\ee where $\iota_{\beta}(f_b) = \beta(\det b) f_b$, 
for
$g=\gamma g_0 p_b \kappa \in U_n(\bfQ) U_n(\bfR) p_b \mK$ one has
$\Psi_{\beta}(f)(g) = \beta( g)f(g)$, and for $h \in U_n(\bfR)$
$\Phi_0(f)_b(h \bfi_n) = j(h, \bfi_n)^{k}f(hp_b)$ and $\Phi_{\nu}(f)_b(h \bfi_n) 
=
(\det h)^{\nu}j(h, \bfi_n)^{k}f(hp_b)$. The map $\Psi_{\beta}$ is
Hecke-equivariant; more precisely for $T=\mK a \mK$ with $a \in U_n(\AQ)$ one has   
$\Psi_{\beta}(Tf)(x) = \beta(a) (T \Psi_{\beta}(f))(x)$ (for the definition of the Hecke action see section \ref{Hecke operators}).\end{prop}

\begin{proof} This is straightforward using the results of this section (cf. (\ref{prod1}) and Remark \ref{hel}). \end{proof}

\section{Hecke operators} \label{Hecke operators}

\subsection{Hermitian Hecke operators}

We study Hecke operators acting on the space $\mM_{k,\nu}$ 
of hermitian
modular forms on $G(\AQ)=G_2(\AQ)$. We also set $U=U_2$. Let $p$ be
a rational prime write $\mK_p$ for $G(\bfZ_p)$.
 Let $\mH_p$ be the
$\bfC$-Hecke algebra generated by the double cosets $\mK_p g \mK_p$, $g
\in G(\bfQ_p)$ with the usual law of multiplication (cf. 
\cite{Shimura97},
section 11), and $\mH_p^+ \subset \mH_p$ be the subalgebra generated by $\mK_p g \mK_p$ with $g \in U(\bfQ_p)$.
If $\mK_pg\mK_p \in \mH_p$, there exists a finite set $A_g \subset
G(\bfQ_p)$ such that $\mK_p g \mK_p = \bigsqcup_{\alpha \in A_g} \mK_p  
\alpha$. For $f \in \mM_{k,\nu}$, $g \in G(\bfQ_p)$, $h \in G(\AQ)$, set  
$$([\mK_p g \mK_p]f)(h) = \sum_{\a
\in A_g } f(h \a^{-1}).$$ It is clear that $[\mK_p g \mK_p]f \in \mM_{k,\nu}$.

\begin{rem} Let $\mK_{0,p}:= \mK_p \cap U(\bfZ_p)$. Every element of
$\mH_p$
can be written as $\mK_p g \mK_p$ with $g$ a diagonal matrix. For 
$\kappa 
\in
\mK_p$ write $m_{\kappa}=
\diag(1,1, \mu(\kappa), \mu(\kappa))$. Then $h_{\kappa}=\kappa
m_{\kappa}^{-1} \in \mK_{0,p}$. Since $g$ is diagonal, $m_{\kappa}$  
commutes
with $g$, hence we get $$\mK_p g \mK_p = \mK_p g \mK_{0,p}.$$ From this 
it
follows
that $$\mK_{0,p} g \mK_{0,p} = \bigsqcup_{\alpha \in A_g} 
\mK_{0,p}\alpha
\implies \mK_p g \mK_p = \bigsqcup_{\alpha \in A_g} \mK_p g.$$ \end{rem}

\subsubsection{The case of a split prime} Let $p$ be a prime which 
splits
in $K$. Write $(p)=\fp \ov{\fp}$.
Recall that $G(\bfQ_p) \cong \GL_4(\bfQ_p)\times \bfG_m(\bfQ_p)$. An element
$g$
of $G(\bfQ_p)$ can be written as $g=(g_1, g_2) \in \GL_4(\bfQ_p) \times
\GL_4(\bfQ_p)$ with $g_2 = -\mu(g) J (g_1^t)^{-1} J.$
 Set \begin{itemize} \item $T_{\fp} := \mK_p (\diag(
1, p,p,p), \diag(1,1,p,1) )\mK_p$,
\item $T_{p}:= \mK_p (\diag (1,1, p,p) ,
\diag(1,1,p,p))\mK_p,$
\item $\Delta_{\fp} := \mK_p (pI_4,I_4)\mK_p.$ \end{itemize}
It is easy to see that the $\bfC$-algebra $\mH_p$ is generated by the
operators $T_{\fp}$, $T_{\ov{\fp}}$, $T_{p}$, $\Delta_{\fp}$,
$\Delta_{\ov{\fp}}$ and their inverses.

\begin{prop} \label{decompsplit1} We have the following decompositions
\be \begin{split} T_{\fp}  = & \bigsqcup_{a,b,c \in \bfZ/p \bfZ} \mK_p
\left(
\bsmat 1&
a & b & c \\ & p\\ && p\\ &&&p \esmat, \bsmat 1&&b\\ &1
& c \\ && p \\ && -a & 1 \esmat \right)  \sqcup \\
& \bigsqcup_{d,e \in \bfZ/p \bfZ}  \mK_p \left( \bsmat p \\ & 1 &
d
& e \\ && p \\ &&& p \esmat , \bsmat 1 &&& d \\ & 1 && e \\ && 1 \\
&&& p \esmat \right)\sqcup \\
& \bigsqcup_{f \in \bfZ/p \bfZ} \mK_p \left( \bsmat p\\ & p \\ && 1 &
 f \\ &&& p\esmat , \bsmat p \\ -f & 1 \\ && 1 \\ &&& 1 \esmat
\right) \sqcup \\
& \hspace{25pt} \mK_p \left( \bsmat p \\ & p \\ && p \\ &&& 1 \esmat,   
\bsmat 1 \\ & p
\\ && 1 \\ &&& 1 \esmat \right).  
\end{split} \ee

\be \begin{split} T_{p} = & \bigsqcup_{b,c,d,e \in \bfZ/ p \bfZ} \mK_p
\left( \bsmat 1 & & b & d \\ & 1 & c & e
\\ && p \\ &&&p \esmat , \bsmat 1 && b & c \\ & 1 & d & e \\ &&p \\ &&& 
p
\esmat \right) \sqcup \\ 
& \bigsqcup_{a,c,f \in \bfZ/p \bfZ} \mK_p \left( \bsmat p\\   
-f
& 1 &  c \\ && p \\ && -a &p \esmat, \bsmat 1 & a
&& c \\ &p
\\ &&1&f \\ &&& p \esmat \right) \sqcup\\
& \bigsqcup_{e,f \in \bfZ/p\bfZ} \mK_p\left( \bsmat p \\ -f & 1
&&
e \\ && 1 \\ &&&p \esmat, \bsmat p \\ & 1 && e \\ && 1& f \\
&&& p \esmat\right) \sqcup \\
& \bigsqcup_{a,b \in \bfZ/ p \bfZ} \mK_p\left( \bsmat 1 && b \\
&
p
\\ && p \\ && -a &1
\esmat , \bsmat 1 & a& b \\ & p \\ && p \\ &&&1 \esmat \right) \sqcup \\
& \bigsqcup_{d\in \bfZ/ p \bfZ} \mK_p\left( \bsmat 1 &&& d \\
&p
\\
&& 1 \\ &&& p \esmat , \bsmat p\\ & 1& d \\ &&p \\ &&&1 \esmat
\right) \sqcup \\
& \hspace{25pt} \mK_p \left( \bsmat p \\ & p \\ && 1 \\ &&& 1 \esmat,
\bsmat p
\\
&
p \\ && 1 \\ &&&1 \esmat \right).\end{split} \ee
\end{prop}

\begin{proof} This follows easily from the corresponding decompositions for the group $\GL_4(\bfQ_p)$. \end{proof}

\subsubsection{The case of an inert prime} Let $p$ be a prime which is
inert in $K$. Set \begin{itemize} \item $T_{p}:=\mK_p \diag(1, 1, p,
p ) \mK_p$,
\item $T_{1,p}:=\mK_p
 \diag (1, p, p^2,p ) \mK_p$,
\item $\Delta_p:= \mK_p pI_4 \mK_p$. \end{itemize}
The operators $T_{p}$, $T_{1,p}$, $\Delta_p$ and their inverses generate
the $\bfC$-algebra $\mH_p$.  

\begin{prop} \label{decompinert1} We have the following decompositions
\be \begin{split} T_{p}  = & \bigsqcup_{\substack{b,d\in \bfZ/p \bfZ \\ 
c
\in \OK/p \OK}} \mK_p \bsmat 1 && b & c \\ & 1 & \ov{c} & d \\ && p 
\\&&&
p
\esmat \sqcup \bigsqcup_{e \in \bfZ/p \bfZ} \mK_p\bsmat p  \\ & 1 && e\\
&&1
\\
&&& p \esmat \sqcup \\
& \bigsqcup_{\substack{a \in \OK/p\OK \\ b \in \bfZ/p\bfZ }} \mK_p 
\bsmat
1 & a &b
\\ &p  \\ && p \\ &&-\ov{a}& 1 \esmat \sqcup \mK_p \bsmat p \\
&p\\
&&1\\ &&&1 \esmat. \end{split} \ee
  
\be \begin{split} T_{1,p} = & \bigsqcup_
{\substack{a,c \in \OK/p \OK\\ b \in \bfZ/p^2 \bfZ}} \mK_p \bsmat 1 & a &
b+a \ov{c} & c  \\ & p & p\ov{c} \\ && p^2 \\ && -\ov{a}p & p \esmat
\sqcup
\bigsqcup_{\substack{c \in \OK/p \OK\\ d \in \bfZ/ p^2 \bfZ}} \mK_p 
\bsmat
p &&& pc \\ & 1 & \ov{c} & d \\ && p \\ &&&p^2 \esmat \\
& \bigsqcup_{a \in \OK/p\OK} \mK_p \bsmat p & pa \\ & p^2\\ &&p \\ &&
-\ov{a} & 1 \esmat \sqcup \mK_p \bsmat p^2 \\ & p \\ && 1 \\ &&& p 
\esmat
\sqcup \\
& \bigsqcup_{\substack{b,d \in \bfZ/p \bfZ \ \\ bd \equiv 0
\pmod{p}\\ (b,d) \neq (0,0)}} \mK_p \bsmat p && b \\ & p & & d \\ && p \\ &&& p \esmat \cup
\bigsqcup_{\substack{b \in (\bfZ/p \bfZ)^{\times}\\ c \in (\OK / p
\OK)^{\times}}} \mK_p \bsmat p && b & c \\ & p &        \ov{c} & |c|^2
b^{-1} \\
&&
p \\ &&& p \esmat \end{split} \ee

\end{prop}

\begin{proof} See the proof of Lemma 5.3 in \cite{Klosin09} and references cited there. \end{proof}

\subsection{Action of the Hecke operators on the Fourier coefficients}

Let $S=S_2$ be as in section \ref{Hermitian modular forms}. Write $\mS_p:= \mS(\bfZ_p)$ for $\{h \in S(\bfQ_p) \mid \tr (S(\bfZ_p)h) \subset \bfZ_p\}$.
For a matrix $h \in S(\AQ)$ such that $h_p \in \mS_p$ for every prime $p$,
set $$\epsilon_p(h) = \max \{n \in \bfZ \mid
\frac{1}{p^n} h_p \in \mS_p \}$$ and $$\epsilon(h) = \prod_{p\nmid \infty}
p^{\epsilon_p(h)}.$$
Note that $\epsilon_p(h) \geq
0$ for every $p$ and $\epsilon(h) = \epsilon(h_{\tuf})$.

For $f \in \mM_{k, \nu}$, $q \in \GL_2(\AK)$ and $h \in S(\bfQ)$ we write
$c_f(h,q)$ for the $(h,q)$-Fourier coefficient of $f$ as in (\ref{fe143}).

\subsubsection{The case of a split prime} Let $p$ be a prime which splits
in $K$. Let $\fp$ be a prime of $K$ lying over $p$ and denote by
$\ov{\fp}$ its conjugate. As before we simultaneously identify 
$G(\bfQ_p)$
with a subgroup
of
$\GL_4(\bfQ_p) \times \GL_4(\bfQ_p)$ (the first factor corresponds to
$\fp$ and the second one to $\ov{\fp}$) and with a subgroup of $G(\AQ)$.
Set $$T_{\fp,1}:= \Delta_{\fp}^{-1} T_{\fp},$$
$$T_{\fp,2}:= \Delta_{\fp}^{-1} T_{p}.$$
Note that the operators $T_{\fp, 1}$, $T_{\fp,2}$ and their inverses generate $\mH_p^+$.
Define the following elements of $\GL_2(\bfQ_p) \times \GL_2(\bfQ_p)$ 
which we regard as elements of $\GL_2(\AK)$,
 \be\begin{split}
\a_a=& \left(\bmat p &a\\ &1 \emat, I_2 \right), \quad a=0,1, \dots,
p-1,\\
\a_{p}= &
\left(\bmat 1 \\ &p \emat, I_2 \right),\\
\beta_{p}= &
\left(\bmat p \\ &p \emat, I_2 \right),\\
\g_a=& \left(\bmat 1 &a\\ &p \emat, I_2 \right), \quad a=0,1, \dots, 
p-1\\
\gamma_{p}= &
\left(\bmat p \\ &1 \emat, I_2 \right).\end{split}\ee
We will write $\pi_{\fp}\in \AK^{\times}$ for the adele whose $\fp$th component is $p$ and whose all other components are 1. Write 
 $\pi_{\fp} = \gamma_{\fp} \pi_{\fp, \infty} b_{\fp} \kappa_{\fp}$ with $\gamma_{\fp}\in K^{\times}$, $\pi_{\fp, \infty}=\gamma_{\fp}^{-1} \in \bfC^{\times}$, $b_{\fp}\in
\AKf^{\times}$, $\kappa_{\fp}\in \hat{\Oo}_K^{\times}$, so that $\val_p(b_{\fp}b_{\fp}^*)=0$ (this is always possible). 

\begin{prop} \label{f3} One has the following formulas  
$$c_{f'}(h,q) = \begin{cases} p^2 \sum_{a=0}^{p}
 c_f(h, q \alpha_a) + \sum_{a=0}^{p} c_f(h, q \hat{\alpha}_a) &
f'=T_{\fp,1}f;\\
 p^4 c_f(h,q \beta_p) + c_f(h, q \hat{\beta}_p) + p
\sum_{a=0}^p \sum_{b=0}^p c_f(h, q \gamma_a
\hat{\gamma}_b) & f'=T_{\fp,2}f; \\
\gamma_{\fp}^{-2k-4\nu} c_f(h,qb_{\fp}^{-1}) & f' = \Delta_{\fp}
f.\end{cases}$$\end{prop}

\begin{proof} This is an easy calculation using Proposition \ref{decompsplit1}. \end{proof}

\subsubsection{The case of an inert prime} Let $p$ be a prime which is
inert in $K$. Set $$T_{p, 0}:= \Delta_p^{-1} T_{1,p}.$$ Define the
following elements of $\GL_2(K_p) \subset \GL_2(\AK)$:
\be\begin{split} \a_a &= \bmat p & a \\ & 1 \emat \quad a \in \OK/p 
\OK,\\
\a_p& = \bmat 1 \\ & p \emat ,\\
\beta_a &= \bmat 1 & a p^{-1} \\ & p^{-1} \emat \quad a \in \OK/p \OK,\\
\beta_p & = \bmat p^{-1} \\ & 1 \emat. \end{split} \ee Write 
$\bfP^1(\OK/p
\OK)$ for the disjoint union of $\OK/p \OK$ and $p$.  Let $h \in S(\bfQ)$
and $q \in \GL_2(\AKf)$ and assume that $q^* h q \in \mS_p$. Since $p
\nmid D_K$, this implies that $q_p^* h q_p \in M_2(\Oo_{K,p})$, where
$q_p$ denotes the $p$-th component of $q$. Set $$s(h,q):= \begin{cases} 
p 
& \val_p(\det (q^* h q))=0;\\
-p(p-1) & \val_p(\det (q^* h q))>0, \hs \epsilon_p(q^* h q)=0;\\
p^2(p-1) & \epsilon_p(q^* h q)>0.\end{cases} $$

\begin{prop} \label{f4} Assume that $q^* h q \in \mS_p$. One has the 
following
formulas:
$$c_{f'}(h,q) = \begin{cases} s(h,q)  c_f(h, q) + p^4 \sum_{a \in
\bfP^1(\OK/p\OK)}
c_f(h, q \alpha_a) + \sum_{a \in \bfP^1(\OK/p\OK)} c_f(h, q \beta_a)
&
f'=T_{p,0}f;\\
p^{-2k+4}c_f(ph,q) + c_{f}(p^{-1}h,q) + p^{-k+1}\sum_{a \in
\bfP^1(\OK/p\OK)}
c_f(ph, \beta_a q)      & f'=T_{p}f; \\
p^{-4\nu-2k} c_f(h,q) & f' = \Delta_{p}
f.\end{cases}$$ If $q^* h q \not\in \mS_p$, $c_{f'}(h,q)=0$ in all
of the above cases. \end{prop}

Finally let us also note that in the inert case the algebra $\mH^+_p$ is generated by the operators $T_{p,0}$ and $U_p:=\Delta_p^{-1} T_p^2$ and their inverses.

\section{Maass space} 
\label{Maass space}

Let $S=S_2$, $\mS$ be as in section \ref{Hermitian modular forms} and $\mS_p$, $\epsilon$, $\epsilon_p$ as in section \ref{Hecke operators}. In this section we assume that $k$ is a positive integer divisible by $\# \OK^{\times}$.

\subsection{Definition and basic properties} \label{The Maass space}

\begin{definition} \label{Maass form} Let $\mB$ be a base. We say that
$f\in
\mM_{k, -k/2}$ is a $\mB$-\textit{Maass form} if there exist functions 
$c_{b,
f}:
\bfZ_{\geq 0} \rightarrow \bfC$, $b \in \mB$, such that for every $q \in
\GL_2(\AK)$ and every $h \in S(\bfQ)$ the Fourier coefficient $c_f(h,q)$
satisfies
\begin{multline} \label{Maass condition} c_f(h,q) = |\det
q_{\iy}|^k
e^{-2\pi\tr(q_{\iy}^* h q_{\iy})} |\det \g_{b, q}|^{-k}
\times
\\
\times \sum_{\substack{d \in \bfZ_{+} \\ d \mid \epsilon(q_{\rm f}^* h
q_{\rm f})}}
d^{k-1} c_{b, f}\left( D_K d^{-2} \det h \hs \prod_{p} p^{\val_p(\det
q_{\rm f}^*q_{\rm f})}\right),\end{multline} where $q_{\rm f} = \g_{b,
q} b
\kappa_q \in \GL_2(K) b \mK'$ for a unique $b
\in \mB$. Here $\mK'=\GL_2(\hat{\Oo}_K)$ is a maximal 
compact subgroup of
$\GL_2(\AKf)$. Also, here and in what follows we will often treat the $K$-points as embedded diagonally in the $\AKf$-points (i.e., instead of writing $q_{\rm f} = \gamma_{b,
q}q_{\infty} b
\kappa_q$ with $q_{\infty} = \gamma_{b,q}^{-1} \in \GL_2(\bfC)$ we will simply write $q_{\rm f} = \g_{b,
q} b
\kappa_q$ as above.) \end{definition}

\begin{rem} Note that by \cite{Shimura97}, Proposition 18.3(2), 
$c_f(h,q)
\neq 0$ only if $(q^* h q)_p \in \mS_p$, so $\epsilon_p(q_{\rm f}^* h
q_{\rm f})\geq 0$. Also, note that Definition \ref{Maass form} is
independent of the decomposition $q_{\rm f} = \g_{b,
q} b
\kappa_q\in \GL_2(K) b \mK'$. Indeed,
if $q_{\rm f} = \g'_{b,
q} b
\kappa'_q \in \GL_2(K) b \mK'$ is
another decomposition of $q_{\rm f}$, then
$$\det \g'_{b,
q} \det \g_{b,
q}^{-1} = \det (\kappa_q (\kappa'_q)^{-1})\in \hat{\Oo}_K^{\times} \cap
K^{\times} =
\OK^{\times},$$ 
so $\det (\g'_{b,
q})^k =  \det \g_{b,
q}^k$. \end{rem}

\begin{definition} \label{Maass space def} The $\bfC$-subspace of 
$\mM_{k, -k/2}$
consisting of $\mB$-Maass forms
will be called the $\mB$-\textit{Maass space}. \end{definition}

\begin{definition} Let $\mB$ be a base. We will say that $q \in 
\GL_2(\AK)$ belongs to a class $b \in \mB$ if there exist $\g \in 
\GL_2(K)$, $q_{\iy} \in \GL_2(\bfC)$ and $\kappa\in \mK'$ such that $q = \g 
b q_{\iy} \kappa$. \end{definition}

It is clear that the class of $q$ depends only on $q_{\rm f}$.

\begin{lemma} \label{sameclass} Suppose $r \in \GL_2(\AK)$ and $q \in
\GL_2(\AK)$ belong to the same class and $r_{\rm f} = \g q_{\rm f}
\kappa \in \GL_2(K) q_{\rm f} \GL_2(\hat{\Oo}_K)$. Then
\be \label{fourierrel} c_f(h,r) = \left| \frac{\det r_{\iy}} {\det 
q_{\iy}}\right|^{k }e^{-2 \pi \tr (r_{\iy}^* h r_{\iy} -
q_{\iy}^*
\g^* h \g q_{\iy})} |\det \g|^{-k} c_f(\g^* h \g, q). \ee 
\end{lemma}

\begin{proof} It follows from the proof of part (4) of
Proposition 18.3 of \cite{Shimura97}, that \be \label{f12} c_f(h,r) = 
|\det
r_{\iy}|^k e^{-2 \pi \tr (r_{\iy}^* h r_{\iy})} 
c_{p_{r_{\tuf}}}(h),\ee
where
$$f_{p_{r_{\tuf}}}(Z) = \sum_{h \in S} c_{p_{r_{\tuf}}}(h) e^{2 \pi i \tr 
hZ}.$$ As is easy
to see (cf. for example the Proof of Lemma 10.8 in \cite{Shimura97}),
$f_{p_{r_{\tuf}}} = f_{p_{q_{\tuf}}} |_k \bsmat \g^{-1} \\ & \g^* \esmat$. 
Hence
\be \label{f13} c_{p_{r_{\tuf}}}(h) = |\det \g|^{-k}
c_{p_{q_{\tuf}}}(\g^* h
\g).\ee The Lemma follows from combining (\ref{f12}) with (\ref{f13}).
\end{proof}

\begin{prop} \label{to check} Choose a base $\mB$ and let $f \in \mM_{k,-k/2}$.
If
there exist functions $c^*_{b,f}: \bfZ_{\geq 0} \rightarrow \bfC$, $b 
\in
\mB$, such
that for every $b \in \mB$ and every $h \in S(\bfQ)$, the Fourier
coefficient
$c_f(h,b)$ satisfies
condition (\ref{Maass condition}) with $c_{b,f}^*$ in place of 
$c_{b,f}$,
then $f$ is a $\mB$-Maass form and one has $c_{b,f}=c^*_{b,f}$ for every
$b \in \mB$. \end{prop}

\begin{proof} Fix $\mB$ and $f \in \mM_{k,-k/2}$. Suppose there exist
$c_{b,f}^*$
such that (\ref{Maass condition}) is satisfied for all pairs $(h,b)$. 
Let
$q = \g b x \kappa = (\g x, \g b \kappa)\in \GL_2(\bfC) \times
\GL_2(\AKf)$, where $\g \in \GL_2(K)$, $x \in \GL_2(\bfC)$ and $\kappa
\in \mK'$. Then by Lemma \ref{sameclass}, 
$$c_f(h,q) = |\det
q_{\iy}|^k e^{-2 \pi \tr (q_{\iy}^* h q_{\iy} -
\g^* h \g )} |\det \g|^{-k} c_f(\g^* h \g, b).$$ Since condition
(\ref{Maass
condition}) is satisfied for $(h,b)$, we know that
$$c_f(h,b)= e^{-2 \pi \tr h}
\sum_{\substack{d \in \bfZ_{+} \\ d \mid \epsilon(b^* hb)}}
d^{k-1} c^*_{b, f}\left( D_K d^{-2} \det h \hs \prod_{p} p^{\val_p(\det
b^*b)}\right).$$ Thus $$c_f(\g^*  h \g,b) = e^{-2 \pi \tr(\g^* h \g)}
\sum_{\substack{d \in \bfZ_{+} \\ d \mid \epsilon(b^*\g^* h\g b)}}
d^{k-1} c^*_{b, f}\left( D_K d^{-2} \det (\g^* h\g) \hs \prod_{p}
p^{\val_p(\det
b^*b)}\right).$$ So, \begin{multline} c_f(h,q) = |\det
q_{\iy}|^k
e^{-2\pi\tr(q_{\iy}^* h q_{\iy})} |\det \g|^{-k}
\times\\
\times \sum_{\substack{d \in \bfZ_{+} \\ d \mid \epsilon(b^* \g^* h
\g b)}}
d^{k-1} c^*_{b, f}\left( D_K d^{-2} \det h \det(\g^* \g) \hs \prod_{p}
p^{\val_p(\det   
b^* b)}\right).\end{multline}
The claim now follows since $\epsilon(b^* \g^* h
\g b) = \epsilon(q_{\rm f}^* h q_{\rm f})$
and $\det
(\g^* \g) \in \bfQ_+$, so $\det (\g^* \g) = \prod_p p^{\val_p(\det \g^*  
\g)}$.
\end{proof}

\begin{prop} \label{independence} If $\mB$ and $\mB'$ are two bases, 
then
the $\mB$-Maass space and the $\mB'$-Maass space coincide, i.e., the
notion of a Maass form is independent of the choice of the base.
\end{prop}

\begin{proof} Let $\mB$ and $\mB'$ be two bases. Write $q_{\rm f} =
\g_{b,q} b \kappa_{\mB}= \g_{b',q} b' \kappa_{\mB'}$ with $b \in
\mB$, $b'\in \mB'$, $\g_{b,q}, \g_{b',q} \in \GL_2(K)$ and
$\kappa_{\mB}, \kappa_{\mB'} \in \mK'$. Suppose $f$ is a $\mB$-Maass
form, i.e., there exist functions $c_{b,f}$ for $b \in \mB$, such
that for every $q$ and $h$, $$c_f(h,q) = t \hs |\det\g_{b,q}|^{-k}
\sum_{\substack{d \in \bfZ_{+} \\ d \mid \epsilon(q_{\rm f}^* h
q_{\rm f})}} d^{k-1} c_{b, f}(s),$$ where $t=|\det q_{\iy}|^k
e^{-2\pi\tr(q_{\iy}^* h q_{\iy})} $ and $s= D_K d^{-2} \det h \hs
\prod_{p} p^{\val_p(\det q_{\rm f}^*q_{\rm f})}$. Our goal is to show
that there exist functions $c_{b',f}$ for $b' \in \mB'$, such that
for every $q$ and $h$, \be \label{want1} c_f(h,q) = t \hs
|\det\g_{b',q}|^{-k} \sum_{\substack{d \in \bfZ_{+} \\ d \mid
\epsilon(q_{\rm f}^* h q_{\rm f})}} d^{k-1} c_{b', f}(s).\ee We have
\be \label{dets1} \det \g_{b,q} = \det \g_{b',q} \det (b' b^{-1})
\det (\kappa_{\mB}^{-1} \kappa_{\mB'}).\ee Since $\det(b'b^{-1})$
corresponds to a principal fractional ideal, say $(\alpha_{b,b'})$,
under the map $((\alpha_{b,b'})_{\fp}) \mapsto \prod_{\fp} 
\fp^{\val_{\fp}((\alpha_{b,b'})_{\fp})}$, using \cite{Bump97},
Theorem 3.3.1, we can write $\det(b' b^{-1}) = \alpha_{b,b'}
\kappa_{b,b'}\in \AKf^{\times}$ with $\kappa_{b,b'} \in
\hat{\Oo}_K^{\times}$. Then it follows from (\ref{dets1}) that   
$$\beta:= \kappa_{b,b'} \det(\kappa_{\mB}^{-1} \kappa_{\mB'}) \in
\hat{\Oo}_K^{\times} \cap K^{\times} = \OK^{\times}.$$ Hence
$\beta^k=1$. Thus $|\det \g_{b,q}|^{-k} = |\det
\g_{b',q}|^{-k} |\alpha_{b,b'}|^{-k}$. Note that
$|\alpha_{b,b'}|^{-k}$ is well defined and only depends on $b$ and
$b'$
(i.e., it is independent of $q$ and $h$). Set $c_{b', f}(n) =
|\alpha_{b,b'}|^{-k}c_{b, f}(n)$. Then it is clear that $c_{b',
f}$ satisfies (\ref{want1}). \end{proof}

\begin{definition} \label{maass2} From now on we will refer to 
$\mB$-Maass
forms simply
as
\textit{Maass forms}. Similarly we will talk about the \textit{Maass
space}
instead of $\mB$-Maass spaces. This is justified by Proposition
\ref{independence}. The Maass space will be denoted by
$\mM^{\tuM}_{k,-k/2}$. \end{definition}

We now recall the definition of Maass space introduced in 
\cite{Krieg91}. 
We will refer to it as the $U(\bfZ)$-Maass space. Assume $\mJ(K) \mid \frac{k}{2}$, so that by Proposition \ref{newmap} the space
$M_k^{\rm h}=M_{k,-k/2}^{\rm h}=M_{k,-k/2}^{\rm h}(U(\bfZ))$. We
say that $F
(Z)= \sum_{h \in \mS} c_{F}(h) e^{2 \pi i \tr (hZ)} \in M^{\rm h}_k$
is a
\textit{$U(\bfZ)$-Maass form} if there exists a function $\alpha_{F}:
\bfZ_{\geq 0}
\rightarrow \bfC$ such that for every $h \in \mS$, one has \be
\label{kriegcond} c_{F}(h) = \sum_{\substack{ d \in \bfZ_+\\ d \mid
\epsilon(h)}} d^{k-1} \alpha_{F}(D_K d^{-2} \det h).\ee The subspace of
$M_k^{\rm h}$ consisting of $U(\bfZ)$-Maass forms will be denoted by
$M_k^{\rm h,M}$.

\begin{prop} \label{oddclass} If $2 \nmid h_K$, then the Maass space
$\mM_{k,-k/2}^{\tuM}$ is isomorphic
(as a
$\bfC$-linear space) to $\# \mB$ copies of the $U(\bfZ)$-Maass space
$M^{\rm h,M}_k$.
\end{prop}

\begin{proof} Since the Maass space is independent of the choice of
a base $\mB$ by Proposition \ref{independence}, we may choose $\mB$
as in Corollary \ref{scalarcor} and $\# \mB = \# \mC=h_K$, with $\mC$ as
in
Proposition \ref{decomp657}. The map $\Phi_{\mB}: \mM_{k,-k/2}
\rightarrow \prod_{b \in
\mB} M^{\rm h}_k$ is an isomorphism. Let $f \in \mM_{k,-k/2}^{\tuM}$
and set $(f_b)_{b \in \mB}= \Phi_{\mB}(f)$. Set $\a_{f_b}:= c_{b,f}$.   
Then using
(\ref{f12}), and the fact that the matrices $b$
commute with
$h$ and
$b^*b=I_2$, we see that condition (\ref{Maass
condition}) for $c_f(h,b)$ translates
into condition (\ref{kriegcond}) for $c_{f_b}(h)$.
Hence $\Phi_{\mB}(\mM^{\tuM}_{k,-k/2}) \subset
\prod_{b \in \mB} M^{\rm h,M}_k$. On the other hand if $(f_b)_{b
\in \mB} \in \prod_{b \in \mB} M^{\rm h,M}_k$, set $c_{b,f}:=
\alpha_{f_b}$. Then conditions
(\ref{kriegcond}) for $c_{f_b}(h)$ translate into conditions
(\ref{Maass condition}) for $c_f(h,b)$. By Proposition
\ref{to check} this implies
that $f$ is a Maass form.
\end{proof}

\subsection{Invariance under Hecke action} \label{Invariance under Hecke
action}

It was proved in \cite{Krieg91} that the $U(\bfZ)$-Maass space is
invariant under the action of a certain Hecke operator $T_p$ associated
with a prime $p$
which is inert in $F$. On the other hand Gritsenko in \cite{Gritsenko90}
proved the invariance of the $U(\bfZ)$-Maass space under all the Hecke
operators when $K=\bfQ(i)$. In this section we
show that the Maass space
$\mM_{k,-k/2}^{\tuM}$ is
in
fact invariant under all the local Hecke algebras (for primes $p \nmid
D_K$) without imposing restrictions on the class number.

\begin{thm} \label{thmmain} Let $p \nmid D_K$ be a rational prime. The
Maass
space is invariant under the action of $\mH^+_p$, i.e., if $f \in
\mM^{\tuM}_{k,-k/2}$, and $g \in U(\bfQ_p)$, then $[\mK_pg\mK_p]f \in 
\mM^{\tuM}_{k,-k/2}$.
\end{thm}

\begin{rem} In fact (at least when $D_K$ is a prime) the Maass space is invariant under $\mH_p$, which is an easy consequence of the invariance under $\mH_p^{+}$. In this case it is also invariant under $\mH_{D_K}$, hence under \emph{all} local Hecke algebras. For both of these facts see Remark \ref{Missingops}. \end{rem}

\begin{proof} [Proof of Theorem \ref{thmmain}] We will only present the proof in the case when $p$
splits in $K$. For such a prime $p$ the invariance of
$\mM_{k,-k/2}^{\tuM}$ under the action of $\mH^+_p$ follows from  Propositions
\ref{invariance}, \ref{invariance2} below. If $p$ 
is
inert one
can proceed along the same lines,
however, it is the case when $p$ splits that is essentially new.
Indeed, if $p$ is inert, the elements of
$\mH^+_p$ respect the decomposition
(\ref{prod1}), hence the statement of the theorem reduces to an 
assertion
about the action of $\mH^+_p$ on $M_k^{\rm h,M}$. Then
the method used in \cite{Gritsenko90} can be adapted to prove the 
theorem.
See
also Theorem 7 in \cite{Krieg91} which proves the invariance of the
$U(\bfZ)$-Maass space for a certain family of Hecke operators in 
$\mH^+_p$.
\end{proof}

Let $p$ be a prime which splits in $K$. Write $p\OK=\fp \ov{\fp}$.
It suffices to prove the invariance of the Maass space under the
operators $T_{\fp, 1}$ and $T_{\fp, 2}$. 
In the rest of section \ref{Invariance under Hecke action}, $l$ will
denote a rational prime.

\subsection{Diagonalizing hermitian matrices mod $l^n$}

We begin by generalizing Proposition 7 of \cite{Krieg91} to split 
primes.

\begin{prop} \label{diagonal} Let $n$ be a positive integer, and
assume $l\nmid D_K$.
Let $h \in \mS_l$, $h \neq 0$. Then there exist
$a$, $d \in \bfZ_l$ with $l \nmid a$ and $u \in \SL_2(\Oo_{K,l})$ such
that
$$u^* h u \equiv l^{\epsilon_l(h)} \bmat a \\ &d \emat \pmod{l^n\mS_l}.$$
\end{prop}

In fact it is enough to prove the following lemma.

\begin{lemma} \label{diagonal2} Proposition \ref{diagonal} holds if 
$\mS_l$
is
replaced by the subgroup of hermitian matrices inside $M_2(\Oo_{K,l})$.
\end{lemma}

\begin{proof} For inert $l$, this is Proposition 7 of \cite{Krieg91}.  
So, assume that $l\OK=\lambda \ov{\lambda}$ with $\lambda \neq
\ov{\lambda}$. Without loss of generality we may assume that
$\epsilon_l(h)=0$. Let $(M, M^t)$ be the image of $h$ under the 
composite
\begin{multline} M_2(\Oo_{K,l}) \twoheadrightarrow M_2(\OK/l^n\OK)
\xrightarrow{\sim}
M_2(\OK/\lambda^n) \oplus M_2(\OK/\ov{\lambda}^n) \cong \\
\cong M_2(\bfZ/l^n\bfZ)
\oplus
M_2(\bfZ/l^n\bfZ).\end{multline} Since the canonical map
$$\SL_2(\Oo_{K,l})
\rightarrow
\SL_2(\OK/l^n\OK) \cong \SL_2(\bfZ/l^n\bfZ) \oplus \SL_2(\bfZ/l^n
\bfZ)$$ is
surjective
(\cite{Serre70}, p. 490), it is enough to find $A_1, A_2 \in  
\SL_2(\bfZ/ l^n \bfZ)$ such that \be \label{conj1} A_2^t M A_1 = \bsmat
\alpha \\
& \delta \esmat\ee
with $\alpha \not\equiv 0$ mod $p$.
The existence of such $A_1$ and $A_2$ is clear.
\end{proof}

\subsection{Invariance under $T_{\fp,1}$} \label{Invariance under T_p}

Let $g = T_{\fp,1} f$. Then for $q \in \GL_2(\AK)$ and $\sigma \in 
S(\AQ)$,
we
can write $$g \left( \bmat q & \sigma \hat{q} \\ & \hat{q} \emat \right) 
= \sum_{h \in S} c_g(h,q) e_{\AQ}(\tr (h \sigma)).$$ Define the 
following
matrices $$\alpha'_{\fp,a} = \left(\bmat p & -a \\ & 1 \emat, I_2
\right) \in
\GL_2(\bfQ_p) \times \GL_2(\bfQ_p)
\quad 
a=0,1, \dots,
p-1$$ and $$\alpha'_{\fp,p} = \left( \bmat 1 \\ & p \emat, I_2
\right),\in
\GL_2(\bfQ_p) \times \GL_2(\bfQ_p).$$
For $a=0,1,
\dots, p$, set $\alpha_{\fp,a}$ to be the images of $\alpha'_{\fp,a}$ in
$\GL_2(F_p)$. To simplify
notation in this section we drop the
subscript $\fp$ and simply write $\alpha_a$ for $\alpha_{\fp,a}$.

\begin{prop} \label{invariance} The Maass space is invariant under the
action of $T_{\fp,1}$, i.e., if $f\in \mM_{k,-k/2}^{\tuM}$, then $g\in
\mM^{\tuM}_{k,-k/2}$.
\end{prop}

\begin{proof}
Choose a base $\mB$ in such a way that for all $b\in \mB$ we have   
$b_l=I_2$ if $l\mid D_K$. For $b\in \mB$ write $b_{\bfQ}$ for $\prod_l
l^{\val_l(\det b^*b)}$. By Propositions \ref{to
check} and \ref{independence}
it is enough to show that there exist functions $c_{b, g}: \bfZ_+
\rightarrow \bfC$ ($b \in \mB$) such that \be \label{condrep} c_g(h,b) =
e^{-2\pi\tr  h}  
\sum_{\substack{ d \in \bfZ_+ \\ d \mid \epsilon(b^*hb)}} d^{k-1} 
c_{b,g} 
\left( D_K d^{-2} b_{\bfQ} \det h \right).\ee For $b \in \mB$, set $b' = 
b
\alpha_p$. Note that all of the matrices: $b \a_a$,
$b \hat{\a}_a$, ($a=0,1, \dots, p$) belong to the same
class $b'$. Denote any of these matrices by $q$. Then $q =
\gamma_{b', q} b' \kappa_q \in \GL_2(K) b' \mK'$ and it is easy to see 
that
$$\det \gamma_{b', q}^k = \begin{cases} 1 & q = b \a_a, \hf a = 0, 1,
\dots, p \\ p^{-k} & q = b \hat{\a}_a \hf a = 0, 1,
\dots, p ,
\end{cases}$$ and $$\prod_l \val_l(\det q^* q ) =b_{\bfQ}
\times \begin{cases}  p & q = b
\a_a \hf a = 0, 1,
\dots, p \\ p^{-1} & q = b \hat{\a}_a \hf a = 0, 1,
\dots, p .
\end{cases}$$ Set $h_0:= b^*h b$ and write $h_0 = \epsilon(h_0) h'$
with $\epsilon(h') = 1$.
Set $D= D_K \det h$ and $D' = D_K \prod_l l^{\val_l(\det h')}$.
Using Proposition \ref{f3} and the fact that $f$ is a Maass form, we obtain
\begin{multline} \label{Fourier2} c_g(h,b) = e^{-2 \pi \tr h} \times
\left( p^2
\sum_{a=0}^p \sum_{\substack {d \in \bfZ_+ \\ d \mid \epsilon(\a_a^* b^* 
h
b
\a_a)}} d^{k-1} c_{b',f}(D d^{-2} b_{\bfQ} p) + \right. \\
+ \left. p^k \sum_{a=0}^p \sum_{\substack {d \in \bfZ_+ \\ d \mid
\epsilon(\hat{\a}_a^*
b^* h b
\hat{\a}_a)}} d^{k-1} c_{b',f}(D d^{-2} b_{\bfQ} p^{-1})\right).
\end{multline}
Using Proposition \ref{diagonal}, one can relate $\epsilon(\a_a^* h_0
\a_a)$
and $\epsilon(\hat{\a}_a^* h_0 \hat{\a}_a)$ to $\epsilon(h_0)$ for 
$a=0,1,\dots, p$, 
and then
(\ref{Fourier2}) becomes \begin{multline} \label{mess} c_g(h,b)  =
e^{-2 \pi
\tr h} p^2 \sum_{0} A^{(1)}_d +
e^{-2 \pi \tr h} \times \\
\times \begin{cases}
 p^3 \sum_0 A_d^{(1)}+p^k(p+1) \sum_{-1} A^{(-1)}_d
& p \nmid D',
\\
p^2(p-1) \sum_0 A_d^{(1)} + p^2 \sum_1 A_d^{(1)} + p^k \sum_0  
A^{(-1)}_d
+
p^{k+1}
\sum_{-1} 
 A^{(-1)}_d
& p \mid
D',
\end{cases} \end{multline}
where $\sum_n A_d^{(m)}:= \sum_{d \mid
p^n \epsilon(h_0)}A_d^{(m)}$, $A^{(m)}_d = d^{k-1} c_{b',f}(D d^{-2}
b_{\bfQ} p^m)$ and if there is no $d$ dividing $p^n \epsilon(h_0)$, we 
set
$\sum_n=0$.

For $D$ in the image of the map $h
\mapsto D_K \epsilon(h)^{-2} b_{\bfQ}\det h$ and $b \in \mB$ we make the  
following
definition
\be \label{maass1} c_{b,g}(D) =  p^2(p+1 ) c_{b', f} (Dp) + p^k (p+1)
c_{b',f} (Dp^{-1}) ,\ee where we assume that
$c_{b',f}(n)=0$ when $n \not\in \bfZ_+$. If $D$ is not in the image of
that map, we set $c_{b,g}(D)=0$. Note that we clearly have
$$c_g(h,b) = e^{-2 \pi \tr h} c_{b,g} (D_K b_{\bfQ} \det h)$$ for every
$h$ with
$\epsilon(b^*hb)=1$. Thus to check if $g$ lies in the Maass space we 
just
need
to check that (\ref{condrep}) holds with $c_{b,g}$ defined by
(\ref{maass1}).
This is an easy calculation using (\ref{mess}). \end{proof}

\subsection{Invariance under $T_{\fp,2}$}

This is completely analogous to the proof for $T_{\fp,1}$, hence we only
include the relevant formulas for the reader's convenience.

Let $g = T_{\fp,2} f$. Define matrices:
$$\beta'_p = \left( \bmat p \\ & p \emat, I_2\right) \in
\GL_2(\bfQ_p)\times
\GL_2(\bfQ_p),$$
$$\gamma'_a = \left( \bmat 1 \\ a & p \emat, I_2\right) \in 
\GL_2(\bfQ_p)
\times \GL_2(\bfQ_p),\quad a=0,1,\dots, p-1,$$
$$\gamma'_p = \left( \bmat p \\ & 1 \emat, I_2 \right) \in \GL_2(\bfQ_p)
\times \GL_2(\bfQ_p),$$ and set $\beta_p$ to be the image of $\beta'_p$ in
$\GL_2(K_p)$, and $\gamma_a$ to be the image of $\gamma'_a$ in $\GL_2(K_p)$ ($a=0,1, 
\dots,
p$).

\begin{prop} \label{invariance2} The Maass space is invariant under the
action of $T_{\fp,2}$. \end{prop}

\begin{proof} This is similar to the proof of Proposition
\ref{invariance}. Let $b$, $D$, $D'$, $h$, $h'$ and $h_0$ be as in that  
proof. Then
for all
$a$, $c$, we
see that $b\beta_p$, $b \hat{\beta}_p$, $b \gamma_a \hat{\gamma}_c$ all 
lie in the same class $b'= b \beta_p$.
One has $$\det
\gamma_{b', q}^k = \begin{cases} 1 & q = b \beta_p \\ p^{-k} & q = b
\gamma_a \hat{\gamma}_c, \hf a,c \in \{ 0,1, \dots, p \} \\ p^{-2k} & q 
=
b \hat{\beta}_p, \end{cases}$$ and
$$\prod_l \val_l(\det q^* q) = b_{\bfQ}
\times \begin{cases} p^2
& q = b \beta_p \\ 1 & q = b
\gamma_a \hat{\gamma}_c, \hf a,c \in \{ 0,1, \dots, p \} \\ p^{-2} & q =
b \hat{\beta}_p. \end{cases}$$

Using Proposition \ref{diagonal} as in the proof of Proposition
\ref{invariance} we obtain
\be\begin{split} c_g(h,b) &= e^{-2 \pi \tr h} \left(p^4 \sum_1A^{(2)}_d 
+
p^{2k}
\sum_{-1} A^{(-2)}_d +\right. \\
&\left. +p^{k+1}(p+1) \sum_0 A^{(0)}_d + p^{k+3}
\sum_{-1}
A^{(0)}_d\right)+\\
& + e^{-2 \pi \tr h} p^{k+1} \begin{cases} p \sum_{-1} A^{(0)}_d & p
\nmid D'\\
p \sum_0 A^{(0)}_d & p \mid D', \hs p^2 \nmid D'\\
\sum_1 A^{(0)} + (p-1) \sum_0 A^{(0)}& p^2 \mid D',\end{cases}
\end{split}\ee
where if there is no $d$ dividing $p^n \epsilon(h_0)$, we set $\sum_n =
0$.
For $D$ in the image of the map $h \mapsto D_K \epsilon(h)^{-2}
b_{\bfQ} \det h$,
we make the following definition: \begin{equation}\begin{split} \label{for1}
c_{b,g}(D) &=
p^4c_{b',f}(Dp^2) + (p^{k+3}+p^{k+2} + p^{k+1}) c_{b',f}(D) + \\
&+ \begin{cases} 0 & p \nmid D \\
p^{k+2} c_{b'f}(D) & p \mid D, \hs p^2 \nmid D \\
p^{k+2} c_{b',f}(D) + p^{2k} c_{b',f}(Dp^{-2}) & p^2 \mid D. \end{cases}
\end{split}\end{equation} We now check as in the proof of Proposition
\ref{invariance} that $g$ is a Maass form.
\end{proof}

\subsection{Descent} \label{Descent}

In this section we assume that $h_K$ is odd and choose a base $\mB$ as
in Corollary
\ref{scalarcor}. One has $\# \mB = \# \mC=h_K$ with $\mC=\{p_b \mid b\in \mB\}$. We also assume that $I_2 \in \mB$. Let $\chi_K$ be the
quadratic Dirichlet
character associated with the extension $K/\bfQ$. For a positive integer 
$n$,  
set $$a_K(n) =
\# \{ \alpha \in (i D_K^{-1/2} \OK)/\OK \mid D_K N_{K/\bfQ}(\alpha)
\equiv -n \pmod{D_K}\}.$$

\begin{thm} \label{desc4} There exists a $\bfC$-linear injection of 
vector
spaces
$$\Desc_{\mB}: \mM_{k,-k/2}^{\tuM} \hookrightarrow \prod_{b \in \mB} M_{k-1}
(D_K,
\chi_K),$$
such that $\Desc_{\mB}(f) = (\phi_b)_{b \in \mB}$ with $$a_{\phi_b}(n) = i   
\frac{a_{K}(n)}{\sqrt{D_K}} c_{b,f}(n),$$ where $a_{\phi_b}(n)$ is the
$n$-th Fourier coefficient of $\phi_b$. The map $\Desc_{\mB}$  
depends on the choice of $\mB$. \end{thm}

\begin{proof} This follows immediately from \cite{Krieg91}, Theorem 6 
and
formula (4) using our assumption on $\mB$ and (\ref{f12}). \end{proof}

\begin{rem} \label{kr} Krieg in \cite{Krieg91} explicitly describes
the image of the
descent map he defines and denotes it by $G_{k-1}(D_K, \chi_K)^*$. The
image of $\Desc_{\mB}$ is exactly $$\prod_{b \in \mB} G_{k-1}(D_K, \chi_K)^*
\subset \prod_{b \in \mB} M_{k-1}(D_K, \chi_K).$$
\end{rem}

Identify $\mB$ with $\Cl_K$ via the map that sends $b$ to $c_K(\det
b)$, where $c_K: \AK^{\times} \twoheadrightarrow \Cl_K$ is the
canonical projection.
Then $\Cl_K$ acts faithfully on $\mB$ by multiplication and defines
an embedding $s: \Cl_K \hookrightarrow S'_{\mB}$, where $S'_{\mB}$ is
the group of permutations of the elements of $\mB$. Write $S_{\mB}$
for the image of $s$. For a split $p$ with $p\OK=\fp
\ov{\fp}$, let $\alpha_{\fp,p}$ be as in section \ref{Invariance 
under T_p}. Note that $\alpha_{\ov{\fp},p}=\alpha_{\fp,p}^*$. Write
$\sigma_{\fp,n}:= s \circ c_K(\det
\alpha_{\fp,p}^n)$. If $A=(a_b)_{b \in \mB}$ is an 
ordered tuple
indexed by elements of $\mB$, and $\sigma \in S_{\mB}$, define
$\sigma(A)$ to be the $\mB$-tuple whose $b$-component is
$a_{\sigma(b)}$. Write $\alpha_{\fp,p}^n= \gamma_{\fp,n}
\sigma_{\fp,n}(I_2) \kappa_{\fp,n}$ for $\gamma_{\fp,n} \in \GL_2(K)$
and $\kappa_{\fp,n} \in \mK'$.

Write $T_p$ for the classical Hecke operator acting on elliptic modular forms, i.e., for $\phi(z) = \sum_{n=1}^{\iy} a(n) e^{2 \pi i nz} \in M_m(N,\psi)$ define $\phi' := T_p \phi$ by $\phi'(z) =
\sum_{n=1}^{\iy} a'(n)
e^{2 \pi i nz}$ with $a'(n) = a(np) + \psi(p)p^{m-1}
a(n/p)$. Here $a(n)=0$ if $n \not\in \bfZ_{\geq 0}$. The action of $T_p$ on $M_{k-1}(D_K, \chi_K)$ extends component-wise to an action on $\prod_{b \in \mB} M_{k-1}(D_K, \chi_K)$ which we will again denote by $T_p$. 
Write
$\bfT_p$
for the
$\bfC$-subalgebra of endomorphisms of $\prod_{b \in \mB} M_{k-1}(D_K,  
\chi_K)$ generated by $T_p$ if $p$ is split (resp. by $T_p^2$ if
$p$ is inert) and the group $S_{\mB}$.

\begin{thm} \label{heckedesc} Let $p$ be a rational prime which splits 
in
$K$. There
exists a $\bfC$-algebra map $$\Desc_{\mB,p}: \mH^+_p \rightarrow
\bfT_p,$$ such
that for
every $T \in \mH^+_p$ the
following diagram
$$\xymatrix@C7em{\mM^{\rm M}_{k,-k/2} \ar[r]^{T} \ar[d]^{\Desc_{\mB}} & \mM^{\rm M}_{k,-k/2}
\ar[d]_{\Desc_{\mB}} \\ \prod_{b \in \mB} M_{k-1} (D_K, \chi_K)
\ar[r]^{\Desc_{\mB,p}(T)} & \prod_{b \in \mB} M_{k-1} (D_K,
\chi_K)}$$ commutes. Moreover, one has \be \label{for11} \begin{split}
\Desc_{\mB,p}(T_{\fp,1}) & = p^{2-k/2}  (p+1) T_p
\circ
\sigma_{\fp,1}, \\
\Desc_{\mB,p}(T_{\fp,2})& =p^{4-k} ( T^2_p
+p^{k-1}+p^{k-3})\circ
\sigma_{\fp,2}.\end{split} \ee
\end{thm}

\begin{proof} This follows from Theorem \ref{desc4},
and formulas (\ref{maass1}) and (\ref{for1}).  
Just for illustration, we include the argument in the case of $T_{\fp,1}$.
Let $f \in \mM_{k,-k/2}^{\tuM}$ and set $g=T_{\fp,1}f$. Fix $b \in \mB$ and write
$b_1$ for $\sigma_{\fp,1}(b)$. One has $\alpha_{\fp,p}=\gamma b'
\kappa$, where $\gamma = \gamma_{\fp,1} \in \GL_2(K)$, $\kappa =
\kappa_{\fp,1} \in \mK'$, $b'=\sigma_{\fp,1}(I_2) \in \mB$ with $\gamma,
\kappa$ diagonal. In fact one can choose $\gamma$ to be of the
form $\bsmat 1 \\ &* \esmat$. Write $bb' =
\gamma_{b,b'} b_1 \kappa_{b,b'} \in \GL_2(K) b_1 \mK'$, where
$\gamma_{b,b'}$, $\kappa_{b,b'}$ are
scalars. Then $$b \alpha_{\fp,p} = \gamma b b' \kappa = \gamma
\gamma_{b,b'} b_1 \kappa \kappa_{b,b'}.$$
Identify $\mM^{\tuM}_{k,-k/2}$ with $\prod_{c \in \mB}
M_k^{\rm h, M}$ via
$f' \mapsto (f'_c)_{c \in \mB}$. For $f'_c \in
M_k^{\rm h, M}$, $h
\in \mS$, denote
by $c_{f'_c}(h)$ the $h$-Fourier coefficient of $f'_c$. We will study 
the
action of $T_{\fp,1}$ on $c_{f_{b_1}}(h)$. Since $f$ is a Maass form it is
enough to consider $h$ of the form $\bsmat 1&* \\ *&*\esmat$. Fix such 
an
$h$. Set $D=D_K \det h$. Then by (\ref{f12}) and (\ref{maass1}),
$$c_{g_b}(h) = e^{2 \pi \tr h} c_g(h,b) = p^2 (p+1)
(c_{b\alpha_{\fp,p},f}(Dp) +
p^{k-2} c_{b\alpha_{\fp,p},f}(Dp^{-1})).$$ By (\ref{f12}) and
(\ref{Maass condition}),
we have \be \begin{split} c_{f_{b_1}}\left(h \bsmat 1 \\ & p \esmat
\right) &
= e^{2 \pi \tr
h \bsmat 1 \\ & p \esmat} c_f \left( h \bsmat 1 \\ & p \esmat, b_1 
\right)
=
c_{b_1,f}(Dp) =
|\det \g\gamma_{b,b'}|^k c_{b\alpha_{\fp,p},f}(Dp)\\
 c_{f_{b_1}}\left(h \bsmat 1 \\ & 1/p \esmat \right) & = e^{2 \pi \tr
h \bsmat 1 \\ & 1/p \esmat} c_f \left( h \bsmat 1 \\ & 1/p \esmat,
b_1
\right) = c_{b_1,f}(Dp^{-1}) =\\
&= |\det \g\gamma_{b,b'}|^k
c_{b\alpha_{\fp,p},f}(Dp^{-1}),\end{split}
\ee where
we have used the 
fact that $\epsilon(h) = \epsilon\left(h \bsmat 1 \\ & p \esmat\right) =
1= \epsilon\left(h \bsmat 1 \\ & 1/p \esmat\right)$ for $h$ as above
(the last equality holding for $h$ such that $h \bsmat 1 \\ & 1/p
\esmat \in \mS$). This gives us $$c_{g_b}(h) = p^2 (p+1) |\det
\g\gamma_{b,b'}|^{-k}
(c_{b_1,f}(Dp) + p^{k-2} c_{b_1,f}(Dp^{-1})).$$ Note that $|\det \gamma|^k=p^{k/2}$ since if we write $x$ for the adele whose $\fp$th and $\ov{\fp}$th component is $p$ and all the other components are 1, then we have $$x^{k/2}=|\det \alpha_{\fp,p}|^k = |\det \gamma|^k |\det b'|^k |\det \kappa|^k.$$  
The claim now follows 
from
Theorem \ref{desc4} and the fact that $|\det \gamma_{b,b'}|^k=1$,
which follows from $bb^*=b'(b')^*=I_2$.
\end{proof}

For completeness we also include the analogue of Theorem \ref{heckedesc}
for an inert
$p$. It can be proved in the same way or can be
deduced
from the results of section 3 of \cite{Gritsenko90}.

\begin{thm} \label{heckedesc2} Let $p$ be a rational prime which is 
inert
in
$K$. There
exists a $\bfC$-algebra map $$\Desc_{\mB,p}: \mH^+_p \rightarrow \bfT_p,$$
such
that for
every $T \in \mH_p$ the
following diagram
$$\xymatrix@C7em{\mM^{\rm M}_{k,-k/2} \ar[r]^{T} \ar[d]^{\Desc_{\mB}} & \mM^{\rm M}_{k,-k/2}
\ar[d]_{\Desc_{\mB}} \\ \prod_{b \in \mB} M_{k-1} (D_K, \chi_K)
\ar[r]^{\Desc_{\mB,p}(T)} & \prod_{b \in \mB} M_{k-1} (D_K,  
\chi_K)}$$ commutes.  Moreover, one has \be \label{for12} \begin{split}
\Desc_{\mB,p}(T_{p,0}) & = p^{-k+4}(p^2+1) T_p^2+ p^4 + p^3+ p
- 1,\\   
\Desc_{\mB,p}(U_p) & = p^8(T_p^4 + (p+3)p^{k-2} T_p^2 +
p^{2k-4}(p^2+p+1)).\end{split} \ee \end{thm}

\subsection{$L$-functions} \label{L-functions}

In this section we study eigenforms in $\mM_{k,-k/2}^{\tuM}$ and
give a formula for the standard $L$-function of such an eigenform.
   
From now on assume that $D_K$ is prime. It is well-known that this 
implies
that
$h_K$ is odd, hence we can (and will) choose $\mB$ be as in Corollary
\ref{scalarcor}. One has $\# \mB = \# \mC = h_K$ with $\mC=\{p_b \mid b \in \mB\}$. On the other hand, for such a $D_K$ the 
space
$S_{k-1}(D_K, \chi_K)$ of cusp forms inside $M_{k-1}(D_K, \chi_K)$ has a
basis $\mN$ consisting of newforms.
In particular, if $\phi \in S_{k-1}(D_K, \chi_K)$ is an
eigenform
for almost all $T_p$, it is so for all $T_p$. For $\phi
=\sum_{n=1}^{\iy} a_{\phi}(n) e^{2\pi inz} \in S_{k-1}(D_K, \chi_K)$,
set $\phi^{\rho}(z):= \sum_{n=1}^{\iy}
\ov{a_{\phi}(n)}
e^{2\pi inz}$. Let $\mN'\subset \mN$
denote the set formed by choosing one element from each pair $(\phi,
\phi^{\rho})$ such that $\phi \in \mN$ and $\phi \neq \phi^{\rho}$. Then
the set $\{\phi - \phi^{\rho} \mid \phi \in \mN'\}$ is a basis of
$S^*_{k-1}(D_K,
\chi_K):=
G_{k-1}^*(D_K, \chi_K) \cap S_{k-1}(D_K, \chi_K)$ (cf. \cite{Krieg91}, Remark (b) on p. 670).  
 Let $\chi
\in \Hom(\Cl_K, \bfC^{\times})$. Recall (cf. (\ref{de1})) that we have the decomposition  \be \label{deco20} \mM_{k,-k/2} = \bigoplus_{\chi\in
\Hom(\Cl_K,
\bfC^{\times})}\mM_{k,-k/2}^{\chi}.\ee
Let $\mS_{k,-k/2}$ denote the subspace of cusp forms in $\mM_{k,-k/2}$ and write   
$\mS_{k,-k/2}^{\tuM}$ for $\mS_{k,-k/2} \cap \mM_{k,-k/2}^{\tuM}$. It is clear
that a decomposition analogous to (\ref{deco20}) holds for $\mS_{k,-k/2}$,
with $\mS_{k,-k/2}^{\chi}$ having the obvious
meaning. Write $\mS_{k,-k/2}^{\tuM, \chi}$ for $\mS_{k,-k/2}^{\tuM} \cap 
\mM_{k,-k/2}^{\chi}$.

Let $\bfT_p$ be as in section \ref{Descent} and write $\bfT'_{\bfC}$
for $\bfC$-subalgebra
of the ring of
endomorphisms of $S_{k-1}(D_K,\chi_K)$ generated by $T_p$ (for $p$
split) and $T_p^2$ (for $p$ inert).
The algebra $\bfT'_{\bfC}$ acts on
$S^*_{k-1}(D_K,\chi_K)$.
For $\phi \in \mN'$, the element $\phi -
\phi^{\rho} \in S_{k-1}^*(D_K, \chi_K)$ is a non-zero eigenform for
$\bfT'_{\bfC}$ and for every $\chi \in \Hom(\Cl_K,
\bfC^{\times})$, the $\mB$-tuple $$\phi_{\chi}:=(\chi(c_K(\det  b))
(\phi-\phi^{\rho}))_{b \in
\mB}$$ is an eigenform for $\bfT_p$ for every $p \neq D_K$. Below we will write $\chi(b)$ instead of $\chi(c_K( \det b))$. Hence
$f_{\phi, \chi}: =
\Desc_{\mB}^{-1}(\phi_{\chi})$ lies in $\mS_{k,-k/2}^{\tuM,\chi}$ and is an
eigenform for $\mH_p$ for every $p \neq D_K$.

\begin{rem} \label{Missingops} Since $2 \nmid h_K$, the prime $\fp$ such that $D_K\OK =
\fp^2$ is principal. Hence one can use the calculations in
\cite{Gritsenko90} to conclude that the Maass space is also invariant
under the action of $\mH_{D_K}$ and hence that $f_{\phi, \chi}$ is an 
eigenform
for $\mH_p$ for all $p$. Also note that for a split prime $p$ and a prime $\fp$ of $K$ lying over $p$ the operator $\Delta_{\fp}$ acts on $f_{\phi, \chi}$ via multiplication by a scalar, so we conclude that the Maass space is in fact invariant under $\mH_p$. \end{rem}

We have proved the following proposition.

\begin{prop} \label{heckedesc3} Let $\Desc_{\mB}$ be the map from Theorem \ref{desc4}. The
composite
$$\mM^{\rm M}_{k,-k/2} \xrightarrow{\Desc_{\mB}} \prod_{b \in \mB} M_{k-1}(D_K, \chi_K)
\xrightarrow{\textup{pr}_{I_2}}  M_{k-1}(D_K, \chi_K)$$ induces
$\bfC$-linear isomorphisms $\mS_{k, -k/2}^{\tuM,\chi} \cong S^*_{k-1}(D_K,
\chi_K)$ for every $\chi \in \Hom(\Cl_K, \bfC^{\times})$. The inverse of 
such an isomorphism (for a fixed $\chi$) is induced by
$$\phi-\phi^{\rho} \mapsto
f_{\phi, \chi}:= \Desc_{\mB}^{-1}(\phi_{\chi}) $$ for any $\phi \in \mN'$.
Moreover,
these isomorphisms are
Hecke-equivariant with respect to the Hecke algebra maps
\textup{(\ref{for11})} and
\textup{(\ref{for12})} except
that in \textup{(\ref{for11})} we replace composition with
$\sigma_{\fp,n}$ by
multiplication by
$\chi(\alpha_{\fp}^n)$. \end{prop}

It follows that if $\phi$ runs over $\mN'$ and
$F_{\phi}$ denotes the Maass
lift of $\phi$ in the sense of Krieg \cite{Krieg91}, i.e.,
$F_{\phi}=\Desc_K^{-1}(\phi-\phi^{\rho})\in M^{\rm h,M}_k$, where
$\Desc_K: M_k^{\rm h,M} \xrightarrow{\sim} S^*_{k-1}(D_K, \chi_K)$ is 
the
(non-Hecke-equivariant) isomorphism constructed in Theorem
on p.676 of \cite{Krieg91}, then the set of
$\mB$-tuples $\{(\chi(b)F_{\phi})_{b \in \mB}\}$ is basis of
eigenforms of $\mS^{\tuM,\chi}_{k, -k/2}$ after we identify $\mS^{\tuM,\chi}_{k, -k/2}$
with its image inside $\prod_{b \in \mB} M_k^{\rm h,M}$.

Let $\phi \in S_{k-1}(D_K, \chi_K)$ be a newform and write $f_{\phi, \chi} \in \mS^{{\rm M}, \chi}_{k, -k/2}$ for its Maass lift as above.  For a unitary Hecke character 
$\psi:
K^{\times} \setminus \AK^{\times} \rightarrow \bfC^{\times}$ and a Hecke eigenform $f \in \mS^{\chi}_{k, -k/2}$ denote by 
$L_{\tust}(f,s,\psi)=Z(s,f,\psi)$ the standard $L$-function of $f$ twisted by $\psi$
as defined in \cite{Shimura00}, section 20.6 with the Euler factor at
$D_K$ removed. 
Moreover, for the newform
 $\phi \in S_{k-1}(D_K, \chi_K)$, and a prime $\fp$ of $\OK$ of
characteristic $p\neq D_K$, set $\alpha_{\fp, j}:= \alpha_{p,j}^d$, 
where
$\alpha_{p,j}$, $j=1,2$ are the $p$-Satake parameters of $\phi$ and
$d:=[\OK/\fp:\bfF_p]$. For $s \in \bfC$ with $\textup{Re}(s)$ 
sufficiently
large, define $$L(\BC(\phi), s ,\psi):= \prod_{\fp \nmid
D_K} \prod_{j=1}^2 (1-\psi^*(\fp)\alpha_{\fp,j}(N\fp)^{-s})^{-1},$$ 
where
$\psi^*$ denotes the ideal character associated to $\psi$ and $N\fp$
denotes the norm of $\fp$. It is well-known that $L(\BC(\phi), s, \psi)$
can be continued to the entire $\bfC$-plane.

\begin{prop} \label{product134} Let $\psi$ be a unitary Hecke character 
of
$K$. The following
identity holds:
$$L_{\tust}(f_{\phi, \chi}, s, \psi) = L(\BC(\phi), s-2+k/2, 
\chi\psi)L(\BC(\phi),
s-3+k/2,
\chi\psi).$$ \end{prop}

\begin{proof} This is an easy calculation involving Satake parameters.
\end{proof}

\begin{rem} In \cite{Ikeda08} Ikeda has studied
liftings from the space of elliptic cusp forms into the space
of hermitian cusp forms defined on the group $U_n$ with no assumptions on the class number of $K$. In
particular he constructs a lifting
$S_{k-1}(D_K, \chi_K) \rightarrow \mS_{k,-k/2}$, which agrees with the map 
$\phi
\mapsto f_{\phi, \mathbf{1}}$, where $\mathbf{1}$ denotes 
the
trivial character. The method used in \cite{Ikeda08} 
is
different
from ours. \end{rem}

\subsection{The Petersson norm of a Maass lift} \label{The Petersson 
norm
of a Maass lift}

Let $\phi \in S_{k-1}(D_K, \chi_K)$ be a newform such that $\phi \neq \phi^{\rho}$. Let $\chi: \Cl_K \to \bfC^{\times}$ be a character. Write $f_{\phi, \chi} \in \mS_{k, -k/2}^{{\rm M}, \chi}$ for the Maass lift of $\phi$. 

\begin{thm} \label{FF11} Let $\ell>3$ be an odd prime and assume $\ell \nmid h_K D_K$. Then
one has
$$\left< f_{\phi, \chi},f_{\phi, \chi}\right> =
 C  \pi^{-k-2}\cdot
\left< \phi, \phi \right>
L(\Symm
\phi, k),$$
where $C_{\chi}\in \bfQ^{\times}$ is
a $\ell$-adic unit and
\begin{multline} L(\Symm
\phi, s)^{-1}:= \prod_{p \nmid D_K} (1-\alpha_p^2 p^{-s})
(1-\alpha_p \beta_p p^{-s})  (1-\beta_p^2 p^{-s}) \times \\
\times \prod_{p \mid D_K} (1-a(p)^2 p^{-s}) (1-\ov{a(p)}^2
p^{-s}).
\end{multline} Here $\a_p$, $\b_p$ are the classical $p$-Satake 
parameters
of $\phi$ and $a(p)$ is the $p$-th Fourier coefficient of $\phi$.
\end{thm}

\begin{proof} This is Proposition 17.4 in \cite{Ikeda08}, which is essentially due to Sugano - see the references cited in [loc.cit.]. \end{proof}

\section{Completed Hecke algebras}\label{Completed Hecke algebras}
Let $\ell$, as before, be a fixed prime such that $\ell \nmid 2  D_K$. Suppose $D_K$ is prime. Then the space $S_{k-1}(D_K, \chi_K)$ has a canonical basis $\mN$ consisting of newforms. 
The goal of this section is to construct a Hecke operator $T^{\rm h}$ acting on the space $\mS_{k, -k/2}^{\chi}$ such that $T^{\rm h}$ preserves the $\ell$-integrality of the Fourier coefficients of the hermitian modular forms in  $\mS_{k, -k/2}^{\chi}$ and such that $T^{\rm h}f_{\phi,\chi}=\eta f_{\phi,\chi} $ for a Maass lift $f_{\phi,\chi}$ of an elliptic modular form $\phi$ and $T^{\rm h}f=0$ for all the $f \in \mS_{k,-k/2}^{{\rm M}, \chi}$ with $\left< f, f_{\phi,\chi}\right>=0$. Here $\eta$ is a generator of the Hida's congruence ideal.

\subsection{Elliptic Hecke algebras}

Let $\bfT_{\bfZ}$ be the $\bfZ$-subalgebra of $\End_{\bfC}(S_{k-1}(D_K, \chi_K))$ generated by the (standard) Hecke operators $T_n$, $n=1,2,3, \dots$ (for the action of $T_p$ on the Fourier coefficients see section \ref{Descent}).

\begin{definition} \label{varioushecke} For every $\bfZ$-algebra $A$ we set \begin{itemize} 
\item[(i)] $\bfT_A:= \bfT_{\bfZ} \otimes_{\bfZ} A$;
\item[(ii)] $\bfT'_A$ to be the $A$-subalgebra of
$\bfT_A$ generated by the set $$\Sigma':=\{T_p\}_{p \hs \textup{split in $K$}}
\cup
\{T_{p^2}\}_{p \hs
\textup{inert in $K$}}.;$$
\item[(iii)] $\tilde{\bfT}'_A$ to be the $A$-subalgebra of $\bfT'_A$ generated by $\tilde{\Sigma}'$, where $\tilde{\Sigma}'= \Sigma' \setminus \{T_{\ell}\}$. \end{itemize} \end{definition}

Suppose $\phi=\sum_{n=1}^{\iy}a_{\phi}(n) q^n \in \mN$. For $T \in
\bfT_{\bfC}$, set $\lambda_{\phi, \bfC}(T)$ to denote the eigenvalue of $T$
corresponding to $\phi$. It
is a well-known fact that $\lambda_{\phi, \bfC}(T_n)= a_{\phi}(n)$ for all $\phi \in
\mathcal{N}$
and that the set $\{a_{\phi}(n)\}_{n \in \bfZ_{>0}}$ is contained in the ring
of integers of a finite extension
$L_{\phi}$ of $\bfQ$. Let $E$ be a finite extension of
$\bfQ_{\ell}$ containing the fields $L_{\phi}$ for all $\phi \in \mathcal{N}$.
Denote by $\mathcal{O}$ the valuation ring of $E$ and by $\varpi$ a
uniformizer of $\Oo$. Then $\{a_{\phi}(n)\}_{\phi \in \mN, n \in \bfZ_{>0}}
\subset \Oo$. Moreover, one has \be
\label{hecke532} \bfT_E = \prod_{\phi \in \mathcal{N}} E\ee
and
\be \label{hecke533} \bfT_{\mathcal{O}} =
\prod_{\fm}
\bfT_{\mathcal{O},\fm},
\ee where $\bfT_{\mathcal{O},\fm}$ denotes the localization
of
$\bfT_{\mathcal{O}}$ at $\fm$
and the product runs over all maximal ideals of $\bfT_{\mathcal{O}}$. Analogous decompositions hold for $\bfT'_{\Oo}$ and $\tilde{\bfT}'_{\Oo}$.
Every $\phi \in \mN$ gives rise to an $\Oo$-algebra homomorphism
$\bfT_{\Oo} \rightarrow \Oo$ assigning to $T$ the eigenvalue of $T$
corresponding to $\phi$. We denote this homomorphism by $\lambda_{\phi}$ and its
reduction mod $\varpi$ by $\ov{\lambda}_{\phi}$. If $\fm = \ker \ov{\lambda}_{\phi}$, we write
$\fm_{\phi}$ for $\fm$. For simplicity in this section we will drop the subscript $\Oo$ from notation, so for example, we will simply write $\bfT$ instead of $\bfT_{\Oo}$. 

Fix a maximal ideal $\fm$ of $\bfT$. Write $\fm'$ (resp. $\tilde{\fm}'$) for the maximal ideal of $\bfT'$ (resp. $\tilde{\bfT}'$) corresponding to $\fm$. We have the following commutative diagram
\be \label{dnew} \xymatrix{\tilde{\bfT}' \ar[d]^{\tilde{\pi}'}\ar[r] &  \bfT' \ar[d]^{\pi'}\ar[r] & \bfT \ar[d]^{\pi}\\
\tilde{\bfT}'_{\tilde{\fm}'} \ar[r]^{i'} &  \bfT'_{\fm'} \ar[r]^{i} & \bfT_{\fm}}\ee where the top arrows are the natural inclusions and the vertical arrows are the canonical surjections coming from the decomposition (\ref{hecke533}) and its analogues.
Note that the localizations of the Hecke algebras in the bottom row of (\ref{dnew}) are Noetherian,
local, complete $\Oo$-algebras. In \cite{Klosin09}, we proved the following properties of the maps $i$ and $i'$ (the proofs in [loc.cit.] are for $D_K=4$, but they generalize verbatim to the general case).

\begin{thm}\label{congsix}
The map $i$ in (\ref{dnew}) is an injection. Moreover if $\phi \in \mN$ is ordinary at $\ell$ and $\ov{\rho}_{\phi}|_{G_K}$ is absolutely irreducible, then both $i$ and $i'$ are surjective. \end{thm}

\begin{proof} The first statement is Proposition 8.5 in \cite{Klosin09}. The surjectivity statement for $i$ is the main result of section 8.2 in [loc.cit.] - see Corollary 8.12. The surjectivity statement for $i'$ follows from Corollary 8.10 in [loc.cit.].  \end{proof}

We also record the following result from \cite{Klosin09}, which again works for any $D_K$. 

\begin{prop}[\cite{Klosin09}, Proposition 8.13] \label{congffrho93} If
$\ov{\rho}_{\phi}|_{G_K}$ is absolutely irreducible, then $\phi \not \equiv
\phi^{\rho}$ \textup{(mod $\varpi$)}. \end{prop}

Fix $\phi \in \mN$ and set $\mathcal{N}_{\phi}:= \{\psi \in
\mathcal{N} \mid
\fm_{\psi} = \fm_{\phi}\}$, where $\fm_{\phi}$ (resp. $\fm_{\psi}$) is the maximal ideal of $\bfT$ corresponding to $\phi$ (resp. to $\psi$).
Similarly, we define $\mN'_{\phi}$ and $\tilde{\mN}'_{\phi}$. 
It is easy to see that we can identify $\bfT_{\fm_{\phi}}$ with the image of $\bfT$ inside  $\textup{End}_{\bfC}(S_{k-1,\phi})$, where $S_{k-1, \phi} \subset
S_{k-1}( D_K, \chi_K)$ is the subspace spanned by $\mN_{\phi}$. Similarly we define $S'_{\phi, k-1}$ and $\tilde{S}'_{\phi,k-1}$. 

\begin{lemma} \label{nfs}  Suppose that $\phi$ is ordinary at $\ell$ and that $\ov{\rho}_{\phi}|_{G_K}$ is absolutely irreducible. Then $\tilde{\mN}'_{\phi} = \mN'_{\phi}$ and the set $\mN_{\phi}$ is formed from the set $\mN'_{\phi}$ by choosing one element from each pair $(\psi, \psi^{\rho})$ such that $\psi \in \mN_{\phi}$. \end{lemma}

\begin{proof} The last statement can be directly deduced from Corollary 8.4 in \cite{Klosin09}, so we just need to show that if $\psi \in \mN'_{\phi}$, then $\psi \in \tilde{\mN}'_{\phi}$. This, as we will demonstrate, follows from ordinarity of $\phi$. Indeed, since the forms $\phi$ and $\psi$ have congruent Hecke eigenvalues for all the operators in $\tilde{\bfT}'$, one can easily show (using Tchebotarev Density Theorem and the Brauer-Nesbitt Theorem) that $\ov{\rho}_{\phi}|_{G_K} \cong \ov{\rho}_{\psi}|_{G_K}$. We will explain the case when $\ell$ is inert in $K$ (the split case being very similar). Using ordinarity at $\ell$ we conclude that both of the representations $\rho_{\phi}$ and $\rho_{\psi}$ when restricted to the decomposition groups at the prime $\fl$ of $K$ lying over $\ell$ have a one-dimensional unramified quotient on which $G_K$ operates by the character which sends $\Frob_{\fl}=\Frob_{\ell}^2$ to the square of the unique unit root $\alpha_h$ of the polynomial $X^2 - a_h(\ell) + \chi_K(\ell) \ell^{k-2}$, where $h \in \{\phi,\psi\}$ and $a_h(\ell)$ is the eigenvalue of the operator $T_{\ell}$ corresponding to $h$. So, $a_h(\ell)^2 = \alpha_h^2 + 2\chi_K(\ell) \ell^{k-2} + \ell^{2k-4} \alpha_h^{-2}$. Since $\ov{\rho}_{\phi}|_{G_K} \cong \ov{\rho}_{\psi}|_{G_K}$, we conclude that $\alpha_{\phi}^2\equiv \alpha_{\psi}^2$ (mod $\varpi$) and hence  $a_{\phi}(\ell)^2 \equiv a_{\psi}^2(\ell)$ (mod $\varpi$). \end{proof}   

 Write $\bfT_{\fm_{\phi}} \otimes E = E \times B_E$, where
$B_E= \prod_{\psi \in \mN_{\phi} \setminus \{\phi\}} E$ and let $B$ denote the image
of $\bfT_{\fm_{\phi}}$ under the composite $\bfT_{\fm_{\phi}}\hookrightarrow \bfT_{\fm_{\phi}}
\otimes E \xrightarrow{\pi_{\phi}} B_E$, where $\pi_{\phi}$ is projection. Denote by
$\delta: \bfT_{\fm_{\phi}} \hookrightarrow \Oo \times B$ the map $T \mapsto
(\lambda_{\phi}(T), \pi_{\phi}(T))$. If $E$ is sufficiently large, there exists $\eta \in
\Oo$ such that $\coker \delta \cong \Oo/\eta \Oo$. This cokernel is
usually called the \textit{congruence module of $\phi$}.

\begin{prop} \label{congthirteen} 
Assume $\phi \in \mN$ is ordinary at $\ell$ and the associated Galois
representation $\rho_{\phi}$ is such that $\ov{\rho}_{\phi}|_{G_K}$
is absolutely irreducible. Then there exists $T \in
\tilde{\bfT}'_{\tilde{\fm}'_{\phi}}$ such that $T\phi=\eta \phi$, $T\phi^{\rho} = \eta \phi^{\rho}$ and
$T\psi=0$ for all $\psi \in
\tilde{\mN}'_{\phi} \setminus \{\phi, \phi^{\rho}\}$.

\end{prop}

\begin{proof} By Theorem \ref{congsix}, the natural
$\Oo$-algebra map $\tilde{\bfT}'_{\tilde{\fm}'_{\phi}} \rightarrow \bfT_{\fm_{\phi}}$ is surjective. So by Lemma \ref{nfs}, it is
enough to find $T \in \bfT_{\fm_{\phi}}$ such that $T\phi= \eta \phi$ and $T\psi=0$ for
every $\psi \in \mN_{\phi}
\setminus
\{\phi\}$. (Note that by Proposition \ref{congffrho93}, $\phi^{\rho} \not\in
\mN_{\phi}$.) It follows from the exactness of the sequence $0 \rightarrow \bfT_{\fm_{\phi}} \xrightarrow{\delta} \Oo
\times B \rightarrow \Oo /\eta\Oo \rightarrow 0$, that $(\eta,0) \in \Oo
\times B$ is in the image of $\T_{\fm_{\phi}} \hookrightarrow \Oo \times B$.
Let $T$ be
the preimage of
$(\eta,0)$ under this injection. Then $T$ has the desired property.
\end{proof}

\begin{prop} [\cite{Hida87}, Theorem 2.5] \label{Hida45} Suppose
$\ell>k$. If $\phi\in
\mN$ is ordinary at
$\ell$, then $$\eta = (*)\frac{\left<\phi,\phi\right>}{\Omega^+_{\phi} \Omega^-_{\phi}},$$
where $\Omega^+_{\phi}, \Omega_{\phi}^-$ denote the ``integral'' periods
defined in
\cite{Vatsal99} and $(*)$ is a $\varpi$-adic unit. \end{prop}

\subsection{Galois representations attached to hermitian modular forms}\label{Galois representations attached to hermitian modular forms}

 Let $f\in \mathcal{S}_{k,-k/2}^{\chi}$ be an eigenform for the local Hecke algebra $\mH_p$ for every $p \nmid D_K$.  For every rational prime $p$, let $\lambda_{p,j}(f)$, $j=1, \dots,
4$, denote the
$p$-Satake
parameters of $f$. (For the definition of $p$-Satake parameters when $p$
inerts or ramifies in $K$, see \cite{HinaSugano}, and for the case when
$p$ splits in $K$, see \cite{Gritsenko90P}.)
Let $\fp$ be a prime of $\OK$ lying over $p$. 

\begin{thm} \label{skinnerurban435} There exists a finite extension $E_f$
of
$\bfQ_{\ell}$
and a $4$-dimensional semisimple Galois representation $\rho_f: G_K
\rightarrow
\GL_{E_F}(V)$ unramified
away from the primes of $K$ dividing $D_K\ell$ and such that
\begin{itemize}
\item [(i)] for any prime $\fp$ of $K$ such that $\fp \nmid D_K\ell$, the
set of eigenvalues of
$\rho_f(\Frob_{\fp})$ coincides with the set of the Satake
parameters of $f$ at $\fp$, i.e., one has $$L(\rho_f,s)_{\fp} = L_{\rm st}(f,s)_{\fp},$$ where $L_{\rm st}(f,s)_{\fp}$ is the $\fp$-component of the function $L_{\rm st}(f,s,\mathbf{1})$ introduced in section \ref{L-functions} and $L(\rho_f,s)_{\fp} = \det(I_4 - \rho_f(\Frob_{\fp})(N\fp)^{-s})^{-1}$;
\item [(ii)] if $\fp$ is a place of $K$ over $\ell$,
the representation
$\rho_f|_{D_{\fp}}$ is crystalline (cf. section \ref{The Bloch-Kato conjecture}).
\item [(iii)] if $\ell > k$, and $\fp$ is a place of $K$
over $\ell$, the
representation
$\rho_f|_{D_{\fp}}$ is short. (For a definition of \textit{short} we refer
the reader to
\cite{DiamondFlachGuo04}, section 1.1.2.)
\item[(iv)] one has $\rho^{\vee}(3) \cong \rho^c\chi^{-2}$, where $c$ denotes the lift  to $G_{\bfQ}$ of the generator of $\Gal(K/\bfQ)$ and $\rho^c(g) = \rho(cgc^{-1})$.
\end{itemize} 
\end{thm}

\begin{rem} Theorem \ref{skinnerurban435} is stated as Theorem 7.1.1 in \cite{SkinnerUrbanpreprint}, where it is attributed to Skinner as a consequence of the work of Morel and Shin. We refer the reader to \cite{SkinnerUrbanpreprint}, section 7 for further discussion. Galois representations attached to hermitian modular forms are also discussed in \cite{BlasiusRogawski94} or \cite{BellaicheChenevierbook}. We will assume Theorem \ref{skinnerurban435} without a proof in what follows. \end{rem}

\begin{rem} \label{unramat2} It is not known if the representation
$\rho_f$ is also unramified at the prime $D_K$. See
\cite{BellaicheChenevier04} for a discussion of this issue. \end{rem}

\subsection{Integral lifts of Hecke operators} 
Fix a character $\chi: \Cl_K \rightarrow \bfC^{\times}$ and for every prime $p \nmid D_K$ write $\mH_p^{\chi}$ for the quotient of $\mH_p^+$ 
acting on $\mS_{k,-k/2}^{\chi}$. 

\begin{definition} \label{hermhecke430}For a prime $p$ which splits in $K$ as $\fp \ov{\fp}$ set $\Sigma_p = \{T_{\fp,1}, T_{\fp,2}, T_{\ov{\fp}, 1}, T_{\ov{\fp},2}\}$ and for a prime $p$ which is inert in $K$ set $\Sigma_p=\{T_{p,0}, U_p\}$.
Set $\mH_{\bfZ}^{\chi}$ to be the $\bfZ$-subalgebra of $\End_{\bfC}(\mS_{k,-k/2}^{\chi})$ generated by the set $\bigcup_{p \nmid D_K\ell} \Sigma_p$. 
 For any $\bfZ$-algebra $A$ set
$\mH^{\chi}_A:=
\mH^{\chi}_{\bfZ} \otimes_{\bfZ} A$. \end{definition}

Note that $\mH_{\bfZ}^{\chi}$ is a finite free $\bfZ$-algebra. As before let $E$ be a sufficiently large finite extension of $\bfQ_{\ell}$ with valuation ring $\Oo$. We fix a uniformizer $\varpi \in \Oo$. To ease notation put $\mH^{\chi}= \mH^{\chi}_{\Oo}$. 

\begin{lemma} \label{integrality382} Let $\ell\nmid 2D_K$ be a rational prime, $E$
a
finite extension of
$\bfQ_{\ell}$ and $\Oo$ the valuation ring of $E$. Suppose that $f \in \mS_{k,-k/2}^{\chi}$, $T \in \mH^{\chi}$ and $c_f(h,q)\in \bfC$ is the $(h,q)$-Fourier coefficient of $f$. Write $c_{Tf}(h,q)$ for the corresponding Fourier coefficient of $Tf$. Assume that there exists $\alpha \in \bfC$ such that $\alpha c_f(h,q) \in \Oo$. Then $\alpha c_{Tf}(h,q) \in \Oo$.  
\end{lemma}

\begin{proof} This follows directly from Propositions \ref{f3} and \ref{f4} (note that the powers of $p$ in Proposition \ref{f4} are $\ell$-adic units). 
\end{proof}

\begin{prop} \label{basis956} The space $\mathcal{S}_{k,-k/2}^{\chi}$ has a
basis consisting of
eigenforms. \end{prop}

\begin{proof} This is a standard argument, which uses the fact that
$\mH^{\chi}_{\bfC}$ is
commutative and all $T \in \mH^{\chi}_{\bfC}$ are self-adjoint.
\end{proof}

From now on $\mathcal{N}^{\hh}$ will denote a fixed basis of eigenforms of
$\mathcal{S}_{k,-k/2}^{\chi}$.

\begin{thm} \label{eichshi930} Let $f \in \mathcal{N}^{\hh}$. There exists
a number field
$L_f$ with ring of integers $\Oo_{L_f}$
such that the $f$-eigenvalue of every Hecke operator $T \in \mH^{\chi}_{\mathcal{O}_{L_f}}$ lies in $\Oo_{L_f}$.
\end{thm}

\begin{proof} This can be seen as a consequence of  Theorem \ref{skinnerurban435}. \end{proof}

Let $\ell$ be a rational prime and $E$ a finite extension of $\bfQ_{\ell}$
containing
the fields $L_f$ from Theorem \ref{eichshi930} for all $f \in
\mathcal{N}^{\hh}$.
Denote by $\Oo$ the valuation ring of $E$ and by $\varpi$ a uniformizer of
$\Oo$. As in the case of elliptic modular forms, $f \in \mathcal{N}^{\hh}$
gives rise to an $\Oo$-algebra homomorphism $\mH^{\chi}\rightarrow
\Oo$ assigning to $T$ the eigenvalue of $T$
corresponding to the eigenform $f$. We denote this homomorphism by $\lambda_f$.
 Proposition \ref{basis956} and Theorem \ref{eichshi930}
imply that we have $$\mH^{\chi}_E\cong \prod_{f \in
\mathcal{N}^{\hh}}
E.$$ Moreover, as 
in the elliptic modular case, we have
\be \label{hermdec} \mH^{\chi}
\cong \prod_{\fm} \mH^{\chi}_{\fm},\ee where the product runs over
the
maximal ideals of $\mH^{\chi}$ and $\mH^{\chi}_{\fm}$ denotes
the localization of
$\mH^{\chi}$ at $\fm$.

The descent map $\textup{Desc}$ defined in section \ref{Descent} induces the following map (for which we use the same name): \be \label{newdesc1}\textup{Desc} : \mH^{\chi} \rightarrow \tilde{\bfT}',\ee given by the following formulas (cf. Theorems \ref{heckedesc}, \ref{heckedesc2} and Proposition \ref{heckedesc3})
\be\label{form666} \begin{split} \Desc (T_{\fp,1}) & = u_1(p+1) T_p, \\
\Desc (T_{\fp,2})& =u_2( T^2_p
+p^{k-1}+p^{k-3}) \\
\Desc (T_{p,0}) & = u_3(p^2+1) T_p^2+ p^4 + p^3+ p
- 1,\\   
\Desc (U_p) & = u_4T_p^4 + u_5 (p+3) T_p^2 +
p^{2k+4}(p^2+p+1)
\end{split}\ee where $u_1, u_2, u_3, u_4, u_5 \in \Oo^{\times}$. 

The first two formulas in (\ref{form666}) are for a prime $\fp$ of $K$ lying over a split prime $p \neq \ell$ and the last two for an inert prime $p \neq \ell$.
The map (\ref{newdesc1}) factors through \be \label{newdesc2} \textup{Desc} : \mH^{\chi} \twoheadrightarrow \mH^{{\rm M}, \chi}\rightarrow \tilde{\bfT}',\ee where $\mH^{{\rm M}, \chi}$ is the quotient of $\mH^{\chi}$ acting on the space $\mS^{{\rm M}, \chi}_{k, -k/2}$. 
Now fix a newform $\phi \in S_{k-1}(D_K, \chi_K)$ such that $\ov{\rho}_{\phi}|_{G_K}$ is absolutely irreducible. Then by Proposition \ref{congffrho93} we in particular have that $\phi \neq \phi^{\rho}$. Write $f_{\phi, \chi}$ for the Maass lift of $\phi$ lying in the space $\mS^{{\rm M}, \chi}_{k, -k/2}$. Write $\tilde{\fm}'_{\phi}$ for the maximal ideal of $\tilde{\bfT}'$ corresponding to $\phi$ and $\fm_{\phi}$ (resp. $\fm_{\phi}^{\rm M}$) for the corresponding maximal ideals of $\mH^{\chi}$ (resp. $\mH^{{\rm M}, \chi}$).
The map (\ref{newdesc2}) induces the corresponding map on localizations:\be \label{newdesc3} \textup{Desc} : \mH^{\chi}_{\fm_{\phi}} \twoheadrightarrow \mH^{{\rm M}, \chi}_{\fm_{\phi}^{\rm M}}\rightarrow \tilde{\bfT}'_{\tilde{\fm}'_{\phi}}.\ee 

\begin{prop} \label{surjhecke} Let $\phi \in \mN$ be such that $\ov{\rho}_{\phi}|_{G_K}$ is absolutely irreducible. Assume $\ell \nmid (k-1)(k-2)(k-3)$ and that $\phi$ is ordinary at $\ell$. Then the map (\ref{newdesc3}) is surjective. \end{prop}
\begin{proof} Let $p \nmid \ell D_K$ be a prime. We need to show that $T_p$ is in the image of $\textup{Desc}$. Note that by the first formula in (\ref{form666}) this is clear if $\ell \nmid (p+1)$ (for a split $p$) and (by the third formula in (\ref{form666})) if $\ell \nmid (p^2+1)$ (for an inert $p$). To prove this result we will work with Galois representations. As discussed in section \ref{Galois representations attached to hermitian modular forms} to every eigenform $f \in\mS_{k, -k/2}^{\chi}$ one can attach an $\ell$-adic Galois representation $\rho_f: G_K \rightarrow \GL_4(E)$ and it follows from Theorem \ref{skinnerurban435}(i) and Proposition \ref{product134} together with the Tchebotarev Density Theorem and the Brauer-Nesbitt Theorem that \be \label{specialform}\rho_{f_{\phi,\chi}} = \bmat \rho_{\phi}|_{G_K} \\ & (\rho_{\phi}\otimes \epsilon)|_{G_K}\emat \otimes \chi\epsilon^{2-k/2},\ee  where we treat $\chi$ as an $\fl$-adic Galois character via class field theory (here $\fl$ denotes a prime of $K$ lying over $\ell$). Note that the reason for the $(p+1)$- and $(p^2+1)$-factors in (\ref{form666}) is exactly the presence of the cyclotomic character in the lower-right corner of (\ref{specialform}), because for a prime $\fp$ of $K$ lying over $p$ we get $\tr \rho_{f_{\phi,\chi}}(\Frob_{\fp}) = u(\epsilon(\Frob_{\fp})+1) \tr \rho_{\phi}(\Frob_{\fp})$ for $u=\epsilon^{2-k/2}(\Frob_{\fp}) \in \Oo^{\times}$. We will now construct an idempotent in the group algebra $\mH^{{\rm M},\chi}_{ \fm^{\tuM}_{\phi}}[G_K]$ that will kill off the $\epsilon$-part from the expression $\epsilon(\Frob_{\fp})+1$. 

From now on assume that $p$ is split (the inert case being analogous) and fix a prime $\fp$ of $\OK$ over it. 
Let $\tilde{\mN}'_{\phi}\subset \mN$, as before, be the subset consisting of those forms $\psi \in S_{k-1}(D_K, \chi_K)$ whose corresponding maximal ideal of $\tilde{\bfT}'$ is $\tilde{\fm}'_{\phi}$. The set $\tilde{\mN}'_{\phi}$ is in one-to-one correspondence with the set consisting of mutually orthogonal Hecke eigenforms in $\mS^{{\rm M}, \chi}_{k, -k/2}$ whose corresponding maximal ideal in $\mH^{{\rm M}, \chi}$ is $\fm_{\phi}^{{\rm M}}$.  Set $R' := \prod_{\psi \in \tilde{\mathcal{N}}'_{\phi}}
\Oo$
and
let $R$ be the
$\Oo$-subalgebra of $R'$ generated by the tuples $(\lambda_{f_{\psi,\chi}}(T))_{\psi \in
\tilde{\mathcal{N}}'_{\phi}}$
for all $T \in \mH^{{\rm M},\chi}$. Then $R$ is a complete
Noetherian local
$\Oo$-algebra with residue field $\bfF= \Oo/\varpi$. It is a standard argument
to show that $R
\cong \mH^{{\rm M},\chi}_{ \fm^{\tuM}_{\phi}}$.
Let $I_{\ell}$ denote the inertia group at $\ell$.
For every $\psi\in \mN$, ordinary
at $\ell$, we have by (\ref{specialform}) and Theorem 3.26 (2) in
\cite{Hida00} that (note that $\chi$ is unramified)
$$\rho_{f_{\psi,\chi}}|_{I_{\ell}} \cong
\bmat
\epsilon^{k/2} &*\\ &\epsilon^{2-k/2}
\\ && \epsilon^{1+k/2} &*\\ &&& \epsilon^{3-k/2} \emat.$$ If $\ell \nmid
(k-1)(k-2)(k-3)$ it is easy to see that there
exists
$\sigma\in I_{\ell}$ such that the elements $\beta_1:=
\epsilon^{k/2}(\sigma)$, $\beta_2:=\epsilon^{2-k/2}(\sigma) $,
$\beta_3:= \epsilon^{1+k/2}(\sigma)$, $\beta_4:= \epsilon^{3-k/2}(\sigma)$ are all
distinct mod $\varpi$ and non-zero mod $\varpi$.
For every $\psi$ as above, we choose a basis of the space of $\rho_{f_{\psi,\chi}}$ so that
$\rho_{f_{\psi, \chi}}$ is $\Oo$-valued and $\rho_{f_{\psi, \chi}}(\sigma) = \diag(\beta_1, \beta_2,
\beta_3, \beta_4)$. Let $S$ be the set consisting of the places of $K$
lying over $\ell$ and the primes dividing $D_K$. Note that we can treat
$\rho_{f_{\psi,\chi}}$ as a representation of $G_{K,S}$, the Galois group of the
maximal Galois extension of $K$ unramified away from $S$. Moreover, $\tr
\rho_{f_{\psi,\chi}}(G_{K,S})\subset R$, since $G_{K,S}$ is generated by conjugates
of $\Frob_{\fp}$, $\fp \not \in S$ and for such a $\fp$, $\tr
\rho_{f_{\psi,\chi}}(\Frob_{\fp}) \in R$ by Theorem \ref{skinnerurban435} (i) and
the fact that the coefficients of the characteristic polynomial of
$\rho_{f_{\psi,\chi}}(\Frob_{\fp})$ belong to $\mH^{\chi}$.
Set $$e_j=\prod_{l\neq j}\frac{\sigma - \beta_l}{\beta_j-\beta_l} \in
\Oo[G_{K,S}] \hookrightarrow
R[G_{K,S}]$$ and $e:= e_1 + e_2$. Let $$\rho:= \prod_{\psi \in
\tilde{\mathcal{N}}'_{\phi}}
\rho_{f_{\psi,\chi}}: G_{K,S}
\rightarrow \prod_{\psi \in \tilde{\mathcal{N}}'_{\phi}} \GL_4(\Oo).$$ We
extend
$\rho$
to an
$R$-algebra map
$\rho': R[G_{K,S}] \rightarrow M_4(R')$.

Set $$r_e(\fp):=
\epsilon^{k/2-2}(\Frob_{\fp})\tr \rho' (e \Frob_{\fp}) \in R'.$$
We claim that $r_e(\fp) \in R$. Note that
$\rho' (e \Frob_{\fp})$ is a
polynomial in $\rho'(\sigma^i \Frob_{\fp})$, $i=0,1,2,3$, with coefficients
in $\Oo$, so it is enough to
show
that $\tr \rho'(\sigma^i \Frob_{\fp}) \in R$. Fix $i$, set $\tau = \sigma^i
\Frob_{\fp}\in G_{K,S}$. Then by the Tchebotarev Density Theorem, $G_{K,S}$ is generated by
conjugacy classes of Frobenii away from $S$, so $\tr \rho'(\tau)$ is
the limit of $\tr \rho'(\Frob_{\fl})\in R$ for some sequence of primes $\fl
\not\in S$. So, we get $\tr \rho'(\tau) \in R$ by completeness of $R$.

 Note that $$\rho'(\Frob_{\fp} e)
=\prod_{\psi \in
\tilde{\mathcal{N}}'_{\phi}}
 \rho_{f_{\psi,\chi}} (\Frob_{\fp}) \rho'_{f_{\psi,\chi}}(e) =
\prod_{\psi \in
\tilde{\mathcal{N}}'_{\phi}}
\rho_{\psi}
(\Frob_{\fp})\epsilon^{2-k/2}(\Frob_{\fp})$$ and thus $$r_e(\fp) = (a_{\psi}(p))_{\psi \in
\tilde{\mathcal{N}}'_{\phi}} \in
R,$$
where
$\psi=\sum_{n=1}^{\iy} a_{\psi}(n) q^n$. Define $T^{\tuM}(p)$ to be the
image of
$r_e(\fp)$ under the $\Oo$-algebra isomorphism $R \xrightarrow{\sim}
\mH^{{\rm M},\chi}_
{\fm^{\tuM}_{\phi}}$. Note that $\textup{Desc}(T^{\tuM}(p)) =T_p$. \end{proof}

The above arguments yield the following result.

\begin{thm} \label{existhecke} Let $\phi \in \mN$ be such that $\ov{\rho}_{\phi}|_{G_K}$ is absolutely irreducible. Assume $\ell \nmid 2D_K(k-1)(k-2)(k-3)$ and that $\phi$ is ordinary at $\ell$. Let $f_{\phi,\chi} \in \mS_{k,-k/2}^{{\rm M},\chi}$ be the Maass lift of $\phi$. Then there exists $T^{\rm h} \in \mH^{\chi}_{\Oo}$ such that $T^{\rm h} f_{\phi,\chi} = \eta f_{\phi,\chi}$ and $T^{\rm h}f=0$ for any eigenform $f \in \mS_{k,-k/2}^{{\rm M},\chi}$ orthogonal to $f_{\phi,\chi}$. \end{thm}
\begin{proof} Let $T \in \tilde{\bfT}'_{\tilde{\fm}'_{\phi}}$ be as in Proposition \ref{congthirteen} and let $T^{\rm h}_0 \in \mH^{\chi}_{\fm_{\phi}}$ be an element of the inverse image of $T$ under the map (\ref{newdesc3}), which by Proposition \ref{surjhecke} is surjective. Pulling back $T_0^{\rm h}$ to $\mH$ via the canonical projection induced by decomposition (\ref{hermdec}) we obtain an operator $T^{\rm h}$ with the desired property. \end{proof}

\section{Eisenstein series and theta series}\label{Inner product
formula}

The goal of this section is to 
express the inner product of a hermitian Siegel Eisenstein series of level $N$ multiplied by a certain hermitian theta series against an eigenform $f \in \mM_{n,k}(N)$ in terms of the standard $L$-function of $f$. In this section we also prove that 
 the Fourier
coefficients of the Eisenstein series and the theta series that we will
use are $\ell$-adically integral. We derive the desired formulas and properties from certain calculations carried out by Shimura in \cite{Shimura97} and \cite{Shimura00}. We will often refer the reader to [loc.cit.] for some definitions, facts and formulas, but whenever we do so, we will explain how the statements in [loc.cit.] referenced here imply what we need. 
We will set $h_K=\# \Cl_K$. The results of this section are valid for $U_n$ for a general $n>1$.

\subsection{Some coset decompositions}

Let $Q$ be any finite subset of $\GL_n(\AKf)$ of cardinality $h_K$ such
that $\det Q= \Cl_K$
under the canonical map $c_K: \AK^{\times} \twoheadrightarrow  \Cl_K$. Then we have by (\ref{gwiazdka})

\be \label{GLndecomp}
\GL_n(\AK) =
\bigsqcup_{q \in Q} \GL_n(K) \GL_n(\bfC) q \GL_n(\hat{\Oo}_K).\ee
For $r \in \GL_n(\AKf)$, the group $r \GL_n(\hat{\Oo}_K) r^{-1}$ is also
an open compact subgroup of $ \GL_n(\AKf)$ with $\det r 
\GL_n(\hat{\Oo}_K)
r^{-1} = \hat{\Oo}_K^{\times}$. Hence by the Strong Approximation 
Theorem
for $\GL_n$ (\cite{Bump97}, Theorem 3.3.1) we also have \be
\label{GLndecomp2} \GL_n(\AK)
= \bigsqcup_{q \in Q} \GL_n(K) \GL_n(\bfC) q 
r\GL_n(\hat{\Oo}_K)r^{-1}.\ee
As before, for any $q \in \GL_n$ we put
$p_q:= \bsmat q \\ & \hat{q} \esmat \in U_n$.
Write $P$ for the Siegel parabolic of $U_n$.

\begin{lemma} \label{pardecomp} For any $r \in \GL_n(\AKf)$ the 
following
decomposition holds:
$$P(\AQ) = \bigsqcup_{q \in Q} P(\bfQ) P(\bfR) p_q p_r\mK_P p_r^{-1},$$
where $\mK_P:=
U_n(\hat{\bfZ}) \cap P(\AQ)$. \end{lemma}

\begin{proof} Write $P=MN$ for the Levi decomposition. As $M \cong
\Res_{K/\bfQ} \GL_{n/K}$, and $M \cap N = \{I_{2n}\}$, we get by
(\ref{GLndecomp}):
$$P(\AQ) = M(\AQ) N(\AQ) = \bigsqcup_{q \in Q} M(\bfQ) M(\bfR) p_q p_r 
\mK_M
p_r^{-1}N(\AQ),$$ where $\mK_M:= \{ p_x \mid x \in \GL_n(\hat{\Oo}_K)\}
\subset
\mK_P.$ Set $\mK_{P,r}:= p_r \mK_P p_r^{-1}$. This is a compact open
subgroup of $P(\AQf)$. Let $$X:= \bigcap_{q \in Q} p_q \mK_{P,r}
p_q^{-1} \cap N(\AQ).$$ (Note if
$(2n, h_K)=1$, Corollary \ref{scalarcor} implies that we can find $Q$ so 
that $p_q$ are scalars, and then $X= \mK_{P,r} \cap N(\AQ)$.) By
\cite{Shimura97}, Lemma 9.6(1), we know that $N(\AQ) = N(\bfQ) X 
N(\bfR)$,
since $X N(\bfR) = N(\bfR) X$ is open in $N(\AQ)$. 
Thus we have \begin{multline} P(\AQ) = \bigsqcup_{q \in Q} M(\bfQ) 
N(\AQ)
M(\bfR) p_q \mK_{P,r}  \\
= \bigsqcup_{q \in Q} M(\bfQ) N(\bfQ) X N(\bfR) M(\bfR) p_q \mK_{P,r}
=
\bigsqcup_{q \in Q} P(\bfQ) P(\bfR) X p_q \mK_{P,r},\end{multline}
where the
first
equality
follows from normality of $N$ in $P$ and the fact that $\mK_M \subset \mK_P
\subset P(\AQ)$ while the third one follows from the fact that
$X$ has trivial infinite components. Note that every $x \in X$ can be
written as $x = p_q k p_q^{-1}$ for some $k \in \mK_{P,r}$. Hence $X p_q
\mK_{P,r}
\subset p_q \mK_{P,r}$. The other containment is obvious, so we have $Xp_q
\mK_{P,r}
=
p_q
\mK_{P,r}$. Thus finally $P(\AQ) = \bigsqcup_{q \in Q} P(\bfQ) P(\bfR) p_q
\mK_{P,r},$   
as desired. 
\end{proof}

Fix $r \in \GL_n(\AKf)$ and an integer $N>1$.
 Set $$\Gamma^{\rm h}_{j,r}(N) = U_n(\bfQ) \cap U_n(\bfR)
p_r \mK_j(N)p_r^{-1},\quad \textup{for} \hs j=0,1.$$
and for any subgroup $\Gamma$ of $U_n(\bfQ)$ we put $\Gamma^P:= \Gamma\cap
P(\bfQ).$ Note that $\Gamma^{\rm h}_{j, I_n}(N) = \Gamma^{\rm h}_{j,n}(N)$ for $j=0,1$ with $\Gamma^{\rm h}_{j,n}(N)$ defined as in section \ref{The unitary group}. In the
discussion below we keep $N$ fixed and to shorten notation we write 
$\Gamma_r=
\G_{0,r}^{\rm h}(N)$.

\begin{lemma} \label{decomp905} The canonical injection $$\G_r^P 
\setminus
\G_r
\hookrightarrow P(\bfQ) \setminus (U_n(\bfQ)\cap P(\AQ) U_n(\bfR)
p_r \mK_{0,n}(N) p_r^{-1})$$ is a bijection. \end{lemma}

\begin{proof} We need to prove surjectivity. Set $\mK_{P,r}= p_r \mK_P
p_r^{-1}$ and $\mK_{0,r}(N):= p_r \mK_{0,n}(N) p_r^{-1}$. By Lemma 
\ref{pardecomp}
we
have $$P(\AQ) = \bigsqcup_{q \in Q} P(\bfQ) P(\bfR) p_q \mK_{P,r},$$ so
$$P(\AQ) 
U_n(\bfR)
p_r \mK_0(N)p_r^{-1} = \bigcup_{q \in Q} P(\bfQ) P(\bfR) p_q \mK_{P,r} U_n(\bfR)
\mK_{0,r}(N).$$
Note that $\det U_n(\bfQ) \subset H(\bfQ)$, where $$H=\{x \in \Res_{K/\bfQ}\bfG_{m/K} \mid x\ov{x}=1\},$$ and $$\det (P(\bfQ) P(\bfR)
p_q \mK_{P,r} U_n(\bfR) \mK_{0,r}(N)) \subset \det p_q H(\bfQ) \det D,$$ 
with
$D=U_n(\bfR)\mK_{0,r}(N)$. Thus $$U_n(\bfQ)\cap P(\bfQ) P(\bfR) p_q
\mK_{P,r} U_n(\bfR)
\mK_{0,r}(N)=\emptyset$$ unless $\det p_q \in H(\bfQ) \det D$.
However,
$Q$ is      
chosen so that $\det p_q$ runs over all the ideal classes of $K$. It follows from Lemma 8.14 in \cite{Shimura97} that there is a bijection between
$\Cl_K$ and $H(\AQ)/H(\bfQ) \det D$, thus $\det p_q \in H(\bfQ) \det D$
only
for one $q$ (which without loss of generality we can take to equal
$I_{2n}$). Thus  \begin{multline} P(\bfQ) \setminus (U_n(\bfQ)\cap P(\AQ)
U_n(\bfR)
p_r\mK_{0,n}(N)p_r^{-1}) \\
= P(\bfQ) \setminus (U_n(\bfQ) \cap P(\bfQ) P(\bfR) \mK_{P,r}
U_n(\bfR)
\mK_{0,r}(N)\\
= P(\bfQ) \setminus (U_n(\bfQ) \cap P(\bfQ)U_n(\bfR)\mK_{0,r}(N)).
\end{multline}
Thus if $g\in U_n(\bfQ)$ can be written as $g=p k$ with $p \in P(\bfQ)$ and $k  \in U_n(\bfR)
\mK_{0,r}(N)$,
then clearly $p^{-1} g \in \G_r$.
\end{proof}

We will need more congruence subgroups. Let $\Gamma(N):=\Gamma_n^{\rm h}(N)$ be the subgroup introduced in section \ref{Notation and terminology} and set 
$$\G_u(N):=\left\{ \bsmat  A&B \\ C&D \esmat \in \Gamma_{0,n}^{\rm h}(N) \mid 1-\det 
D\in
N\OK \right\}.$$

\begin{lemma} \label{gulemma} The canonical injection $$\Gamma(N)^P 
\setminus 
\G(N)
\hookrightarrow \G_u(N)^P \setminus \G_u(N)$$ is a bijection. 
\end{lemma}

\begin{proof} We need to prove surjectivity. Let $g=\bsmat A&B\\ C&D
\esmat
\in \G_u(N)$. We have $D \in M_n(\OK)$ with $1- \det D \in N\OK$. Note
that the reduction of $D$ mod
$N\OK$ lies in $\SL_n(\OK / N\OK)$. By the
strong approximation for $\SL_n$, the reduction map $\SL_n(\OK)
\rightarrow \SL_n(\OK /N \OK)$ is surjective. Thus there exists $q \in
\SL_n(\OK)$ such that $D-q \in M_n(N \OK)$. Put $$h = \bmat q^* &
-D^*Bq^{-1} \\ & q^{-1}\emat .$$ Then $h \in \G_u(N)^P$ and $h g \in
\G(N)$.
\end{proof}

\subsection{Eisenstein series}

As before, let $N>1$ be an integer and set $\Oo_{K,p}= \bfZ_p \otimes \OK$. Let $\psi$ be a Hecke character of $K$
satisfying \be \label{cond11}
\psi_{\iy}(x) = x^{m} |x|^{-m}\ee for a positive integer $m$ and \be \label{cond12} \psi_p(x)=1
\hspace{5pt} \textup{if} \hspace{5pt} p \neq \iy, \hs x \in 
\Oo_{K,p}^{\times}
\hspace{5pt} and \hspace{5pt} x-1 \in N \Oo_{K,p}.\ee
Set $\psi_N =
\prod_{p\mid N}
\psi_p$. Let $\d_P$
denote the
modulus character of $P$. We define
$$\mu_P: M(\bfQ) N(\AQ) \setminus U_n(\AQ) \rightarrow \bfC$$ by
setting
$$\mu_P(g)=\begin{cases} 0 & g \not \in P(\AQ) \mK_{0,n,\infty}^+ \mK_{0,n}(N)\\
\psi(\det
d_q)^{-1} \psi_N(\det
d_k)^{-1} j(k_{\iy}, \bfi_n)^{-m} & g=qk\in P(\AQ)
(\mK_{0,n,\infty}^+\mK_{0,n}(N)).\end{cases}$$
Note that $\mu_P$ has a local
decomposition
$\mu_P=\prod_p \mu_{P,p}$, where
\be \label{muv87}\mu_{P,p}(q_p k_p) = \begin{cases} \psi_p(\det
d_{q_p})^{-1}& \textup{if $p
\nmid N\iy$},\\ \psi_p(\det d_{q_p})^{-1} \psi_p(\det d_{k_p}) &
\textup{if $p \mid N, p \neq
\iy$}, \\ \psi_{\iy}(\det d_{q_{\iy}})^{-1}j(k_{\iy}, \I)^{-m}&
\textup{if $p=\iy$}
\end{cases}\ee
and $\delta_P$ has a local decomposition $\delta_P=
\prod_p \delta_{P,p}$, where \be \label{deltav87}\d_{P,p}\left(\bmat A 
\\
& \hat{A} \emat uk
\right) = |\det A \det \ov{A}|_{\bfQ_p}.\ee

\begin{definition} \label{siegelpos454} The series $$E(g,s,N,m,\psi):=
\sum_{\g \in P(\bfQ) \setminus U_n(\bfQ)} \mu_P(\g g) \d_P(\g
g)^{s/2}$$ is called the \textit{\textup{(}hermitian\textup{)} Siegel
Eisenstein series of
weight $m$,
level $N$ and character $\psi$}.
\end{definition}

For $x \in U_n(\AQf)$, $g \in U_n(\bfR)$ and $Z=g \bfi_n$ we define
(\cite{Shimura00}, (17.23a)) $$E_x(Z,s,m,\psi,N) =
j(g,\bfi_n)^{m} E(xg,s,N,m,\psi).$$ Fix $r \in
\GL_n(\AKf)$, write $A_r = P(\bfQ)
\setminus (U_n(\bfQ) \cap P(\AQ) p_r U_n(\bfR) K_0(N) p_r^{-1}$. Then
 \be\label{eisen630}
E_{p_r}(Z,s,m,\psi,N) = \psi_{\tuf} (\det r^*) |\det (rr^*)|_{\bfQ}^{s}
\sum_{a \in A_r} N(\fa_{p_r}(a))^s \psi[a]_{p_r} (\det (\textup{Im}
Z))^{s-m/2}|_m
a,\ee where $\fa_{p_r}(a)$, $\psi[a]_{p_r}$ are defined in section 18 of
\cite{Shimura97}. (Our notation differs slightly from that in \cite{Shimura97}, which we quote here. In particular our $r$ corresponds to $\hat{g}$ in [loc.cit.] and one has $\delta(Z) = \det \eta(Z)$ by (6.3.11) in [loc.cit.] and $\eta(Z) = 2 \textup{Im}(Z)$ by (6.1.8) in [loc.cit.].) As stated in the proof of Lemma 17.13 of
\cite{Shimura00}, there exists a finite set $B\subset U_n(\bfQ)$ such 
that
$A_r = \bigsqcup_{b \in B} S_b b$, where $S_b = (P(\bfQ) \cap U_n(\bfR) 
b
\G_r b^{-1}) \setminus b \G_r b^{-1}$. By Lemma \ref{decomp905}, we can
take $B=\{I_{2n}\}$. It follows then from the proof of Lemma 17.13 of
\cite{Shimura00}, that $\fa_{p_r}(\gamma) = \fa_{p_r}(I_{2n})$ for 
$\gamma
\in S_{I_{2n}} = (P(\bfQ) \cap U_n(\bfR)\G_r) \setminus \G_r$. By the
definition of $\fa_{p_r}(a)$ in \cite{Shimura97}, Lemma 18.7(3), we get
$\fa_{p_r}(I_{2n}) = \OK$ (\cite{Shimura97}, (18.4.4)), so for $\gamma 
\in
S_{I_{2n}}$, we have $$N(\fa_{p_r}(\gamma)) = N(\fa_{p_r}(I_{2n})) = 
1.$$
Moreover, for $\g \in S_{I_{2n}}$, we have by \cite{Shimura97}, Lemma
18.7(3) and (12.8.2) \begin{multline} \psi[\gamma]_{p_r} = 
\psi_{\iy}(\det
d_{\gamma}) \psi^*(\det d_{\gamma} \fa_{p_r}(\gamma)^{-1}) =
\psi_{\iy}(\det d_{\gamma}) \psi^*(\det d_{\gamma} 
\fa_{p_r}(I_{2n})^{-1})
=\\ = \psi_{\iy}(\det d_{\gamma}) \psi^*(\det d_{\gamma}\OK) = 
\psi_N^{-1}(\det d_{\gamma}).\end{multline} Hence we get \be
\label{Eisen737} E_{p_r}(Z,s,m,\psi,N) = \psi_{\tuf} (\det r^*) |\det
(rr^*)|_{\bfQ}^{s} \sum_{\gamma\in \G_r^P \setminus \G_r } \psi_N(\det
d_{\gamma})^{-1} \det(\textup{Im} (Z))^{s-m/2}|_m \g.\ee

In what follows we will write $E_r$ instead of $E_{p_r}$ for $r \in
\GL_n(\AKf)$. For any congruence subgroup $\G$ of $U_n(\bfQ)$ we define 
an
Eisenstein series (cf. \cite{Shimura00} (17.3), (17.3a), where a similar 
definition is made in the case when $\G$ is a congruence subgroup of
$SU_n(\bfQ)$):
\be \label{defeisen5} E(Z,s,m,\G) = \sum_{\g \in \G^P\setminus \G}
\det(\textup{Im})(Z)^{s-m/2} |_{\g}.\ee
Let $X=X_{m,N}$ be the set of all Hecke characters $\psi$ of $K$ satisfying
(\ref{cond11}) and (\ref{cond12}).

\begin{lemma} \label{allhecke3} Assume $r \in \GL_n(\AKf)$ is such that
$p_r$
is a scalar.
Then
$$\sum_{\psi \in X} \psi_{\tuf}
(\det
r^*)^{-1} |\det (rr^*)|_{\bfQ}^{-s}
E_r(Z,s,m,\psi,N) = \# X E(Z,s,m,\Gamma_{1,n}(N)).$$
\end{lemma}

\begin{proof} Note that $\# X \neq \emptyset$ because of our assumption that $N>1$ by Lemma 11.14(1) in \cite{Shimura97}. Let $x \in \OK$, 
$(x,N)=1$,
be such
that there exists $\psi' \in X$ with $\psi'_N(x)\neq 1$. Then 
$\sum_{\psi
\in X} \psi_N(x)=0$.  Thus, \begin{multline} A:= \sum_{\psi 
\in
X} \psi_{\tuf} (\det r^*)^{-1} |\det (rr^*)|_{\bfQ}^{-s} 
E_r(Z,s,m,\psi,N)
=\\ = \sum_{\psi \in X} \sum_{\gamma\in \G_r^P \setminus \G_r }
\psi_N(\det d_{\gamma})^{-1} \det(\textup{Im}(Z))^{s-m/2}|_m \g=\\ =   
\sum_{\gamma\in
\G_r^P \setminus \G_r } \left( \sum_{\psi \in X} \psi_N(\det
d_{\gamma})^{-1} \right) \det(\textup{Im}(Z))^{s-m/2}|_m 
\g.\end{multline}
By our
assumption on $r$ we have $\G_r = \G^{\rm h}_{0,n}(N)$. Thus the inner sum equals 0 unless   
$\g \in \G_u(N)$, in which case it equals $\# X$. Hence we get $$A= \#X  
\sum_{\gamma \in \G_u(N)^P \setminus \G_u(N)}
\det(\textup{Im}(Z))^{s-m/2}|_m \g.$$
Using Lemma \ref{gulemma} we further get \be \label{aux79} A = \#X
\sum_{\gamma \in
\G(N)^P \setminus \G(N)} \det(\textup{Im}(Z))^{s-m/2}|_m \g =  
\# X E(Z,s,m,\G(N)).\ee Now
apply $\sum_{\g \in \G(N) \setminus \G_{1,n}^{\rm h}(N)} |_m \g$ to both sides of
(\ref{aux79}). We have $E_r(Z,s,m,\psi,N)|_m \g =
E_r(Z,s,m,\psi,N)$ for every $\g \in \G_{1,n}^{\rm h}(N)$, and $$\sum_{\g \in \G(N)
\setminus \G_{1,n}^{\rm h}(N)} E(Z,s,m,\G(N))|_m \g = [\G_{1,n}^{\rm h}(N)^P: \G(N)^P]
E(Z,s,m,\G_{1,n}^{\rm h}(N))$$ by \cite{Shimura00} (17.5) together with Remark 
17.12(2). Note also that $[\G_{1,n}^{\rm h}(N)^P: \G(N)^P] = [\G_{1,n}^{\rm h}(N): \G(N)]$, 
hence 
we finally get $A= \# X E(Z,s,m,\G_{1,n}^{\rm h}(N))$.
 \end{proof}

\begin{rem} Lemma \ref{allhecke3} is true even without assuming that 
$p_r$
is a scalar. Its conclusion can then be stated as $$\sum_{\psi \in X}
\psi_{\tuf}
(\det
r^*)^{-1} |\det (rr^*)|_{\bfQ}^{-s}
E_r(Z,s,m,\psi,N) = E(Z,s,m,\G_{1,r}(N)).$$ We omit the proof however, as we will have no need for this result. \end{rem}

\subsection{Theta series and inner products} \label{Theta series and 
inner
products}   

Let $\Gamma$ be a congruence subgroup of $U_n(\bfQ)$. Let 
$F,
G \in M^{\textup{sh}}_k(\Gamma)$ (with at least one of the forms cuspidal), where the superscript `sh', as before, indicates 
that
the
forms are not necessarily holomorphic. 
We set $$\left< F,G \right>_{\Gamma}:=\int_{\G \setminus \bfH_n} F(Z) \ov{G(Z)}
\det (\textup{Im}(Z))^k dZ,$$ and
$$\left< F,G \right> = \left(\int_{\G \setminus \bfH_n} dZ\right)^{-1}
\int_{\G \setminus \bfH_n} F(Z) \ov{G(Z)}
\det (\textup{Im}(Z))^k dZ,$$
(compare with \cite{Shimura97}, (10.9.2)).
Let $\xi$ be an (algebraic) Hecke character of $K$. Then $\xi^c$ 
defined
by $\xi^c(x) = \xi(\ov{x})$ is also an algebraic Hecke character of 
$K$ 
whose infinity type is the conjugate of the infinity type of $\xi$.  Let $Q$ be as in
(\ref{GLndecomp}). Let $f,g \in \mM'_{n,k,\nu}(\mK)$ for some open compact subgroup $\mK$ of $U_n(\AQf)$. Then by Proposition \ref{adelicclassical} the forms $f$ and $g$ correspond to $\# Q$-tuples of
functions on $\bfH_n$ which we denote by $(f_{p_q})$ and $(g_{p_q})$
respectively or simply by $(f_{q})$ and $(g_{q})$. If either $f$ or $g$ is cuspidal, set $$\left< f,g
\right> = (\# Q)^{-1} \sum_{q \in Q} \left<
f_{q},g_q \right>,$$ (compare with \cite{Shimura97}, (10.9.6)), and $$\left< f,g
\right>_{\Gamma} = (\# Q)^{-1} \sum_{q \in Q} \left<
f_{q},g_q \right>_{\Gamma}$$ if for all $q \in Q$ the integrand $f_q(Z) \ov{g_q(Z)}
\det (\textup{Im}(Z))^k dZ$ is $\Gamma$-invariant (for example if $f, g \in \mM'(\mK_{0,n}(N))$ and all $q \in Q$ are scalars, then one can take $\Gamma = \Gamma^{\rm h}_{0,n}(N)$).

Let $k$ be a positive integer. Fix a Hecke
character $\xi$ of $K$ with conductor $\ff_{\xi}$ and
infinity type $|x|^t x^{-t}$ for an integer $t$ about which we for now
only assume that $t\geq-k$. We will now define a theta series
associated with the character $\xi$. Let $\lambda: M_n(\AKf) 
\rightarrow \bfC$ be a function given by
$$\lambda(g) = \begin{cases} \xi_{\ff_{\xi}}(\det g) & \textup{if} \hs  
g\in M_n(\hat{\Oo}_K) \hs \textup{and} \hs g_v \in \GL_n(\Oo_v) \hs
\textup{for all} \hs v \mid \ff_{\xi}\\
0 & \textup{otherwise}.\end{cases}$$
The map $\lambda$ is a Schwarz function (cf. \cite{Shimura00}, section A5 for a more precise statement). 
 Fix  $\tau \in  \mS$.
For $Z \in \bfH_n$ set $$\theta_{\xi}(Z,\lambda) = \sum_{\alpha \in M_n(K)} \lambda (\alpha)
\ov{\det \alpha} \cdot e(\tr (\alpha^* \tau \alpha Z)).$$
For $g \in U_n(\AQ)$ set $$\theta_{\xi}(g)=j(g,i)^{-l} \theta (gi,
\lambda^g), $$ where $l=t+k+n$  and the automorphism of the space of Schwarz functions on $M_n(\AKf)$ given by $\lambda \mapsto \lambda^g$ is defined in Theorem A5.4 of \cite{Shimura00}.
\begin{rem} \label{choice112} Note that $\theta_{\xi}$ depends on the
choice of the matrix
$\tau$. If $\{g^* \tau g \}_{g \in \OK^2} =\bfZ$ and $c$ is a positive integer such that $\{g^* \tau^{-1} g\}_{g \in \OK^2} 
\subset
\frac{1}{c} \bfZ$, then $\theta_{\xi}\in
\mM_{l}(N, \psi')$, where $N=D_K c N_{K/\bfQ}(\ff_{\xi})$ 
by \cite{Shimura00}, section A5.5 and
\cite{Shimura97}, Proposition A7.16. Note that such a $c$ always exists (for example one can take $c=\det \tau$). In what follows we fix $\tau$ and $c$ so that $\{g^* \tau g \}_{g \in \OK^2} =\bfZ$
and we fix $N$ as above. In particular, we have $\ff_{\xi} \mid N$. \end{rem}

 By \cite{Shimura00}, (22.14b),
$\psi' = \xi^{-1}
\varphi^{-n}$, where $\varphi$ is a Hecke character of $K$ with infinity type $\frac{|a_{\infty}|}{a_{\infty}}$ and such that $\varphi|_{\AQ}^{\times} = \chi_K$ (such a character always exists, but is not unique - cf. \cite{Shimura00}, Lemma A.5.1). Thus $$\psi'_{\iy}(x) = x^{t+n} |x|^{-t-n}.$$ Let $Q=\mB$ be 
as
in Corollary \ref{scalarcor}. (In fact, we believe our
result holds for a more general $Q$, but for simplicity we proceed with  
$Q$ as in that corollary.) Then $\theta_{\xi}$ corresponds to a $\#
Q$-tuple of functions, which we denote following \cite{Shimura00} by
$(\theta_{\chi,p_q})$ or simply by $(\theta_{\chi,q})$.

Set $m=-t-n$. Let $\gamma \in \G_{0,n}^{\rm h}(N)$. Note that if $\psi$ is a Hecke
character of $K$, then \be \label{c1} \psi_N(\det a_{\gamma}) =
\psi_N(\ov{\det d_{\gamma}}^{-1})
= \psi^c_N(\det d_{\gamma})^{-1}.\ee Note that this makes sense because
$N \in \bfZ$. Let $f\in \mM_{n,k}(N)$. Then $f$ corresponds to a $\# Q$-tuple of functions $(f_q)$. We have \begin{multline} \left< E(\cdot,s,m, \G_{1,n}^{\rm h}(N)) \theta_{\chi,q}
, f_q
\right>_{\G_{1,n}^{\rm h}(N)} = \\
=\int_{\G_{0,n}^{\rm h}(N) \setminus \mH} \theta_{\chi,q}(Z)\left( \sum_{\g
\in
\G_{1,n}^{\rm h}(N) \setminus \G_{0,n}^{\rm h}(N)} \psi'_N(\det a_{\g}) E(Z,s,m,\G_{1,n}^{\rm h}(N))|_m \g
\right) \ov{f_q(Z)} \delta(Z)^k dZ\end{multline} By Lemma 
\ref{allhecke3}
one has
\begin{multline} \label{aux430} \sum_{\g \in
\G_{1,n}^{\rm h}(N) \setminus \G_{0,n}^{\rm h}(N)}\psi'_N(\det a_{\g}) E(Z,s,m,\G_{1,n}^{\rm h}(N))|_m \g =   
\\
(\# X)^{-1} \sum_{\psi \in X} \sum_{\g \in
\G_{1,n}^{\rm h}(N) \setminus \G_{0,n}^{\rm h}(N)} \psi_{\tuf}
(\det
q^*)^{-1} |\det (qq^*)|_{\bfQ}^{-s}
E_q(Z,s,m,\psi,N)|_{\gamma}. \end{multline}

Note that for $Z=g_{\iy} \bfi_n$ with $g=(g_{\iy},1)$, we have
$$ E_q(Z,s,m,\psi,N)
=
j(g_{\iy}, \bfi_n)^m E(p_qg,s,N,m,\psi),$$ and
\be \begin{split}E_q(Z,s,m,\psi,N)|_m \g & =
j(\g,
Z)^{-m} E_q(\g Z, s,m,\psi,N) \\
&= j(g_{\iy}, \bfi_n)^m j(\g,
Z)^{-m} E(p_q(\gamma g)_{\iy},s,N,m,\psi) \\
&=  j(g_{\iy}, \bfi_n)^m j(\g,
Z)^{-m} E((\gamma g_{\iy}, p_q), s,N,m,\psi)\\
& = j(g_{\iy}, \bfi_n)^m j(\g,
Z)^{-m} E((g_{\iy}, p_q \gamma^{-1}),s,N,m,\psi)\end{split}\ee where we
have
used the assumption that $p_q$ is a
scalar. Then by
\cite{Shimura97}, (18.6.2), we have \begin{multline*} E((g_{\iy}, p_q
\gamma^{-1}),s,N,m,\psi) =\\
 = \psi_{N}(\det d_{\gamma^{-1}})^{-1} E((g_{\iy}, p_q),s,N,m,\psi) =
\psi_{N}(\det
d_{\gamma^{-1}})^{-1}E(p_q g,s,N,m,\psi).\end{multline*} Note that 
$$\det
d_{\gamma^{-1}} =
\det d_{\gamma}^{-1} \hf \textup{mod} \hs N.$$ Hence finally we get \be
\label{eistrans} E_q(Z,s,m,\psi,N)|_m \g =
\psi_{N}(\det
d_{\gamma})E_q(Z,s,m,\psi,N).\ee Then (\ref{aux430}) equals
\begin{multline} (\# X)^{-1} \sum_{\psi \in X} \psi_{\tuf}
(\det
q^*)^{-1} |\det (qq^*)|_{\bfQ}^{-s} E_q(Z,s,m,\psi,N)
\times \\
\times \sum_{\g \in
\G_{1,n}^{\rm h}(N) \setminus \G_{0,n}^{\rm h}(N)} \psi'_N(\det a_{\g}) \psi_N(\det
d_{\g}) =\\
= (\# X)^{-1} \sum_{\psi \in X} \psi_{\tuf}
(\det
q^*)^{-1} |\det (qq^*)|_{\bfQ}^{-s} E_q(Z,s,m,\psi,N)\times \\
\times \sum_{\g \in
\G_{1,n}^{\rm h}(N) \setminus \G_{0,n}^{\rm h}(N)} (\psi')^c_N(\det d_{\gamma})^{-1} \psi_N(\det
d_{\g}),
\end{multline}
where the last equality follows from (\ref{c1}), according to which
$\psi'_N(\det
a_{\g}) = (\psi')^c_N(\det d_{\gamma})^{-1}$. Using the fact
that 
$$\sum_{\g \in
\G_{1,n}^{\rm h}(N) \setminus \G_{0,n}^{\rm h}(N)} (\psi')^c_N(\det d_{\gamma})^{-1} \psi_N(\det
d_{\g})=0$$ unless $\psi = (\psi')^c$, we obtain
\begin{multline} \sum_{\g
\in
\G_{1,n}^{\rm h}(N) \setminus \G_{0,n}^{\rm h}(N)} \psi'_N(\det a_{\g}) E(Z,s,m,\G_{1,n}^{\rm h}(N))|_m \g
= \\
= (\# X)^{-1}[\G_{0,n}^{\rm h}(N): \G_{1,n}^{\rm h}(N)] (\psi')^c   
(\det
q^*)^{-1} |\det (qq^*)|_{\bfQ}^{-s} E_q(Z,s,m,\psi,N).\end{multline} 
Hence
finally
\begin{multline} \label{inner111}\left<
E(\cdot,s,m, \G_{1,n}^{\rm h}(N)) \theta_{\chi,q}   
, f_q
\right>_{\G_{1,n}^{\rm h}(N)} \\
= (\# X)^{-1}[\G_{0,n}^{\rm h}(N): \G_{1,n}^{\rm h}(N)](\psi')^{-1}
(\det
q) |\det (qq^*)|_{\bfQ}^{-s} \left<
E_q(\cdot,s,m, ((\psi')^c),N) \theta_{\chi,q}
, f_q
\right>_{\G_{0,n}^{\rm h}(N)}.\end{multline}
Note that the inner product (\ref{inner111}) makes sense, because first 
of
all $(\psi')^c_{\iy}(x) = x^{-t-n}|x|^{t+n} = x^m |x|^m$, so the
definition of $E(Z,s,m, (\psi')^c,\G_{1,n}^{\rm h}(N))$ makes sense, and secondly,   
\begin{equation}\begin{split} E(Z,s,m, (\psi')^c,N)|_{\g} &= (\psi')^c_{N}(\det
d_{\gamma^{-1}})^{-1} E_q(Z,s,m,(\psi')^c,N) \\
&=(\psi')^c_N(\det d_{\g})
E_q(Z,s,m,(\psi')^c,N) \\
&= (\psi')^{-1}_N(\det a_{\g})E_q(Z,s,m,(\psi')^c,N)  , \end{split}\end{equation} where 
the
first equality
follows from (\ref{eistrans}) and the last one from (\ref{c1}). Hence
$(E_q(Z,s,m, (\psi')^c,N) \theta_{\chi,q})|_k \g = E_q(Z,s,m, (\psi')^c,N) \theta_{\chi,q}$ for
every $\g \in \G_{0,n}^{\rm h}(N)$.

Set $\G:= \G_{1,n}^{\rm h}(N)\cap SU_n(\bfQ)$. We now relate $\left<
E(\cdot,s,m, \G_{1,n}^{\rm h}(N)) \theta_{\chi,q}
, f_q
\right>_{\G_{1,n}^{\rm h}(N)}$ to $\left<
E(\cdot,s,m, \G) \theta_{\chi,q}
, f_q
\right>_{\G}$. By \cite{Shimura00}, formula (17.5) and Remark 17.12(2), 
we
have
$$E(Z,s,m,\G_{1,n}^{\rm h}(N))=\frac{1}{[\G_{1,n}^{\rm h}(N):\G]} \sum_{\a \in
\G \setminus 
\G_{1,n}^{\rm h}(N)}
E(Z,s,m,\G)|_m \a.$$
Hence
\begin{multline}\left<E(\cdot, s,m,\G_{1,n}^{\rm h}(N)) \theta_{\chi,q},
f_q\right>_{\G_{1,n}^{\rm h}(N)}  = [\G_{1,n}^{\rm h}(N): \G]^{-1} \hs  
\left<E(\cdot,
s,m,\G_{1,n}^{\rm h}(N))\theta_{\chi,q},
f_q\right>_{\G} =\\
= [\G_{1,n}^{\rm h}(N):\G]^{-2}\left<\left(\sum_{\a \in \G \setminus
\G_{1,n}^{\rm h}(N)}
E(Z,s,m,\G)|_m \a \right)
\theta_{\chi,q},
f_q\right>_{\G}.\end{multline}
Since $\theta_{\chi,q}|_l \a = \theta_{\chi,q}$ and $f_q|_k \a =f_q$
for $\a \in \G_{1,n}^{\rm h}(N)$ we finally have
\begin{multline} \label{gammagammaprim} \left<E(\cdot, s,m,\G_{1,n}^{\rm h}(N))
\theta_{\chi,q},
f_q\right>_{\G_{1,n}^{\rm h}(N)} = \\
=\frac{1}{[\G_{1,n}^{\rm h}(N): \G]^2} \hs \sum_{\a \in \G
\setminus \G_{1,n}^{\rm h}(N)}
\left<\left(E(\cdot,
s,m,\G)|_m \a \right)
\left(\theta_{\chi,q}|_l \a\right),
f_q|_k \a\right>_{\G} =\\
= \frac{1}{[\G_{1,n}^{\rm h}(N):\G]} \hs \left<E(\cdot, s,m,\G)\theta_{\chi,q},
f_q\right>_{\G}. \end{multline}

\subsection{The standard $L$-function} \label{The standard L-function}

Let $Q$ and $f$ be as before and let $q \in Q$. From now on we assume that $f$ is a Hecke eigenform. Let $D(s,f,\xi)$ and
$D_q(s,f,\theta_{\xi})$
denote the Dirichlet series
defined in
\cite{Shimura00} by formulas (22.11) and (22.4) respectively. Let $r
\in \GL_n(\AKf)$ and $\tau \in \mS^+:= \{ h \in \mS \mid h>0\}$. 
 Then (22.18b)
in [loc. cit.] gives \be\label{form71} D(s+3n/2, f, \xi) = (\det
\tau)^{s+(k+l)/2}
|\det
r|_K^{-s-n/2} \sum_{q \in Q} (\psi')^{-1}(\det q) |\det qq^*|^s_{\bfQ}
D_q(s,f,\theta_{\chi}),\ee while \cite{Shimura00}, (22.9) gives  \be
\label{form72} D_q(s,f,\theta_{\xi}) =A_N \Gamma((s))^{-1} \left< f_q , \theta_{\chi,q} E(\cdot,
\ov{s}+n, m, \G) \right>_{\G},\ee where $A_N$ and $\Gamma((s))$ are defined as follows. Let $X_{\rm re} = \{h\in M_n(\bfC)\mid h=h^*\}/ \{h\in M_n(\OK)\mid h=h^*\}$ and $X_{\rm im} =  \{h\in M_n(\bfC)\mid h=h^*, h>0\}/\sim$, where $h \sim h'$ if there exists $g \in \GL_n(\OK)$ such that $h'=ghg^*$. Then $X_{\rm re}\times X_{\rm im}$ is commensurable with $\Gamma_{0,n}^{\rm h}(N) \cap P(\bfQ)\setminus \bfH_n$ i.e, the ratio of their volumes is a positive rational number (\cite{Shimura00}, p. 179). We set $A_N$ to be this rational number times the $\textup{vol}(X_{\rm re})^{-1}$. Note that $A_N \in \bfQ$. We also set (cf. \cite{Shimura00}, p.179 and formulas (22.4a), (16.47))
 $$\Gamma((s))= (4\pi)^{-\frac{n}{2}(2s+k+l)} \pi^{2n(n-1)/4}\prod_{i=0}^{n-1} \Gamma(s-i).$$
Combining (\ref{gammagammaprim}) with (\ref{inner111}), we obtain
\bmls  \left< f_q , \theta_{\chi,q} E(\cdot,s, m, \G) \right>_{\G} = \\
(\# X)^{-1} [\G_{0,n}^{\rm h}(N):\G]\ov{(\psi')^{-1}
(\det
q)} |\det (qq^*)|_{\bfQ}^{-\ov{s}} \left<f_q,
E_q(\cdot,s,m, (\psi')^c,N) \theta_{\chi,q}
\right>_{\G_{0,n}^{\rm h}(N)}=\\
(\# X)^{-1} [\G_{0,n}^{\rm h}(N):\G]\psi'
(\det
q) |\det (qq^*)|_{\bfQ}^{-\ov{s}} \left<f_q,
E_q(\cdot,s,m, (\psi')^c,N) \theta_{\chi,q}
\right>_{\G_{0,n}^{\rm h}(N)}.\end{multline*} Hence \bml \label{form73} \left< f_q ,
\theta_{\chi,q}
E(\cdot,\ov{s}+n, m,
\G) \right>_{\G} = \\
(\# X)^{-1} [\G_{0,n}^{\rm h}(N):\G]\psi'
(\det
q) |\det (qq^*)|_{\bfQ}^{-s-n} \left<f_q,   
E_q(\cdot,\ov{s}+n,m, (\psi')^c,N) \theta_{\chi,q}
\right>_{\G_{0,n}^{\rm h}(N)}.
 \end{multline}
Combining (\ref{form71}), (\ref{form72}) and (\ref{form73}) we finally
obtain \begin{multline} \label{inner222} D(s+3n/2, f, \xi) = A_N (\#
X)^{-1} [\G_{0,n}^{\rm h}(N): \G] \G((s))^{-1} (\det
\tau)^{s+(k+l)/2}   
|\det
r|_K^{-s-n/2} \times \\
\times \sum_{q \in Q} |\det qq^*|_{\bfQ}^{-n} \left<f_q,
E_q(\cdot,\ov{s}+n,m, (\psi')^c,N) \theta_{\chi,q}
\right>_{\G_{0,n}^{\rm h}(N)}.\end{multline} By our assumption $\det qq^*=1$, so we
get \be \label{inner223} \bs D(s+3n/2, f, \xi) = & A_N (\#
X)^{-1} [\G_{0,n}^{\rm h}(N): \G] \G((s))^{-1} (\det
\tau)^{s+(k+l)/2}
|\det
r|_K^{-s-n/2}  \\
& \times \sum_{q \in Q} \left<f_q,
E_q(\cdot,\ov{s}+n,m, (\psi')^c,N) \theta_{\chi,q}
\right>_{\G_{0,n}^{\rm h}(N)}= \\
=& A_N (\#
X)^{-1} [\G_{0,n}^{\rm h}(N): \G] \G((s))^{-1} (\det
\tau)^{s+(k+l)/2}
|\det
r|_K^{-s-n/2} \\
&\times \left<f,
E(\cdot,\ov{s}+n,m, (\psi')^c,N) \theta_{\xi}
\right>_{\G_{0,n}^{\rm h}(N)}.\end{split}\ee

We now relate $D(s,f,\xi)$ to the standard $L$-function $L_{\rm st}(f, s, \psi)=Z(s,f,\psi)$ of $f$ as
defined in \cite{Shimura00}, section 20.6. In this section we assume
that $(D_K, \ff_{\xi})=1$. (This assumption is not strictly necessary, but our
formulas complicate if we do not make it.) 
Let $\xi$, $\varphi$, $\psi'$ be Hecke characters
of $K$ as we defined them above. Let $N=\bfZ \cap (\cond \psi')$.

One has (by \cite{Shimura00}, formula (22.19)) \be \label{diri64} D(s,f,\xi) = \frac{c_{f_r}(\tau)
Z(s,f,\xi)}{\prod_{p \in \mathbf{b}} g_p(\xi^*(p \OK) p^{-2s})
\prod_{j=1}^n L(2s-n-j+1, \chi_K^{n+j-1}\xi_{\bfQ}^*)}.\ee Here
$c_{f_r}(\tau)$ is the $\tau$-Fourier coefficient of $f_r$, 
and $L(s,\psi)$ is the usual Dirichlet $L$-function of $\psi$ while $\prod_{p \in
\mathbf{b}} g_p(\xi^*(p \OK) p^{-2s})$ is defined in \cite{Shimura00},
Lemma 20.5. Note that the definition of Fourier coefficients in \cite{Shimura00} (cf. Proposition 20.2 in [loc.cit.]) differs from ours and from the one given in (18.6.6) in \cite{Shimura97} (which in turn agrees with ours) by the factor $e^{-2 \pi \tr \tau}$ (so, $c_{f}(\tau, r)$ in (22.19) of \cite{Shimura00} agrees with our $c_{f_r}(\tau)$).

Combining (\ref{diri64}) with (\ref{inner223}) we obtain

\be \label{inner224} \left<f,
E(\cdot,s,m, (\psi')^c,N) \theta_{\xi}
\right>_{\G_{0,n}^{\rm h}(N)}
= C(s) \frac{
Z(\ov{s}+n/2,f,\xi)}{
\prod_{j=1}^n L(2\ov{s}-j+1, \chi_K^{n+j-1}\xi_{\bfQ}^*)},
\ee where \be \label{formulaforC}C(s) = \frac{ (\# X) \G((\ov{s}-n)) (\det
\tau)^{-\ov{s}+n -(k+l)/2} |\det r|_K^{\ov{s}-n/2}
c_{f_r}(\tau)}{A_N[\G_{0,n}^{\rm h}(N): \G]\prod_{p \in \mathbf{b}} g_p(\xi^*(p \OK)  
p^{-2\ov{s}-n})}.\ee 
Following \cite{Shimura00} (17.24), we define $$ D(g,s,N,m,\psi):=
E(g,s,N,m,\psi) \prod_{j=1}^{n}
L(2s-j+1, \psi_{\bfQ}
\chi_K^{j-1}).$$
 Using (\ref{inner224}) we finally  get 
 \be \label{inner225} \left<
D(\cdot,s,m, (\psi')^c,N) \theta_{\xi}, f
\right>_{\G_{0,n}^{\rm h}(N)}
= \ov{C(s)}
\ov{Z(\ov{s}+n/2,f,\xi)}.
\ee

We record this as a theorem.
\begin{thm} \label{ransel} Assume $(h_K, 2n)=1$. Let $f \in \mM_{n,k}(N)$ be a Hecke eigenform. Let $\xi$, $\psi'$ be as above. Then 
$$ \left<
D(\cdot,s,m, (\psi')^c,N) \theta_{\xi}, f
\right>_{\G_{0,n}^{\rm h}(N)}
= \ov{C(s)} \cdot 
\ov{L_{\rm st}(f, \ov{s}+n/2,\xi)}$$ with $C(s)$ defined by (\ref{formulaforC}). \end{thm}

\subsection{$\ell$-integrality of Fourier coefficients of Eisenstein series
and theta series} \label{ellint} 

Let $D(g,s,N,m,\psi)$ be as above. Fix $r \in
\GL_n(\AKf)$.
As for the Eisenstein series $E$, we define $$D_r(Z,s,m,\psi,N)=D_{p_r}(Z,s,m,\psi,N) =
j(g,\bfi_n)^{m} D(p_rg,s,N,m,\psi),$$ where $g \in U_n(\bfR)$ and $Z=g \bfi_n$

\begin{thm} Let $r \in \GL_n(\AKf)$.
Then $D_r(Z, n-m/2, N,m,\psi)$ is a holomorphic function of $Z$.
\end{thm}

\begin{proof} This follows from Theorem 17.12(iii) in \cite{Shimura00}.
\end{proof}

For a function $A: U_n(\AQ) \rightarrow \bfC$ set $A^*(g) = A(g
\eta_{\rm f}^{-1})$, where $\eta_{\rm f} \in U_n(\AQ)$ is a matrix with trivial infinity
component and all finite components equal to the matrix $J=\bsmat & -I_n \\
I_n\esmat$.

We write
$$D^*\left( \bmat q & \sigma q \\ & \hat{q} \emat ,n-m/2,N,m,\psi\right)
=
\sum_{h \in S} c(h,q) e_{\AQ}(h\sigma).$$

\begin{thm} \label{fce} Suppose we take $q=(y^{1/2}, q_1)$, where $y^{1/2}$ (resp. $q_1$) denotes the infinite (resp. finite) component of $q$. One has
\be\begin{split} c(h,q) &= (*) e^{-2\pi \tr((q
q^*)_{\infty} h)} \det(qq^*)_{\infty}^{m/2} \psi(\det q_1) |\det
q_1|_K^{m/2} \times \\
& N^{-n^2} \Phi \pi^{n(n+1)/2} \cdot \frac{\prod_{i=0}^{n-1-\turk
(h)} L_{N}(n-m-i, \psi, \chi_K^{n+i-1})}
{\prod_{i=0}^{n-1} \Gamma(n-i)
},\end{split}\ee
where $\Phi
$ is an algebraic integer and $(*)$ is an $\ell$-adic unit. If $r<1$ we 
set
$\prod_{j=0}^r =
1$.\end{thm}
\begin{proof} The theorem follows from Propositions 18.14 and
19.2 in \cite{Shimura97}, combined with Lemma 18.7 of
\cite{Shimura97} and formulas (4.34K) and (4.35K) in
\cite{Shimura82}. It is a long but straightforward
calculation. \end{proof}

\begin{definition} \label{admis}
Let $\mB$ be a base. We will say that $\mB$ is \emph{admissible} if
\begin{itemize}
\item All $b \in \mB$ are scalar matrices with $bb^*=I_n$, which implies
$p_b$ are scalar matrices with $p_b p_b^*=I_{2n}$;
\item For every $b \in \mB$ there exists a rational prime $p\nmid
2D_K \ell$
such that $b_{\fq}=I_n$ for all $\fq \nmid p$ and $b_{\infty}=I_n$.
\end{itemize} The set of primes $p$ for which $b_{\fq}\neq I_n$ with $\fq \mid p$ (or by a slight abuse of
terminology the product of such primes) will be called the
\emph{support} of $\mB$.
\end{definition}

\begin{rem} If $(h_K, 2n)=1$ it follows from Corollary \ref{scalarcor}
together with the Tchebotarev Density Theorem that an admissible base
exists. \end{rem}

For $r \in \GL_2(\AKf)$ write $D^*_r(Z)$
for $D^*_r(Z,n-m/2, ,N,m,\psi)$.

\begin{cor} \label{intE}  Assume $\ell \nmid
N(n-2)!$. If
$\mB$ is an admissible base whose support is relatively prime to $\cond 
\psi$, then
for every $h \in S$ and for every $b \in \mB$, the product
$$\pi^{-n(n+1)/2}c_{h,b}$$ lies in the ring of integers of a finite
extension
of $\bfQ_{\ell}$. Here $c_{h,b}$ stands for the $h$-Fourier
coefficient of $D_b(Z)$.\end{cor}

\begin{proof} Let $\Oo$ be the ring of integers in some sufficiently
large finite extension of $\bfQ_{\ell}$. For any admissible base $\mB$,
one has by Theorem \ref{fce} $$c_h = \psi( \det
b)\pi^{n(n+1)/2}\cdot \prod_{i=0}^{n-1-\turk
(h)} L_{N}(n-m-i, \psi, \chi_K^{n+i-1}) \cdot x,$$ where $x \in \Oo$. So,
the corollary follows from the fact that $\psi(\det b) \in \Oo^{\times}$ upon noting that 
for
every
Dirichlet character $\psi'$ of conductor dividing $N$ and every $n \in
\bfZ_{<0}$, one has $L(n, \psi') \in \bfZ_{\ell}[\psi']$ (by a simple
argument using \cite{Washingtonbook}, Corollary 5.13) and
$(1-\psi'(p)p^{-n}) \in \bfZ_{\ell}[\psi']$ for every $p \mid N$.
\end{proof}

We now turn to the theta series. 
First note that $\xi_{\ff_{\xi}}( \det
g)^r =1$ for a sufficiently large integer $r$ (because
$\xi_{\ff_{\xi}}$ is a character of finite order). So
$\lambda(\alpha)\neq 0$ only if $\alpha \in M_n(K) \cap
M_n(\hat{\Oo}_K) = M_n(\OK)$ and for such $\alpha$ one has $\lambda(\alpha)
\ov{\det \alpha} \in \Oo$ (more precisely  $\lambda(\alpha)$ is a root of
unity in $\Oo$ and $\det \alpha \in \OK$).

Fix $r,\tau$ as in section \ref{The standard L-function}. Let $q \in \GL_n(\AKf)$.
\begin{prop} Assume $q_v = r_v = I_n$ for every $v \mid \ff_{\xi}$.
Write $$A(\sigma, q,r)=\{ \alpha \in M_n(K) \cap r M_n(\hat{\Oo}_K) q^{-1}
\mid \alpha^* \tau \alpha =\sigma, \hs \alpha_v \in \GL_n(\OKv) \hs \textup{for
all} \hs v \mid \ff_{\xi}\}.$$ Then the Fourier coefficient
$c_{\theta_{\xi}}(\sigma, q)=0$ if $\det \sigma=0$. If $\det \sigma
\neq 0$, one gets $$c_{\theta_{\xi}}(\sigma, q)=e^{-2 \pi \tr \sigma}|\det q|_K^{n/2}
\varphi^n(\det q) \xi_{\ff_{\xi}}(\det q) \xi (\det r) \sum_{\alpha \in
A(\sigma, q,r)} \xi_{\ff_{\xi}}(\det \alpha) \ov{\det \alpha},$$ where $|\cdot|_K$ denotes the idele norm on $\AK^{\times}$ (cf. \cite{Shimura00}, p.180). \end{prop}

\begin{proof} The Fourier coefficient of $\theta_{\xi}(g)$ is computed
in section A5 of \cite{Shimura00} (formula (A.5.11)), where we have     
$\omega' = \xi^{-1} \varphi^n$. Our formula is slightly reordered and
simplified due to the assumptions we imposed, but this is an easy
calculation. Again note the discrepancy in the definitions of Fourier coefficients between us and \cite{Shimura00} pointed out after formula (\ref{diri64}). \end{proof}

\begin{cor} \label{inttheta} Assume $r_v=I_n$ for all $v \mid \ff_{\xi}
D_K \ell$. For
an admissible base $\mB$ we have that
$c_{(\theta_{\xi})_b}(\sigma) = e^{2\pi \tr \sigma} c_{\theta_{\xi}}(\sigma,b) \in \Oo$ for all $b \in\mB$ and
all $\sigma \in \mS$. \end{cor}

\begin{proof} Let us fix an admissible $\mB$. Then we automatically have
$|\det b|_K=1$. Since $\cond \varphi \mid D_K$, we see that $\varphi(\det b)
\in \Oo^{\times}$. Similarly we get $\xi_{\ff_{\xi}}(\det b)=1$ and    
$\xi (\det b) \in \Oo^{\times}$. Since $\xi_{\ff_{\xi}}$ is of finite
order, we see that $\xi_{\ff_{\xi}}(\det \alpha) \in \Oo^{\times}$ for any matrix $\alpha$ as
remarked above. Finally, since $r_v=b_v=I_n$ for all $v \mid \ell$, we
get that for such $v$ and $\alpha \in A(\sigma, b,r)$ one has  $\alpha_v \in M_n(\OKv)$. Since $\alpha \in M_n(K)$, we 
get that $\alpha \in M_n(\OKv)$, so $\det \alpha \in \Oo$. Since $c_{(\theta_{\xi})_b}(\sigma) = e^{2\pi \tr \sigma} c_{\theta_{\xi}}(\sigma,b)$ by (\ref{f12}), we are done. \end{proof}

\section{Congruence}\label{Congruence}

In this section we set $n=2$ and write $U=U_2$. Let $K$ be an imaginary
quadratic field of discriminant $-D_K$, which we assume to be prime.
It is a well-known fact that this implies that the class number of
$K$ is odd. Then
the space $S_{k-1}(D_K,
\chi_K)$ has a (unique) basis of newforms, which we, as before,  denote by $\mN$. 
We fix the following set of data: \begin{itemize}
\item a positive even integer $k$ divisible by $\# \OK^{\times}$;
\item a rational prime $\ell> k$ such that $\ell \nmid  h_K D_K$;
\item $\phi \in \mN$, which is ordinary at $\ell$ such that 
$\ov{\rho}_{\phi}|_{G_K}$ is absolutely irreducible;
\item $\chi \in \Hom(\Cl_K, \bfC^{\times})$ and write $f_{\phi,\chi}$ for
the Maass lift of $\phi$ lying in the space $\mS^{\chi}_{k, -k/2}$;
\item $\xi$ a Hecke character of $K$ of $\infty$-type $z^t |z|^{-t}$ for
an integer $-6 > t \geq -k$; we write $\ff_{\xi}$ for the conductor of
$\xi$ and set $N=D_K h_K N_{K/\bfQ}(\ff_{\xi})$;
\item $\beta \in \textup{Char}(k/2)$;
\item an admissible base $\mB$ whose support is prime to $N$.
\end{itemize}

Let $E$ be a finite extension of $\bfQ_{\ell}$, which we will always assume to be sufficiently large to contain any (finite number of) number fields that we encounter. Write $\Oo$ for its valuation ring. We also fix a uniformizer $\varpi \in \Oo$. 

\begin{lemma} Let $b \in \mB$ and $\tau \in \mS$.  Write $c_{f_{\phi,\chi}}(\tau,b)$ for the $(\tau,b)$-Fourier coefficient of $f_{\phi,\chi}$ Then $e^{2\pi \tr \tau} c_{f_{\phi,\chi}}(\tau,b) 
\in \Oo$ for arbitrary $\tau \in S$. \end{lemma}
\begin{proof} It is a standard fact that the Fourier coefficients of $\phi$ 
are algebraic integers (the field which they all generate is a finite 
extension of $\bfQ$), 
so 
by our 
assumption that $\Oo$ be sufficiently 
large we may assume that they lie in $\Oo$. Then the Lemma follows 
from (\ref{Maass condition}) and the 
formula in 
Theorem \ref{desc4}. \end{proof}

\begin{definition} \label{cord1} Let $\tau \in \mS$ and $b \in \mB$. We will
call $(\tau, b)$ an \emph{ordinary pair} if the following two conditions
are simultaneously satisfied:
\begin{enumerate}
\item $\val_{\ell}(\ov{e^{2 \pi \tr(\tau)}c_{f_{\phi,\chi}}(\tau, b)})=0$;
\item $(\det \tau, N)=1$.
\end{enumerate} \end{definition}

\begin{lemma} An ordinary pair exists. \end{lemma}
\begin{proof} This is proved like Lemma 7.10 in \cite{Klosin09}. Note that
the prime $p_0$ chosen in  the proof of that lemma can be taken to be
arbitrarily large because its existence is guaranteed by the Tchebotarev
Density Theorem. Hence condition (2) can also be satisfied. \end{proof}

\begin{definition} Let $f$ and $g$ be two hermitian modular forms.
We say that $f$ is congruent to $g$ modulo $\varpi^n$, a
property which we denote by $f \equiv g$ (mod $\varpi^n$) if there exists a
base $\mB$ such that $f_b$ and $g_b$ both have Fourier coefficients lying
in $\Oo$ for all $b \in \mB$ and $f_b \equiv g_b$ (mod $\varpi^n$) for all
$b \in \mB$. The latter congruence means that $\varpi^n \mid (c_{f_b}(h)-c_{g_b}(h))$ for all $h \in \mS$. \end{definition}

\begin{rem} Note that if $f$ and $g$ are congruent (mod $\varpi^n$)
with respect to one admissible base, say $\mB$, and $\mB'$ is another admissible base such that
$f_b$ and $g_b$ both have Fourier coefficients lying
in $\Oo$ for all $b \in \mB'$, then $f_b \equiv g_b$ (mod $\varpi^n$) for
all
$b \in \mB'$. Indeed this follows from admissibility of $\mB$ and $\mB'$ and formula (\ref{f13}). \end{rem}

\subsection{The inner product ratio} \label{The inner product ratio}

 Let
$\Psi_{\beta}: \mM_{k} \xrightarrow{\sim} \mM_{k,-k/2}$ be the isomorphism
defined in Proposition \ref{newmap}. Denote the restriction of $f_{\phi,\chi}$ to $U(\AQ)$ again by $f_{\phi,\chi}$ (cf. Proposition \ref{isom99}).
Fix $\tau \in \mS$ and assume 
that the theta series $\theta_{\xi}$ was defined using that $\tau$ (cf. section \ref{Theta series and 
inner
products}). For 
the moment 
we will not 
assume that we are working with an ordinary pair to obtain a more 
general formula for the ratio of the inner products. 
Set $c=\det \tau$. Set $\psi' = \xi^{-1} \varphi^{-2}$ with
$\varphi$ as in section \ref{Theta series and 
inner
products}. Then
$\theta_{\xi} \in \mM_l(Nc, \psi')$, where $l=t+k+2$ (see Remark 
\ref{choice112}). Set $m=k-l$. 
To shorten notation write
$D(g) := D(g,2-m/2,Nc,m,(\psi')^c)$ and $D^*(g):= D(g \eta_{\rm f}^{-1},
2-m/2, Nc,
m, (\psi')^c)$.  

Since $D(g) \in \mM_m(Nc, (\psi')^{-1})$ and $\theta_{\xi} \in \mM_l(Nc,
\psi')$, we get $D\theta_{\xi}\in \mM_k(\mK_0(Nc))$ and $D^*\theta^*_{\xi} \in \mM_k(\eta_{\rm f}^{-1}\mK_0(Nc)\eta_{\rm f})$. 

For $F \in M_{k,-k/2}^{\tuh}(J^{-1}\Gamma_0^{\rm h}(Nc) J)$ define the
\emph{trace operator} $ \tr : M_{k,-k/2}^{\tuh}(J^{-1}\Gamma_0^{\rm h}(Nc) J)
\rightarrow M_{k, -k/2}^{\tuh}(U(\bfZ))$ by
$$\tr F = \sum_{\gamma \in J^{-1} \Gamma_0^{\rm h}(Nc) J \setminus 
U(\bfZ)} 
F|_k
\gamma.$$ Note that if $F$ has $\ell$-integral Fourier coefficients, then
so does $\tr F$ by the $q$-expansion principle (Theorem \ref{qexpansion12}). The form $D^* \theta_{\xi}^*$ corresponds to $\# \mB$ forms $(D^*\theta_{\xi}^*)_{p_b} \in M^{\rm h}_k(J^{-1}\Gamma_0^{\rm h}(Nc)J)$, $b \in \mB$. This way we can define $\tr(D^*\theta^*_{\xi})$.
Since both $\Psi_{\beta}(\tr (D^*\theta^*_{\xi}))$ and
$f_{\phi,\chi}$
are elements of the finite-dimensional $\bfC$-vector space 
$\mM_{k, -k/2}(\mK_0(Nc))$ 
we
can write \be \label{we1} \Xi':= \Psi_{\beta}(\tr(D^*\theta^*_{\xi})) = C f_{\phi,\chi} + 
g,
\quad
\textup{with
$\left<g,f_{\phi,\chi}\right>=0$}\ee and
$$C=\frac{\left<\Xi',
f_{\phi,\chi}\right>}{\left<f_{\phi,\chi},f_{\phi,\chi}\right>}.$$
(For the definition of the inner product on $\mM_{k,\nu}$ see section \ref{Theta series and 
inner
products}.)
\begin{lemma}\label{inn1} For any $f \in \mM_k$ one has $$\left<\Psi_{\beta}(f),
f_{\phi,\chi}\right> =
\left<f,
\Psi_{\beta}^{-1}(f_{\phi,\chi})\right>.$$

\end{lemma}
\begin{proof} We have \begin{multline} \left<\Psi_{\beta}(f),
f_{\phi,\chi}\right> = (\# \mB)^{-1} \sum_{b \in \mB} \left<
\Psi_{\beta}(f)_b, (f_{\phi,\chi})_b\right>  \\
= (\# \mB)^{-1}\sum_{b \in \mB} \left<\beta(\det b)f_b,
(f_{\phi,\chi})_b\right>\\
 =  (\# \mB)^{-1}\sum_{b \in \mB} \left<f_b,
\beta^{-1}(\det b)(f_{\phi,\chi})_b\right>=
\left<f,
\Psi_{\beta}^{-1}(f_{\phi,\chi})\right>, \end{multline} where the second and
the
fourth equality follow from commutativity of diagram (\ref{cen2}) while
the third one follows from the fact that $\beta(\det
b)^{-1}=\ov{\beta(\det b))}$. \end{proof}

\begin{lemma}\label{inn2} One has $$\left<\Xi', f_{\phi,\chi}\right> = 
\left<D\theta_{\xi},
\Psi_{\beta}^{-1}(f_{\phi,\chi})\right>_{ \Gamma_0^{\rm h}(Nc)}.$$ \end{lemma}
\begin{proof} Write $x= [U(\bfZ): \Gamma_0^{\rm h}(Nc)]^{-1}$. We have \begin{equation}\begin{split} \left<\Xi',
f_{\phi,\chi}\right> &= \left<\tr(D^*\theta^*_{\xi}), \Psi_{\beta}^{-1}(f_{\phi,\chi})\right>\\
& =(\# \mB)^{-1} \sum_{b \in \mB} \left<
\tr (D^*\theta_{\xi}^*)_b, \Psi_{\beta}^{-1}(f_{\phi,\chi})_b\right> \\
&=(\# \mB)^{-1}x\sum_{b \in \mB} \left<
\tr (D^*\theta_{\xi}^*)_b, \Psi_{\beta}^{-1}(f_{\phi,\chi})_b\right>_{ \Gamma_0^{\rm h}(Nc)}\\
&=(\# \mB)^{-1} x\sum_{b \in \mB} \sum_{\gamma \in J^{-1}\Gamma_0^{\rm h}(Nc)J\setminus U(\bfZ)}\left<
 (D^*\theta_{\xi}^*)_b|_k\gamma, \Psi_{\beta}^{-1}(f_{\phi,\chi})_b\right>_{ \Gamma_0^{\rm h}(Nc)}  \\
&=(\# \mB)^{-1}x \sum_{b \in \mB} \sum_{\gamma \in J^{-1}\Gamma_0^{\rm h}(Nc)J\setminus U(\bfZ)}\left<
 (D^*\theta_{\xi}^*)_b, \Psi_{\beta}^{-1}(f_{\phi,\chi})_b|_k\gamma^{-1}\right>_{ \Gamma_0^{\rm h}(Nc)}  \\
&=(\# \mB)^{-1} x\sum_{b \in \mB} \sum_{\gamma \in J^{-1}\Gamma_0^{\rm h}(Nc)J\setminus U(\bfZ)}\left<
 (D^*\theta_{\xi}^*)_b, \Psi_{\beta}^{-1}(f_{\phi,\chi})_b\right>_{ \Gamma_0^{\rm h}(Nc)} \\
&=(\# \mB)^{-1}  \sum_{b \in \mB}\cdot\left<
 (D^*\theta_{\xi}^*)_b, \Psi_{\beta}^{-1}(f_{\phi,\chi})_b\right>_{ \Gamma_0^{\rm h}(Nc)} \\
&=  
\left<D\theta_{\xi},
\Psi_{\beta}^{-1}(f_{\phi,\chi})\right>_{ \Gamma_0^{\rm h}(Nc)}, \end{split}\end{equation} where the first equality follows from Lemma \ref{inn1}, the fifth one and the last one from (10.9.3) in \cite{Shimura97} and the sixth one from the fact that $\Psi_{\beta}^{-1}(f_{\phi,\chi}) \in \mM_{k}$. 
\end{proof}

\begin{prop} \label{the frac1}
Let $\tau$ be as before. Suppose there exists $r \in \GL_2(\AKf)$ such 
that $c_{f_{\phi,\chi}}(\tau, r) \neq
0$. One has \be \begin{split}\frac{\left<\Xi',
f_{\phi,\chi}\right>}{\left<f_{\phi,\chi},f_{\phi,\chi}\right>} &=
(*)\eta^{-1}\frac{\#X_{m,Nc} \pi^3 (\det \tau)^{-k} | \det r|_K^{t/2}
c_{f_{\phi,\chi}}(\tau, r)}{A_N [\Gamma_0^{\rm h}(Nc): \Gamma]} \\
& \times \frac{L^{\tuint}(\BC(\phi),\frac{t+k}{2}+1, \beta\ov{\xi} \chi^{-1})
L^{\tuint}(\BC(\phi),\frac{t+k}{2}+2, \beta\ov{\xi} \chi^{-1})}
{L^{\tuint} (\Symm \phi,k)},\end{split}\ee   
where $\Gamma = \Gamma_1^{\rm h}(N) \cap \SU_2(\bfQ)$ (cf. section \ref{Theta series and 
inner
products}), $(*)\in E$ with $\val_{\ell}((*))\leq 0$,
$$L^{\textup{int}}(\BC(\phi),j+(t+k)/2,\omega):=
\frac{\Gamma(t+k+j)L(\BC(\phi),j+\frac{t+k}{2},\omega)} 
{\pi^{t+k+2j} \Omega_{\phi}^{+} \Omega_{\phi}^{-}},$$ for a Hecke character $\omega: \AK^{\times} \rightarrow \bfC^{\times}$ and
$$L^{\textup{int}}(\Symm \phi,n):=\frac{\Gamma(n)L(\Symm
\phi,n)}{\pi^{n+2} \Omega_{\phi}^{+} \Omega_{\phi}^{-}}$$ for any integer $n$.
\end{prop} 

\begin{rem} The normalized $L$-values are algebraic (see the proof 
below) and are expected to be algebraic integers, but for the moment we 
know of no proof of the latter claim. \end{rem}

\begin{proof}[Proof of Proposition \ref{the frac1}] By Lemma \ref{inn2} we get $$\left<\Xi', f_{\phi,\chi}\right> =  
\left<D\theta_{\xi},
\Psi_{\beta}^{-1}(f_{\phi,\chi})\right>_{\Gamma^{\rm h}_0(Nc)}.$$
Since the isomorphisms in diagram  (\ref{cen2}) are 
Hecke-equivariant,
$\Psi_{\beta}^{-1}(f_{\phi,\chi})$ is still a Hecke eigenform, hence we can
apply
Theorem \ref{ransel} and get
\begin{equation}\begin{split}\left<D\theta_{\xi},
\Psi_{\beta}^{-1}(f_{\phi,\chi})\right>_{\Gamma^{\rm h}_0(Nc)}& = 
\ov{C(2-m/2)}\cdot \ov{ L_{\rm st}(
\Psi_{\beta}^{-1}(f_{\phi,\chi}),3-m/2,
\xi)} \\
&= \ov{C(2-m/2)}\cdot  \ov{L_{\rm st}(
f_{\phi,\chi}, 3-m/2,
\beta^{-1}\xi)}.\end{split}\end{equation}
 Using Proposition
\ref{product134} 
 we
get
$$L_{\rm st}(
f_{\phi,\chi}, 3-m/2,
\beta^{-1}\xi)=L(\BC(\phi), \frac{t+k}{2}+1, \beta^{-1} \xi \chi) L(\BC(\phi),
\frac{t+k}{2}+2, \beta^{-1} \xi\chi).$$ 
Moreover, by Fact \ref{fact1} (and the fact that $\ov{\beta} = \beta^{-1}$) we get $$\ov{L(\BC(\phi), \frac{t+k}{2}+j, \beta^{-1} \xi \chi)} =L(\BC(\phi), \frac{t+k}{2}+j, \beta \ov{\xi} \chi^{-1}).$$
To ease notation write $c(\tau)$ for the $\tau$-Fourier coefficient $c_{(f_{\phi, \chi})_r}(\tau)$ of the $r$-component of $f_{\phi, \chi}$. Using the formula (\ref{formulaforC}) we get $$C(2-m/2) = (*) \frac{ (\# 
X_{m, Nc})
\pi^{-2t-2k-3} \Gamma(t+k+2) \Gamma(t+k+1) (\det
\tau)^{-k} |\det r|_K^{t/2}
c(\tau)}{A_N[\G_{0,n}^{\rm h}(Nc): \G]},$$ 
 with $A_N$ defined in section \ref{The standard L-function} and $(*) \in E$ with $\val_{\ell}((*))\leq 0$ (note that the product $\prod_{p \in \mathbf{b}} g_p(\xi^*(p \OK)  
p^{-2\ov{s}-n})$ in (\ref{formulaforC}) is a finite product and $g_p$ is a polynomial with coefficients in $\bfZ$ and constant term $1$ - cf. \cite{Shimura00}, Lemma 20.5).  
On the other hand we have by Theorem \ref{FF11} (note that $\val_{\ell}(\Gamma(k))=0$) 
$$\left< f_{\phi,\chi},f_{\phi,\chi}\right> =
 (*)  \pi^{-k-2}\cdot 
\left< \phi, \phi\right>
L(\Symm
\phi, k) \Gamma(k),$$ where $\val_{\ell}((*))=0$. 
Define $L^{\tualg}$ the same way as
$L^{\tuint}$ except with $\left<\phi,\phi\right>$ instead of
$\Omega_{\phi}^+\Omega_{\phi}^-$.
It follows from Remark 6.3 in \cite{Klosin09} and from Theorem 1 on page
325 in
\cite{Hida93} that \be   
\label{value43}
L^{\textup{alg}}(\BC(\phi),1+(t+k)/2, \beta \ov{\xi} \chi^{-1})\in \ov{\bfQ}\ee
and
\be
\label{value44} L^{\textup{alg}}(\BC(\phi),2+(t+k)/2, \beta \ov{\xi} \chi^{-1})\in
\ov{\bfQ}\ee and
from a result of Sturm \cite{Sturm80} that \be \label{value45}
L^{\textup{alg}}(\Symm \phi,k) \in
\ov{\bfQ}.\ee We note here that \cite{Sturm80} uses a definition of the
Petersson norm of $\phi$
which differs from ours by a factor of $\frac{3}{\pi}$, the volume of the
fundamental domain for the action of $\SL_2(\bfZ)$ on the complex upper
half-plane. We assume that $E$ contains values (\ref{value43}),
(\ref{value44}),
and
(\ref{value45}).
It follows from Proposition \ref{Hida45} that 
\be \label{periods2} \left<\phi,\phi\right> = (*)\hs \eta \hs \Omega_{\phi}^+
\Omega_{\phi}^-,\ee where $(*)$ is a
$\l$-adic unit as long as $\phi$ is ordinary at $\ell$ and $\ell>k$, which 
we have assumed. We also assume that $E$ contains $\eta$. The 
Proposition now follows.
\end{proof}

\subsection{Congruence between $f_{\phi,\chi}$ and a non-Maass form}

The goal of this section is to prove the following theorem, which is the
main result of the paper. To make the statement self-contained we repeat the assumptions made at the beginning of the section (the constant $A_N$ is defined in section \ref{The standard L-function}). In the next section we will formulate some 
consequences of this theorem.

\begin{thm} \label{mainthm} Let $K=\bfQ(i\sqrt{D_K})$ be an imaginary quadratic field of prime discriminant $-D_K$ and class number $h_K$. Let $k$ be an even positive integer divisible by $\#
\OK^{\times}$
and $\ell>k$ a rational prime such that $\ell \nmid D_K h_K$.
Let $\phi \in S_{k-1}(D_K, \chi_K)$ be a newform ordinary at $\ell$ and such that $\ov{\rho}_{\phi}|_{G_K}$
is absolutely irreducible.
Fix a Hecke character $\xi$ of $K$ such that
$\val_{\ell}(\cond \xi)=0$, $\xi_{\iy} (z) =
\bigl(\frac{z}{|z|}\bigr)^{-t}$ for some integer $-k \leq t < -6$, $\val_{\ell}(A_N)\geq 0$ and 
$\ell \nmid \# (\OK/N\OK)^{\times}$, where $N=D_Kh_KN_{K/\bfQ}(\cond \xi)$. Let $E$ be a
sufficiently large
finite extension of $\bfQ_{\ell}$ with uniformizer $\varpi$. Fix $\chi \in
\Hom(\Cl_K, \bfC^{\times})$ and $\beta\in \textup{Char}(k/2)$.
If $$-n:=\val_{\varpi}\left(\prod_{j=1}^2
L^{\textup{int}}(\BC(\phi),j+(t+k)/2, \beta \ov{\xi} \chi^{-1})\right) -
\val_{\varpi}(L^{\textup{int}}(\Symm \phi,k))<0$$
then there exists $f \in \mS_{k,-k/2}^{\chi}$, orthogonal to the Maass
space, such that $f
\equiv f_{\phi,\chi}$ \textup{(mod} $\varpi^n$\textup{)}.
\end{thm}

\begin{rem} Theorem \ref{mainthm} is a generalization of Theorem 7.12 in \cite{Klosin09}, which applied to the case $K=\bfQ(i)$. In that case the character $\beta$ is unique (and equals $\ov{\omega}$ in [loc.cit.]), the character $\chi$ is trivial since the class number of $\bfQ(i)$ equals $1$ and the character $\xi$ corresponds to the character which in [loc.cit.] was denoted by $\chi$. \end{rem} 

\begin{proof}[Proof of Theorem \ref{mainthm}]  
Consider again equation (\ref{we1}). Note that $\Xi'=\Psi_{\beta}(\tr(D^*
\theta^*_{\chi}))$ and $g$ lie in  $\mM_{k, -k/2}$ and $f_{\phi,\chi} \in
\mM_{k,-k/2}^{\chi}$. We would like all of the forms to be in
$\mM_{k,-k/2}^{\chi}$.

\begin{lemma} Let $Z$ denote the center of $U$. The quotient $Z(\AQ)/Z(\bfQ)$ is compact. \end{lemma}
\begin{proof} Note that $Z(\AQ) = \bigcup_{b \in \mB} Z(\bfQ) Z(\bfR) p_b
Z(\hat{\bfZ}).$ Since $Z(\bfR)$ is compact, the lemma follows. \end{proof}

Let $dz$ be a Haar measure on $Z(\AQ)/Z(\bfQ)$ normalized so that $\vol
(Z(\AQ)/Z(\bfQ)) = 1.$ For $f \in \mM_{k,-k/2}$ set $$(\pi_{\chi}f)(g) =
\int_{Z(\AQ)/Z(\bfQ)} f(gz) \chi^{-1}(z) dz =
\frac{1}{h_K} \sum_{b \in \mB} \chi^{-1}(p_b) f(g p_b) \in
\mM_{k,-k/2}^{\chi},$$ where by $\chi^{-1}(z)$, $\chi^{-1}(p_b)$ we mean $\chi^{-1}(c_K(\det z)^{1/2})$ and $\chi^{-1}(c_K(\det b))$ respectively with $c_K: \AK^{\times} \twoheadrightarrow \Cl_K$ the canonical map. The last equality
clearly implies that 
a Fourier coefficient of $\pi_{\chi}f$ is in $\Oo$ when
the corresponding Fourier coefficient of $f$ is in $\Oo$ (since $\ell 
\nmid h_K$). Note the slight abuse of terminology 
when we say that $f$ (or other adelic hermitian modular form) has 
Fourier coefficients in $\Oo$. By saying so, we mean that for 
$h \in \mS$ and $r \in \GL_2(\AKf)$ one has $e^{2 \pi \tr h} 
c_{f}(h,r) \in \Oo$. We will continue this abuse. Apply 
$\pi_{\chi}$ 
to 
both 
sides of
(\ref{we1}). Write $\Xi = \pi^{-3} \pi_{\chi} \Xi'$ and $g_0 = \pi^{-3} \pi_{\chi} g$. Then we have \be
\label{inn4} \Xi = C_{\phi,\chi}  
f_{\phi,\chi} + g_0 \in
\mM_{k, -k/2}^{\chi}\ee with $\left< g_0, f_{\phi,\chi} \right> =0$ 
and $C_{\phi, \chi}:=\pi^{-3}C$.

Combining Corollaries \ref{intE} and \ref{inttheta}
we get that for every $b \in \mB$ and every $h \in \mS$, the 
$h$-Fourier coefficient $c_{\Xi_b}(h)$ of  $\Xi_b$ is in $\Oo$.

Fix an ordinary pair $(\tau, b_0)$ and as before set $c=\det \tau$. Then $\val_{\ell}(c(\tau))=0$ with $c(\tau)$ as in section \ref{The inner product ratio}. Since $Nc>1$ it follows from the proof of Lemma 11.14 in \cite{Shimura97} together with Lemma 11.15 in [loc.cit.] and the remark following it that the order of $X_{m,Nc}$ equals the index of the group $\{x \in \AK^{\times} \mid x_{\fp} \in \Oo_{K, \fp}^{\times} \hf \textup{and} \hf x_{\fp} -1 \in Nc\Oo_{K, \fp} \hf \textup{for every $\fp \nmid \infty$}\}$ inside $\AK^{\times}$. 
Hence in particular
the 
assumptions in the theorem imply that $\ell 
\nmid \#X_{m,Nc}$. So, from (\ref{inn4}), (\ref{we1}) and Proposition 
\ref{the frac1} we obtain that \be \label{cfor} C_{\phi, \chi}=(*) 
\eta^{-1}
 \frac{L^{\tuint}(\BC(\phi)\frac{t+k}{2}+1,  \beta \ov{\xi} \chi^{-1})
L^{\tuint}(\BC(\phi)\frac{t+k}{2}+2,  \beta \ov{\xi} \chi^{-1})}
{ L^{\tuint} (\Symm \phi,k)},\ee where $\val_{\ell}((*))\leq 0$ since $\det 
\tau \in \Oo$, $A_N \in \Oo$ and $[\Gamma_0^{\rm h}(Nc): \Gamma]\in \bfZ$. 
Note that under our assumption on the $L$-function (and ignoring the 
factor $\eta^{-1}$), this equality (together with fact that for every $b \in \mB$ the forms $\Xi_b$ 
and $(f_{\phi,\chi})_b$ have Fourier coefficients in $\Oo$) implies that we 
must 
have a mod 
$\varpi^n$ congruence between $f_{\phi,\chi}$ and $-\varpi^n g_0$. 
However, there is no guarantee that $g_0$ is orthogonal to the Maass 
space. So, we will now 
use the Hecke operator $T^{\rm h}$ which we constructed in section \ref{Completed Hecke algebras} to 
`kill' the `Maass' part of $g_0$. 

Indeed, by Theorem \ref{existhecke} there exists $T^{\rm h} \in \mH^{\chi}_{\Oo}$ such that $T^{\rm h} f_{\phi,\chi} = \eta f_{\phi,\chi}$ and $T^{\rm h}f=0$ for any eigenform $f \in \mS_{k,-k/2}^{{\rm M},\chi}$ orthogonal to $f_{\phi,\chi}$.

 We apply $T^{\hh}$ to both sides of (\ref{inn4}).
 As for all $b \in \mB$ and $h \in 
\mS$, the Fourier coefficients $e^{2 \pi
\tr h} c_{\Xi}(h, b)$ of
$\Xi$ 
lie in 
$\Oo$, so
do the Fourier coefficients of
$T^{\hh}\Xi$ by Propositions \ref{f3} and \ref{f4}. 
 Moreover, since
$\theta_{\chi}$ is a cusp form, so are $\Xi$ and $T^{\hh}\Xi$.

 We thus get \be \label{ortho974}
T^{\tuh} \Xi = \eta
C_{\phi, \chi} f_{\phi, \chi} +  T^{\tuh} g_0\ee with
$T^{\tuh}  
g_0$
orthogonal to the Maass space.

As $C_{\phi, \chi} \in E \subset
\bfC$ by (\ref{cfor}), it makes sense to
talk about
its 
$\varpi$-adic valuation. Suppose $\val_{\varpi}(\eta \hs
C_{\phi, \chi}) =
- n
\in \bfZ_{<0}$. 
Note that
since
the $(h,b)$-Fourier coefficients of $T^{\rm h}\Xi$ and of 
$f_{\phi, \chi}$ lie
in
$\Oo$ for all $b \in \mB$ and all $h \in \mS$, but $\eta
\hs C_{f_{\phi, \chi}}
\not\in \Oo$, we must have that either $T^{\rm h} g_0 \neq
0$ or
$e^{2 \pi \tr h}c_{f_{\phi, \chi}}(h, b) \equiv 0$ mod $\varpi$ for all $b\in 
\mB$ and all 
$h 
\in \mS$. 
\begin{lemma} \label{missl} There exists a pair $( h, b)$ such that  
$e^{2 \pi \tr h}c_{f_{\phi, \chi}}( h, b) \not\equiv 0$ mod 
$\varpi$. \end{lemma}

\begin{proof} Assume on the contrary that $e^{2 \pi \tr h}c_{f_{\phi, \chi}}(h, b) \equiv 
0$ mod
$\varpi$ for all pairs $( h,b)$. By our choice of $\mB$, taking $h=\bsmat p 
& 0 \\ 0 & 1\esmat$ with $p=1$ or a prime, (\ref{Maass 
condition}) implies then that $c_{b,f_{\phi, \chi}}(D_K p) \equiv 0$ mod 
$\varpi$ for all primes $p$ and for $p=1$. Using Theorem \ref{desc4} we 
get that 
$a_{\phi}(p D_K) 
- \ov{a_{\phi}(p D_K)} \equiv 0$ mod $\varpi$ for all primes $p$ and $p=1$. 
Here $a_{\phi}(n)$ stands for the $n$th Fourier coefficient of $\phi$. 
By taking 
$p=1$ 
we conclude that 
$a_{\phi}(D_K) \equiv \ov{a_{\phi}(D_K)}$ (mod $\varpi$). Since $|a_{\phi}(D_K)| = D_K^{(k-2)/2}$ 
(see for example \cite{Iwaniec97}, formula (6.90)), we have 
$\val_{\ell}(a_{\phi}(D_K))=0$ since $\ell \nmid D_K$. Hence 
we must have 
$a_{\phi}(p)\equiv \ov{a_{\phi}(p)}$ (mod $\varpi$) 
for all primes $p$. 
This on the other hand implies 
that $\ov{\rho}_f|_{G_K}$ is not absolutely irreducible by Proposition \ref{congffrho93}. This contradicts our assumptions. \end{proof}

By Lemma \ref{missl} we must have
$ T^{\rm h} g_0 \neq
0$.
Write
$\eta \hs C_{\phi, \chi}= a \varpi^{-n}$ with $a \in \Oo^{\times}$. Then
the Fourier
coefficients of $(\varpi^n T^{\rm h}g_0)_b$ lie in $\Oo$ and one has
$$f_{\phi,\chi}
\equiv
-a^{-1}\varpi^n
T^{\rm h}g_0 \quad (\textup{mod} \hf \varpi^n).$$
As explained above, $-a^{-1}\varpi^nT^{\rm h}g_0$ is a
hermitian modular form orthogonal to
the Maass
space. This completes the proof of Theorem \ref{mainthm}.
\end{proof}

\begin{cor} \label{cormain} Suppose that $\xi, \chi, \beta$ in Theorem \ref{mainthm}
can be chosen so that $$\val_{\varpi}\left(\prod_{j=1}^2
L^{\textup{int}}(\BC(\phi),j+(t+k)/2, \beta \ov{\xi} \chi^{-1})\right)=0,$$ 
 then
$n$  
in Theorem \ref{mainthm} can be taken to be
$\val_{\varpi}(L^{\textup{int}}(\Symm \phi,k))$. \end{cor}
  
\begin{rem} As already discussed in Remark 7.15 of \cite{Klosin09}, the existence of
character $\xi$ 
as in Corollary \ref{cormain} is not known in general. Note that one needs to ``control'' two $L$-values at the same time to ensure that their product is a $\varpi$-adic unit. 
However, 
if the class number of $K$ is larger than one, we now have (slightly) more flexibility as we also get to choose the character $\beta$ (or equivalently $\chi^{-1}\beta$). While we still do not have a proof for this fact it seems very unlikely that for all the possible combinations of the characters $\xi$ and $\beta$ the product of $L$-values should always involve non-zero powers of $\varpi$. 
\end{rem}
\begin{rem} The ordinarity assumption on $\phi$ in Theorem \ref{mainthm} is used in section \ref{Completed Hecke algebras} to
construct the Hecke operator $T^{\rm h}$ annihilating the Maass part of $g_0$ as
above as well as to ensure that $(*)$ in (\ref{periods2}) is a
$\varpi$-adic unit. Note that the operator $T^{\rm h}$ is not necessary provided that $\phi$ is not congruent (mod $\varpi$) to any other $\phi' \in S_{k-1}(D_K, \chi_K)$. Indeed, then there cannot be any `Maass part' of $g_0\in \mS^{\chi}_{k, -k/2}$ that is congruent to $f_{\phi, \chi}$. One expects that the set of primes $\ell$ of $\bfQ$
such that a given (non-CM) form $\phi$ is ordinary at $\ell$ has Dirichlet
density one, but for now no proof of this fact is known. An analogous
statement for elliptic curves was proved by Serre \cite{Serre66}.
\end{rem}

\subsection{The Maass ideal} \label{The Maass ideal}

\begin{cor} \label{cormain78} Under the assumptions of Theorem
\ref{mainthm} there a is cuspidal Hecke eigenform $g \in \mS_{k, -k/2}^{\chi}$ such that
$\val_{\varpi}(\lambda_{f_{\phi,\chi}}(T) - \lambda_g(T))>0$ for all
Hecke
operators
$T \in \mH^{\chi}_{\mathcal{O}}$. Here the homomorphism $\lambda$ sends the Hecke operator to its eigenvalue.
\end{cor}

\begin{proof} Let $f$ be as in Theorem \ref{mainthm}. Using the
decomposition (\ref{hermdec}), we see that there exists a Hecke
operator
$T_0 \in
\mH^{\chi}_{\mathcal{O}}$ such that $T_0f_{\phi, \chi}=f_{\phi, \chi}$ and
$T_0f'=0$
for each
$f'\in
\mathcal{S}_{k, -k/2}^{\chi}$
which is orthogonal to all Hecke eigenforms whose eigenvalues are
congruent to those of $f_{\phi, \chi}$ (mod $\varpi$). Let $\mN^{\rm h}$ be a basis of eigenforms for $\mH_{k, -k/2}^{\chi}$ such that $f_{\phi, \chi} \in \mN^{\rm h}$. Suppose all the elements of
$\mN^{\hh}$
whose eigenvalues are congruent to those of $f_{\phi, \chi}$ (mod $\varpi$) lie in the Maass space.
Then $T_0 f=0$. However, since $f_{\phi, \chi} \equiv f$ (mod $\varpi$), this implies that with respect to some base $\mB$ for all $b \in \mB$ all the Fourier coefficients of the  $b$-component of $f_{\phi, \chi}$ are in $\varpi \Oo$. By Theorem \ref{desc4} and (\ref{Maass condition})
this is only possible if $\phi \equiv
\phi^{\rho}$ (mod $\varpi$). This however leads to a contradiction by Proposition
\ref{congffrho93}. \end{proof}

Recall that we have a Hecke-stable decomposition
$$\mathcal{S}_{k, -k/2}^{\chi} = \mathcal{S}_{k, -k/2}^{{\rm M}, \chi}
\oplus
\mathcal{S}_{k, -k/2}^{\textup{NM}, \chi},$$
where $\mathcal{S}_{k, -k/2}^{\textup{NM}, \chi},$
 denotes the orthogonal
complement of $\mathcal{S}_{k, -k/2}^{{\rm M}, \chi}$
 inside $\mathcal{S}_{k, -k/2}^{\chi}$. Denote by $\mH^{\textup{NM}, \chi}_{\Oo}$
the image of $\mH^{\chi}_{\Oo}$ inside $\textup{End}_{\bfC}
(\mathcal{S}_{k, -k/2}^{\textup{NM}, \chi})$ and let $\Phi:
\mH^{\chi}_{\Oo} \twoheadrightarrow \mH^{\textup{NM}, \chi}_{\Oo}$ be the
canonical $\Oo$-algebra epimorphism. Let
$\textup{Ann}(f_{\phi, \chi})\subset \mH^{\chi}_{\Oo}$ denote the annihilator of
$f_{\phi, \chi}$. It is a prime ideal of
$\mH^{\chi}_{\Oo}$ and $\lambda_{f_{\phi, \chi}}: \mH^{\chi}_{\Oo}
\twoheadrightarrow \Oo$ induces an $\Oo$-algebra isomorphism
$\mH^{\chi}_{\Oo} / \textup{Ann}(f_{\phi, \chi}) \xrightarrow{\sim}
\Oo$.

\begin{definition} \label{CAPidedef} As $\Phi$ is surjective,
$\Phi(\textup{Ann}(f_{\phi, \chi}))$ is an ideal of
$\mH^{\textup{NM}, \chi}_{\Oo}$. We call it the \textit{Maass ideal
associated to $f_{\phi, \chi}$}. \end{definition}
There exists a
non-negative integer $r$ for which the diagram
\be \label{diagramCAP32} \xymatrix{\mH^{\chi}_{\Oo}\ar[r]^{\Phi} \ar[d]&
\mH^{\textup{NM}, \chi}_{\Oo}\ar[d]\\
\mH^{\chi}_{\Oo}/\textup{Ann}(f_{\phi, \chi}) \ar[r]^{\Phi} \ar[d]^{\wr}_{\lambda_{f_{\phi, \chi}}} &
\mH^{\textup{NM}, \chi}_{\Oo}/\Phi(\textup{Ann}(f_{\phi, \chi}))\ar[d]^{\wr}\\
\Oo\ar[r]& \Oo/\varpi^r \Oo  }\ee
all of whose arrows are $\Oo$-algebra epimorphisms, commutes.

\begin{cor} \label{CAPideal1} If $r$ is the integer from diagram
(\ref{diagramCAP32}), and $n$ is as in Theorem
\ref{mainthm}, then $r\geq n$. \end{cor}

\begin{proof} Set $\mathcal{N}^{\textup{NM}}:= \{ f \in \mathcal{N}^{\hh}
\mid f \in
\mathcal{S}_{k, -k/2}^{\textup{NM}, \chi}\}.$ Choose any
$T\in
\Phi^{-1}(\varpi^r) \subset \mH^{\chi}_{\Oo}$.
 Suppose that $r <n$, and let $f$ be as in
Theorem \ref{mainthm}.
We have \be \label{aux549} f_{\phi, \chi} \equiv f \quad (\textup{mod} \hf \varpi^n).
\ee and $Tf = \varpi^r f$ and $Tf_{\phi, \chi}=0$. Hence applying $T$ to both sides of
(\ref{aux549}),
we see that for some base $\mB$ and every $b \in \mB$ all the Fourier coefficients of the $b$-component of $\varpi^r f$ lie in $\varpi^n \Oo$.
Since $r < n$ this along with (\ref{aux549}) implies that all the Fourier coefficients of the $b$-component of $f_{\phi, \chi}$ lie $\varpi \Oo$, which is impossible as shown in the proof of
Corollary
\ref{cormain78}. \end{proof}

\begin{rem} The Maass ideal plays a role similar to the classical Eisenstein ideal. Its index inside the Hecke algebra is a measure of the congruences between $f_{\phi, \chi}$ and eigenforms in $\mS_{k, -k/2}^{\chi}$ which are orthogonal to the Maass space. While Corollary \ref{cormain78} only guarantees a Hecke eigenform $f$ orthogonal to the Maass space congruent to $f_{\phi, \chi}$ modulo $\varpi$, the quotient $\mH^{\textup{NM}, \chi}_{\Oo}/\Phi(\textup{Ann}(f_{\phi, \chi}))$ takes into account all such eigenforms $f$ at the same time and hence gives a better idea of how much congruence there is between $f_{\phi, \chi}$ and eigenforms orthogonal to the Maass space. Also, it is exactly the index of the Maass ideal inside $\mH^{\textup{NM}, \chi}_{\Oo}$ that bounds the order of the appropriate Selmer group from below, as we discuss in the next section. \end{rem}

\subsection{Unitary analogue of Harder's conjecture} \label{Unitary analogue of Harder's conjecture}  Let $E$ be a sufficiently large finite extension of $\bfQ_{\ell}$, with ring of integers $\Oo$ and uniformizer $\varpi$. The original Harder's conjecture states that if $\ell$ is ``large'' and  $\varpi \mid L^{\rm alg}(f,j+k)$ (the appropriately normalized algebraic part of the special value of the standard $L$-function of $f$) for a cuspidal elliptic eigenform $f =\sum_{n}a_n(f)e(z) \in S_{r}(\SL_2(\bfZ))$, then there exists a cuspidal (vector-valued) Siegel modular eigenform $F$ of full level, whose eigenvalues for the Hecke operators $T(p)$ (for all primes $p$) are congruent to $p^{k-2} + p^{j+k-1} + a_p(f)$ modulo $\varpi$. Here $r=j+2k-2$, and $T(p)$ is the Hecke operator acting on the space of Siegel modular forms given by the double coset $\Sp_4(\bfZ) \diag(1,1,p,p) \Sp_4(\bfZ)$. For details see  \cite{Harder08} or \cite{VanderGeer08}. 

Recently Dummigan \cite{Dummigan11Harderconjecture} formulated an analogue of this conjecture for the group $\U(2,2)$. (We are grateful to him for sending us his preprint). Let $\phi \in S_{k-1}(D_K, \chi_K)$ be as before. Let $j$ be an integer such that $0 \leq j \leq (k-4)/2$ (note that our $k$ differs from Dummigan's $k$ by 1). Suppose $$\val_{\varpi}(L^{\rm alg}(\Symm \phi, 2k-4-2j))>0.$$ Write $\pi_{\phi}$ for the automorphic representation of $\GL_2(\AQ)$ associated with $\phi$. Let $\Pi(\phi)$ denote the representation $\textup{Ind}_{P(\AQ)}^{U(\AQ)}(\BC_{K/\bfQ}(\pi_{\phi}) \cdot |\det|^{k-(3/2)-j})$ of $U(\AQ)$. Then the unitary analogue of Harder's conjecture asserts that if $0 \leq j < (k-4)/2$ then there exists a cuspidal automorphic representation $\Pi$ of $U(\AQ)$  (whose finite part contributes to the cuspidal cohomology of degree 4 - for details cf. \cite{Dummigan11Harderconjecture}), unramified away from $D_K$, such that \be \label{har1} \lambda_{\Pi(\phi)}(T) \equiv \lambda_{\Pi}(T) \pmod{\varpi}\ee for all Hecke operators $T$ in the local Hecke algebras away from $D_K$. Here $\lambda$ denotes the appropriate Hecke eigenvalue.

Let us now briefly explain the relation of Corollary \ref{cormain78} to this conjecture.
First, assume that $\phi$ is not congruent to any other $\psi \in S_{k-1}(D_K, \chi_K)$ mod $\varpi$. This implies that the Hida congruence module of $\phi$ is a $\varpi$-adic unit, so that $\val_{\varpi}(L^{\rm int}(\Symm \phi, k)) = \val_{\varpi}(L^{\rm alg}(\Symm \phi, k))$ (cf. Proposition \ref{Hida45}). Secondly, when $j=(k-4)/2$, the automorphic representation $\Pi_{\phi}$ attached to the Maass lift $f_{\phi, \mathbf{1}}$ is associated (in the sense of Piatetski-Shapiro \cite{Piatetski-Shapiro83}) with $\Pi(\phi)$. In fact the local representations are isomorphic at all finite places, hence $\Pi_{\phi}$ shares the Hecke eigenvalues with $\Pi(\phi)$. Let $g\in \mS_{k, -k/2}^{\mathbf{1}}$ be a Hecke eigenform congruent to $f_{\phi, \chi}$ as in Corollary \ref{cormain78}. Then the automorphic representation $\Pi$ of $U(\AQ)$ associated with $g$ is cuspidal and unramified everywhere and its eigenvalues satisfy  (\ref{har1}).

However, note that Dummigan's conjecture specifically excludes the case $j=(k-4)/2$ hence our result should be viewed as complementary to that conjecture rather than as a case of it. Indeed, the case $j=(k-4)/2$ is special because of the existence of the CAP representation $\Pi_{\phi}$ of $U(\AQ)$ associated with $\Pi(\phi)$ which has a holomorphic vector $f_{\phi, \chi}$ in it.  The holomorphicity of $f_{\phi, \chi}$ in particular allows us to use a holomorphic Eisenstein series to define the form $\Xi$ and for such Eisenstein series we know the integrality of their Fourier coefficients thanks to results of Shimura (cf. section \ref{ellint}). The main point of Dummigan's conjecture is the prediction of the congruence (\ref{har1}) in the absence of a CAP representation.

\section{The Bloch-Kato conjecture} \label{The Bloch-Kato conjecture}

In section \ref{Selmer groups} we will discuss how the results of the previous sections can be applied to give evidence for the Bloch-Kato conjecture for a twist of the adjoint motive of an elliptic modular form $\phi$. Since these results (and proofs) are completely analogous to the case considered in \cite{Klosin09}, we will just give the relevant statements and refer the reader to [loc. cit.] for details.

\subsection{Selmer groups} \label{Selmer groups}
We begin by defining the Selmer group. For a profinite
group $\mG$ and a $\mG$-module $M$ (where we assume the action of $\mG$ on
$M$
to be continuous) we will consider the group $H^1_{\textup{cont}}(\mG,M)$
of
cohomology classes of continuous cocycles $\mG \rightarrow M$. To shorten
notation we will suppress the subscript `cont' and simply write
$H^1(\mG,M)$.
For a field $L$, and a $\Gal(\ov{L}/L)$-module $M$ (with a continuous
action of $\Gal(\ov{L}/L)$) we sometimes write $H^1(L, M)$ instead of
$H^1_{\textup{cont}} (\Gal(\ov{L}/L), M)$. 

Let
$L$ be a number field. For a rational prime $p$ denote by $\Sigma_{p}$
the set of primes of
$L$ lying over $p$. Let $\S \supset \S_{\ell}$ be a finite set of
primes of $L$ and denote by $G_{\S}$ the Galois group of the maximal
Galois extension $L_{\S}$ of $L$ unramified outside of $\S$. Let $E$ be a (sufficiently large) finite extension of $\bfQ_{\ell}$ with ring of integer $\Oo$ and a fixed uniformizer $\varpi$. Let $V$ be a
finite
dimensional $E$-vector space with a continuous $G_{\Sigma}$-action.
Let $T \subset V$ be a $G_{\S}$-stable
$\Oo$-lattice. Set $W:= V/T$.

We begin by defining local Selmer groups. For every $\fp \in \S$
set
$$H^1_{\textup{un}}(L_{\fp}, M):= \ker \{ H^1(L_{\fp},M)
\xrightarrow{\textup{res}}
H^1(I_{\fp},M)\}.$$ Define the local $\fp$-Selmer
group (for $V$) by
$$H^1_{f}(L_{\fp},V):=\begin{cases} H^1_{\textup{un}}(L_{\fp},
V)& \fp \in \S \setminus \S_{\ell}\\
\ker \{ H^1(L_{\fp},V) \rightarrow
H^1(L_{\fp},V\otimes
B_{\textup{crys}})\} &\fp \in \S_{\ell}. \end{cases} $$
Here $B_{\textup{crys}}$
denotes Fontaine's
ring of $\ell$-adic
periods (cf. \cite{Fontaine82}).

For $\fp \in \S_{\ell}$, we call the $D_{\fp}$-module $V$
\textit{crystalline} (or the $G_L$-module $V$ \textit{crystalline at
$\fp$}) if
$\dim_{\bfQ_{\ell}} V = \dim_{\bfQ_{\ell}} H^0(L_{\fp}, V\otimes
B_{\textup{crys}})$. When we
refer to a Galois representation $\rho: G_L \rightarrow GL(V)$ as being
crystalline at $\fp$,
we mean that $V$ with the $G_L$-module structure defined by $\rho$ is
crystalline at
$\fp$.

For every $\fp$, define $H^1_{f}(L_{\fp},W)$ to be
the image of
$H^1_{f}(L_{\fp},V)$ under the natural map $H^1(L_{\fp},V)
\rightarrow
H^1(L_{\fp},W)$. Using the fact that $\Gal(\ov{\kappa}_{\fp}:
\kappa_{\fp}) = \hat{\bfZ}$ has
cohomological dimension 1, one easily sees that if $W$ is unramified at
$\fp$ and $\fp \not\in \S_{\ell}$, then
$H^1_{f}(L_{\fp},
W) = H^1_{\textup{un}}(L_{\fp},W)$. Here $\kappa_{\fp}$ denotes the
residue field of $L_{\fp}$.

For a $\bfZ_{\ell}$-module $M$, we write $M^{\vee}$ for its Pontryagin
dual defined as
$$M^{\vee} = \Hom_{\textup{cont}}(M, \bfQ_{\ell}/\bfZ_{\ell}).$$ Moreover,
if $M$ is a Galois module, we denote by $M(n):= M\otimes\epsilon^n$ its
$n$-th Tate twist.

\begin{definition} The group $$H^1_f(L, W):= \ker\left\{H^1(G_{\Sigma},W)
\xrightarrow{\textup{res}}
\bigoplus_{\fp \in \Sigma}
\frac{H^1(L_{\fp},W)}{H^1_{f}(L_{\fp},W)}
\right\}$$
is called the (global) \textit{Selmer group of  $W$}.
\end{definition}

For $L=\bfQ$, the group
$H^1_{f}(\bfQ, W)$ is the Selmer group defined by
Bloch and Kato \cite{BlochKato90}, section 5.
Let $\rho: G_{\S} \rightarrow \GL_E(V)$ denote the representation giving
the action of $G_{\S}$ on $V$. The following two lemmas are easy (cf.
\cite{Rubin00}, Lemma 1.5.7 and \cite{Skinner04}).
\begin{lemma} \label{fingen} $H^1_f(L,W)^{\vee}$ is a finitely generated
$\Oo$-module. \end{lemma}

\begin{lemma} \label{length} If the mod
$\varpi$ reduction $\ov{\rho}$ of $\rho$ is absolutely
irreducible, then the length of $H^1_f(L,W)^{\vee}$
as an $\Oo$-module is independent of the choice of the lattice $T$.
\end{lemma}

\begin{rem} \label{length2} For an $\Oo$-module $M$,
$\val_{\ell}(\# M) = [\Oo/\varpi: \bfF_{\ell}]\length_{\Oo}(M)$. \end{rem}

Let $K$ be an imaginary quadratic field of prime discriminant $-D_K$.
 Let $\phi=\sum_{n=1}^{\iy} a(n) q^n \in \mN$ be such
that
$\ov{\rho}_{\phi}|_{G_K}$ is absolutely irreducible. Then by Proposition
\ref{congffrho93}, $f_{\phi, \chi}
\neq 0$. From now on we also assume that $\ad^0 \ov{\rho}_{\phi}|_{G_K}$, the
trace-0-endomorphisms of the representation space of $\ov{\rho}_{\phi}|_{G_K}$
with the usual $G_K$-action, is absolutely irreducible. Finally, to be able to show that the cohomology classes we produce are unramified at the prime $D_K$ we assume that (under the chosen embedding $\ov{\bfQ}_{\ell} \hookrightarrow \bfC$) the Fourier coefficient $a(D_K)$ is neither congruent to $D_K^k$ nor to $D_K^{k-4}$ modulo $\varpi$ (see \cite{Klosin09}, Lemma 9.23 for how this assumption is used).  By (\ref{specialform}) we have $$\rho_{f_{\phi, \chi}}\cong
(\rho_{\phi}|_{G_K} \oplus (\rho_{\phi}\otimes\epsilon)|_{G_K})\otimes \chi \epsilon^{2-k/2}.$$

Let $V$
denote the
representation space of $$\ad^0\rho_{\phi}|_{G_K}(-1)=\ad^0
\rho_{\phi}|_{G_K} \otimes
\epsilon^{-1} \subset \Hom_E((\rho_{\phi}\otimes \epsilon)|_{G_K}, \rho_{\phi}|_{G_K})$$ of
$G_K$. Let $T \subset
V$ be some choice of a
$G_K$-stable lattice.
Set $W=V/T$. Note that the action of $G_K$ on $V$ factors through
$G_{\Sigma_{\ell}}$. Since the mod $\varpi$ reduction of $\ad^0
\rho_{\phi}|_{G_K} \otimes
\epsilon^{-1}$ is absolutely irreducible by assumption,
$\val_{\ell}(H^1_f(K,W)^{\vee})$ is independent of the choice of $T$.

Let $\mN^{{\rm NM}}$, $\mH^{{\rm NM}, \chi}_{\Oo}$ and $\Phi$ be as in section \ref{The Maass ideal}. Let $\fm_{\phi}$ be the maximal ideal of $\mH^{{\rm NM}, \chi}_{\Oo}$ corresponding to $f_{\phi, \chi}$ and write $\mH^{\textup{NM}, \chi}_{\fm_{\phi}}$ for the localization of $\mH^{{\rm NM}, \chi}_{\Oo}$ at $\fm_{\phi}$ and $\Phi_{\fm_{\phi}}$ for the corresponding ``local'' component of $\Phi$.  Write $\mN_{f_{\phi, \chi}}^{\tuNM}$ for the subset of $\mN^{{\rm NM}}$ consisting of eigenforms whose corresponding maximal ideal of $\mH^{{\rm NM}, \chi}_{\Oo}$ is $\fm_{\phi}$.

The main result of this section is the following theorem.

\begin{thm} \label{Selmerrefined} Let $W$ be as above.
Suppose that for each $f \in \mN_{f_{\phi, \chi}}^{\tuNM}$, the representation
$\rho_f: G_K \rightarrow \GL_4(E)$ is absolutely irreducible. Then
$$\val_{\ell}(\# H^1_f(K, W)^{\vee}) \geq \val_{\ell} (\#
\mH^{\textup{NM}, \chi}_{\fm_{\phi}}/\Phi_{\fm_{\phi}}(\Ann (f_{\phi, \chi}))).$$  \end{thm}

\begin{proof} This is proved in the same way as Theorem 9.10 in \cite{Klosin09} and we will not reproduce the proof here. The key point is that eigenforms $f \in \mS_{k, -k/2}^{{\rm NM},\chi}$ congruent to $f_{\phi, \chi}$ modulo powers of $\varpi$ give rise to non-split extensions of $\rho_{\phi}(1)|_{G_K}$ by $\rho_{\phi}|_{G_K}$. These extensions are checked to satisfy the local conditions defining the Selmer group and can be put together to generate a submodule of $H^1_f(K,W)$ of order no smaller than the index of the Maass ideal inside the local Hecke algebra. In this one mostly follows Urban \cite{Urban01}. \end{proof}

\begin{rem} The irreducibility assumption in Theorem \ref{Selmerrefined} is presumably unnecessary. If one assumes multiplicity one for the Maass forms in the sense that the only eigenforms in $\mM^{\chi}_{k,-k/2}$ (i.e., in particular holomorphic) sharing all Hecke eigenvalues with $f_{\phi, \chi}$ are multiples of $f_{\phi, \chi}$ then one needs to show that all non-CAP cuspidal automorphic representations have irreducible Galois representations. This is expected to be the case, but we know of no proof of this fact. \end{rem} 

\begin{cor} \label{Selemrrefined2} With the same assumptions and notation
as in Theorem \ref{mainthm} and Theorem \ref{Selmerrefined} we have
$$\val_{\ell}(\# H^1_f(K,W)) \geq n [\Oo/\varpi : \bfF].$$ If in addition the characters
$\xi, \beta, \chi$
in Theorem \ref{mainthm} can be taken as in Corollary \ref{cormain},
then
$$\val_{\ell}(\# H^1_f(K,W)) \geq \val_{\ell}(\#
\Oo/L^{\textup{int}}(\Symm \phi,
k)).$$ \end{cor}

\begin{proof} The corollary follows immediately from Theorem
\ref{Selmerrefined} and Corollary \ref{CAPideal1}. \end{proof}

With the assumptions as in Corollary \ref{Selemrrefined2} we have thus the following inclusion of the fractional ideals of $\Oo$: 
\be \label{cont}\#H^1_f(K,W )\cdot \Oo \subset
L^{\tuint}(\Symm
\phi,k)\cdot \Oo .\ee
Note that since $\chi_K$ is the nebentypus of $\phi$ one has $$\ad^0\rho_{\phi}(-1)\chi_K = \Symm \rho_{\phi} (k-3).$$ Here we treat $\chi_K$ as a Galois character via class field theory.
Hence assuming that a certain technicality concerning the Tamagawa factors at the prime $D_K$ can be proved (cf. section 9.3 in \cite{Klosin09}), the Bloch-Kato conjecture can be formulated as follows:

\begin{conj}[Bloch-Kato] \label{BKconj8} One has the following equality of fractional ideals of $\Oo$:
\be \label{BK8eq} \#H^1_f(\bfQ, \Symm \rho_{\phi}(k-3))\cdot \Oo = \#H^1_f(\bfQ, \ad^0\rho_{\phi}(-1)\chi_K)\cdot \Oo =
L^{\tuint}(\Symm
\phi,k)\cdot \Oo .\ee 
\end{conj}

Thus Corollary \ref{Selemrrefined2} provides evidence for Conjecture \ref{BKconj8}, but falls short of proving that the left-hand side of (\ref{BK8eq}) is contained in the right-hand side, because the module $H^1_f(K,W ) = H^1_f(K, \ad^0\rho_{\phi}|_{G_K}(-1))$ can potentially be larger than $H^1_f(\bfQ, \ad^0\rho_{\phi}(-1)\chi_K)$. 
For a more detailed discussion see \cite{Klosin09}, section 9.3.

\bibliography{standard2}
\bibliographystyle{plain}

\vspace{50pt}

\noindent Author's address:\\

\noindent Department of Mathematics\\
Queens College,\\
City University of New York,\\
65-30 Kissena Blvd.\\
Flushing, NY 11367\\
USA\\
\\
email: klosin@math.utah.edu

\end{document}

%% file: makros.tex
\def\a{\alpha}
\def\b{\beta}
\def\g{\gamma}
\def\G{\Gamma}
\def\d{\delta}

\def\l{\lambda}

\def\S{\Sigma}

\newcommand{\mB}{\mathcal{B}}
\newcommand{\mC}{\mathcal{C}}

\newcommand{\mG}{\mathcal{G}}
\newcommand{\mH}{\mathcal{H}}

\newcommand{\mJ}{\mathcal{J}}
\newcommand{\mK}{\mathcal{K}}

\newcommand{\mM}{\mathcal{M}}
\newcommand{\mN}{\mathcal{N}}

\newcommand{\mS}{\mathcal{S}}

\newcommand{\fa}{\mathfrak{a}}

\newcommand{\ff}{\mathfrak{f}}

\newcommand{\fl}{\mathfrak{l}}

\newcommand{\fm}{\mathfrak{m}}

\newcommand{\fp}{\mathfrak{p}}

\newcommand{\fq}{\mathfrak{q}}

\newcommand{\bfC}{\mathbf{C}}

\newcommand{\bfF}{\mathbf{F}}
\newcommand{\bfG}{\mathbf{G}}
\newcommand{\bfH}{\mathbf{H}}

\newcommand{\bfP}{\mathbf{P}}
\newcommand{\bfQ}{\mathbf{Q}}
\newcommand{\bfR}{\mathbf{R}}

\newcommand{\bfT}{\mathbf{T}}

\newcommand{\bfZ}{\mathbf{Z}}

\newcommand{\bfi}{\mathbf{i}}

\newcommand{\Oo}{\mathcal{O}}

\newcommand{\OK}{\mathcal{O}_K}

\newcommand{\AK}{\mathbf{A}_K}

\newcommand{\AQ}{\mathbf{A}}

\newcommand{\AKf}{\mathbf{A}_{K,\textup{f}}}

\newcommand{\AQf}{\mathbf{A}_{\textup{f}}}

\newcommand{\OKv}{\mathcal{O}_{K,v}}

\newcommand{\tuf}{\textup{f}}
\newcommand{\tuh}{\textup{h}}

\newcommand{\tuM}{\textup{M}}

\newcommand{\tualg}{\textup{alg}}

\newcommand{\tuint}{\textup{int}}

\newcommand{\tuNM}{\textup{NM}}

\newcommand{\turk}{\textup{rk}}

\newcommand{\tush}{\textup{sh}}

\newcommand{\tuss}{\textup{ss}}
\newcommand{\tust}{\textup{st}}

\newcommand{\ov}{\overline}

\newcommand{\be}{\begin{equation}}
\newcommand{\ee}{\end{equation}}
\newcommand{\bes}{\begin{equation*}}
\newcommand{\ees}{\end{equation*}}

\newcommand{\bs}{\begin{split}}
\newcommand{\es}{\end{split}}
\newcommand{\bss}{\begin{split*}}
\newcommand{\ess}{\end{split*}}

\newcommand{\bmat}{\left[ \begin{matrix}}
\newcommand{\emat}{\end{matrix} \right]}
\newcommand{\bsmat}{\left[ \begin{smallmatrix}}
\newcommand{\esmat}{\end{smallmatrix} \right]}

\newcommand{\bml}{\begin{multline}}
\newcommand{\eml}{\end{multline}}
\newcommand{\bmls}{\begin{multline*}}
\newcommand{\emls}{\end{multline*}}

\DeclareMathOperator{\ad}{ad}
\DeclareMathOperator{\Ann}{Ann}

\DeclareMathOperator{\BC}{BC}

\DeclareMathOperator{\Cl}{Cl}
\DeclareMathOperator{\coker}{coker}
\DeclareMathOperator{\cond}{cond}
\DeclareMathOperator{\Desc}{Desc}
\DeclareMathOperator{\diag}{diag}

\DeclareMathOperator{\End}{End}

\DeclareMathOperator{\Frob}{Frob}
\DeclareMathOperator{\Gal}{Gal}
\DeclareMathOperator{\GL}{GL}

\DeclareMathOperator{\GU}{GU}
\DeclareMathOperator{\Hom}{Hom}

\DeclareMathOperator{\length}{length}

\DeclareMathOperator{\Res}{Res}

\DeclareMathOperator{\SL}{SL}
\DeclareMathOperator{\Sp}{Sp}

\DeclareMathOperator{\SU}{SU}

\DeclareMathOperator{\Symm}{Symm^2}

\DeclareMathOperator{\U}{U}
\DeclareMathOperator{\val}{val}
\DeclareMathOperator{\vol}{vol}

\def\f{\phi}
\def\p{\psi}

\def\quf{\mathbf{Q}}

\def\AQf{\mathbf{A}_{\textup{f}}}

\def\adele{\mathbf{A}}

\def\hh{\textup{h}}

\newcommand{\no}{\noindent}
\newcommand{\R}{\textup{Res}_{K/\mathbf{Q}}\hspace{2pt}}
\newcommand{\sep}{\hspace{3pt} | \hspace{3pt}}

\newcommand{\Ga}{\mathbf{G}_a}
\newcommand{\hs}{\hspace{2pt}}
\newcommand{\hf}{\hspace{5pt}}

\newcommand{\iy}{\infty}

\newcommand{\I}{\mathbf{i}}

\newcommand{\tr}{\textup{tr}\hspace{2pt}}
\newcommand{\Ur}{\textup{Im}\hspace{2pt}}
\newcommand{\Rz}{\textup{Re}\hspace{2pt}}

\newcommand{\T}{\mathbf{T}}
\newcommand{\invlim}{\mathop{\varprojlim}\limits}

%% file: settings.tex
\usepackage[all]{xy}
\usepackage{amsthm}
\usepackage{amssymb, amsfonts}
\SelectTips{cm}{10}\UseTips
\theoremstyle{plain}
\newtheorem{thm}{Theorem}
\newtheorem{prop}[thm]{Proposition}
\newtheorem{cor}[thm]{Corollary}
\newtheorem{lemma}[thm]{Lemma}
\newtheorem{conj}[thm]{Conjecture}

\newtheorem{fact}[thm]{Fact}

\theoremstyle{definition}
\newtheorem{definition}[thm]{Definition}
\newtheorem{notation}[thm]{Notation}

\newtheorem{rem}[thm]{Remark}

\numberwithin{thm}{section}
\numberwithin{equation}{section}